\documentclass[11pt,fullpage]{article}
\usepackage{amsmath}
\usepackage{amsfonts,amssymb}
\usepackage{color}	
\usepackage{amsthm}
\usepackage{graphics,graphicx,pspicture,graphpap}
\usepackage{hyperref}
\usepackage[right = 2.5cm, left=2.5cm, top = 2.5cm, bottom =2.5cm]{geometry}
\numberwithin{equation}{section}

\newcommand{\Real}{\mathbb R}
\newcommand{\T}{\mathbb T}
\newcommand{\Integers}{\mathbb Z}
\newcommand{\Integer}{\mathbb Z}
\newcommand{\norm}[1]{\left\lVert #1 \right\rVert}
\newcommand{\abs}[1]{\left\vert#1\right\vert}
\newcommand{\set}[1]{\left\{#1\right\}}

\newcommand{\grad}{\nabla}

\newcommand{\G}{\mathcal{G}}
\newcommand{\R}{\mathcal{R}}

\newcommand{\I}{\mathbf{I}}
\newcommand{\Naturals}{\mathbb N}
\newcommand{\jap}[1]{\left\langle #1 \right\rangle} 
\newcommand{\Torus}{\mathbb T}

\newcommand{\RE}{\textbf{Re}}

\newtheorem{theorem}{Theorem}
\theoremstyle{definition}
\newtheorem{remark}{Remark}

\theoremstyle{lemma}

\newtheorem{proposition}{Proposition}[section]

\theoremstyle{definition}

\theoremstyle{lemma}
\newtheorem{lemma}{Lemma}[section]

\begin{document}

\title{Enhanced dissipation and inviscid damping in the inviscid limit of the Navier-Stokes equations near the 2D Couette flow} 
\author{Jacob Bedrossian\footnote{\textit{jacob@cscamm.umd.edu}, University of Maryland, College Park} \, and Nader Masmoudi\footnote{\textit{masmoudi@cims.nyu.edu}, Courant Institute of Mathematical Sciences} \, and Vlad Vicol\footnote{\textit{vvicol@math.princeton.edu}, Princeton University}}

\date{\today}
\maketitle

\section*{Abstract} 
In this work we study the long time, inviscid limit of the 2D Navier-Stokes equations near the periodic Couette flow, and in particular, we confirm at the nonlinear level the qualitative behavior predicted by Kelvin's 1887 linear analysis.
At high Reynolds number $\textbf{Re}$, we prove that the solution behaves qualitatively like 2D Euler for times $t \lesssim \textbf{Re}^{1/3}$, and in particular exhibits inviscid damping (e.g. the vorticity weakly approaches a shear flow). 
For times $t \gtrsim \textbf{Re}^{1/3}$, {which is sooner than the natural dissipative time scale $O(\textbf{Re})$}, the viscosity becomes dominant and the streamwise dependence 
of the vorticity is rapidly eliminated by an enhanced dissipation effect.
 Afterward, the remaining shear flow decays on very long time scales $t \gtrsim \textbf{Re}$ back to the Couette flow. 
When properly defined, the dissipative length-scale in this setting is $\ell_D \sim \textbf{Re}^{-1/3}$, larger than the scale $\ell_D \sim \textbf{Re}^{-1/2}$ predicted in classical Batchelor-Kraichnan 2D turbulence theory.
The class of initial data we study is the sum of a sufficiently smooth function and a small (with respect to $\RE^{-1}$) $L^2$ function. 

\bigskip

\setcounter{tocdepth}{1}
{\small\tableofcontents}

\section{Introduction} 

Consider the 2D incompressible Navier-Stokes equations in 
the vicinity of the Couette flow:
\begin{equation} \label{def:NSE}
\left\{
\begin{array}{l}
  \omega_t + y\partial_x\omega + U \cdot \grad \omega = \nu\Delta \omega, \\ 
  U  = \grad^{\perp}\Delta^{-1} \omega,  \\\omega(t=0) =\omega_{in}. 
\end{array}
\right. 
\end{equation}
Here, $t\in \Real_+$, $(x,y) \in \mathbb \Torus \times \Real$, $\grad^\perp = (-\partial_y,\partial_x)$,  $U= (U^x,U^y)$ is the velocity, and $\omega$ is the vorticity of the {\em perturbation} to the Couette flow. 
The physical velocity is $(y,0) + U$ and the total vorticity of the flow is $-1 + \omega$. 
We denote the streamfunction of the perturbation by $\psi = \Delta^{-1}\omega$. 
We have already non-dimensionalized \eqref{def:NSE} by normalizing the Couette flow to $(y,0)$ and the side-length of the torus to $2\pi$, which together with the kinematic viscosity of the fluid, sets the Reynolds number $\RE$. We use the notation $\nu = \textbf{Re}^{-1}$. 
If $\nu = 0$, equivalently $\textbf{Re} = \infty$, the fluid is called \emph{inviscid} and \eqref{def:NSE} corresponds to the incompressible 2D Euler equations.
The velocity satisfies the momentum equation: 
\begin{equation} \label{def:NSEMomentum}
\left\{
\begin{array}{l}
  U_t + y\partial_x U + (U^y,0) + U \cdot \grad U = -\grad P + \nu \Delta U, \\ 
  \grad \cdot U = 0, 
\end{array}
\right. 
\end{equation}
where $P$ denotes the pressure. As we are interested in the behavior of solutions to \eqref{def:NSE} for high Reynolds number, without loss of generality we assume throughout the paper that $\nu \in (0,1]$.

The study of \eqref{def:NSE} for small perturbations is an old problem in hydrodynamic stability, 
considered by both Rayleigh \cite{Rayleigh80} and Kelvin \cite{Kelvin87}, as well as by many modern authors with new perspectives (see e.g. the classical texts \cite{DrazinReid81,Yaglom12} and the references therein). 
Rayleigh and Kelvin both studied the linearization of \eqref{def:NSE}, which is simply  
\begin{align} \label{def:2DNSE_Linear} 
\left\{
\begin{array}{l}
\partial_t \omega + y\partial_x \omega = \nu \Delta \omega, \\ 
\Delta \psi = \omega, \\ 
\omega(0) = \omega_{in}. 
\end{array} 
\right. 
\end{align}
Rayleigh first proved spectral stability (the absence of unstable eigenmodes) and Kelvin went further to explicitly solve \eqref{def:2DNSE_Linear}. 
Indeed, if we denote by $\hat{\omega}(t,k,\eta)$ the Fourier transform of $\omega(t,x,y)$ (note $k \in \Integer$, $\eta \in \Real$), then it is straightforward to verify Kelvin's solution
\begin{subequations} \label{eq:KelvinSoln}
\begin{align} 
\hat{\omega}(t,k,\eta) & = \hat{\omega}_{in}(k,\eta+kt)\exp\left[-\nu\int_0^t\abs{k}^2 + \abs{\eta - k\tau + kt}^2 d\tau \right], \\
\hat{\psi}(t,k,\eta) & = -\frac{\hat{\omega}_{in}(k,\eta+kt)}{k^2 + \eta^2}\exp\left[-\nu\int_0^t\abs{k}^2 + \abs{\eta - k\tau + kt}^2 d\tau \right]. \label{eq:psiKelvin}
\end{align}  
\end{subequations}
From \eqref{eq:KelvinSoln} we see directly that if $k\neq 0$ (that is, the mode depends on $x$), 
the mode is strongly damped by the viscosity by the time $t \gtrsim \nu^{-1/3}$, 
whereas for $k = 0$ the evolution is that of the 1D heat equation with a slower $t \gtrsim \nu^{-1}$ dissipative time-scale; hence the $k \neq 0$ modes experience \emph{enhanced dissipation}.   
For shorter times, the behavior of \eqref{eq:KelvinSoln} essentially matches that of the inviscid problem, studied in detail by Orr in \cite{Orr07}. 
One of Orr's key observations was the \emph{inviscid damping} predicted by \eqref{eq:psiKelvin}, that is, even for $\nu = 0$,   \eqref{eq:psiKelvin} shows that if the vorticity is sufficiently regular, then the velocity field $U = \grad^\perp \psi$ will return to a shear flow at an algebraic rate.
Both inviscid damping and enhanced dissipation are manifestations of the same \emph{vorticity mixing} induced by the background shear flow, which gives rise to a linear-in-time (for each $k \neq 0$) transfer of enstrophy to high frequencies (as seen in \eqref{eq:KelvinSoln}).
Together, the effects can be summarized by the following. 
\begin{proposition}[Linearized behavior (Kelvin \cite{Kelvin87} and Orr \cite{Orr07}) ] \label{prop:LinBehave}
Let $\omega(t)$ be a solution to \eqref{def:2DNSE_Linear} with initial data $\omega_{in} \in H^\sigma$ for any $\sigma \geq 3$ such that $\int \omega_{in} dx dy = 0$ and $\int \abs{y \omega_{in}} dy dx < \infty$. Then from \eqref{eq:KelvinSoln} we have for some $c > 0$ (writing also $(U^x,U^y) = \grad^\perp \psi$)\begin{subequations} 
\begin{align} 
\norm{P_{\neq 0}\omega(t,x+ty,y)}_{H^\sigma} & \lesssim \norm{\omega_{in}}_{H^\sigma} e^{-c \nu t^3 } \label{ineq:LinearMixing} \\ 
\norm{P_{0}\omega(t,x,y)}_{H^\sigma} & \lesssim \frac{\norm{\omega_{in}}_{H^\sigma} + \norm{\omega_{in}}_1}{\jap{\nu t}^{1/4}} \\ 
\norm{P_{\neq 0} U^x(t)}_2 + \jap{t} \norm{U^y(t)}_2 & \lesssim \frac{\norm{\omega_{in}}_{H^\sigma} e^{-c \nu t^3 }}{\jap{t}}. \label{ineq:UlinDamp}
\end{align}
\end{subequations}
where all implicit constants are independent of $\nu$ and $t$ and we are defining the orthogonal projections to $x$-independent modes 
$P_0 g = \frac{1}{2\pi} \int g(x,y) dx$ and $x$-dependent modes  $P_{\neq 0}g = g - \frac{1}{2\pi} \int g(x,y) dx$.
\end{proposition}
There are several details of Proposition \ref{prop:LinBehave} that warrant attention. 
First, is the requirement of three derivatives, which is necessary to deduce \eqref{ineq:UlinDamp}. 
Indeed, in order to get the inviscid algebraic decay from \eqref{eq:psiKelvin}, one requires localization of the Fourier transform of the initial data. Physically, this is to control the amount of information which is being \emph{un-mixed} by the Couette flow (see below and \cite{BM13} for further discussion). 
Secondly, notice that the regularity estimates in \eqref{ineq:LinearMixing} are only after pulling back by the Couette flow characteristics, another evidence of the mixing and a common theme in \cite{MouhotVillani11,BM13,BMM13}. Note further that this implies that in general, if one does not pull back by the shear flow we have $\norm{\omega(t)}_{H^\sigma} \approx \jap{t}^\sigma$ for $\nu t^3 \ll 1$. Thirdly, notice that the $x$-dependent contribution to the vorticity decays rapidly in \eqref{ineq:LinearMixing} after $\nu t^3 \gg 1$ whereas the $x$-independent contribution takes much longer to decay. Finally, the reason that $c < 1/3$ in the above theorem is due to the fact that the enhanced dissipation is anisotropic in time and frequency, as can be seen in the degeneracy of the exponentials in \eqref{eq:KelvinSoln}. Like the loss of regularity, this is also due to a transient un-mixing effect. 

In our work we study the long time, inviscid limit of `small' solutions to \eqref{def:NSE} and seek to find the appropriate nonlinear analogue of Proposition~\ref{prop:LinBehave}. 
The primary motivation of our work is to study the enhanced dissipation by mixing exhibited in \eqref{eq:KelvinSoln}. 
A second, deeply related, motivation for our work is to affirm the physical relevance of the inviscid ($\nu = 0$) work completed by two of the authors in \cite{BM13}, where inviscid damping is exhibited in the Euler equations.  

The enhanced dissipation effect is sometimes referred to by modern authors as the `shear-diffusion mechanism' 
and has been explicitly or implicitly pointed out by numerous authors studying shear flows \cite{RhinesYoung83,LatiniBernoff01,Chapman02,BeckWayne11}, vortex axisymmetrization \cite{BernoffLingevitch94,Gilbert88,Gilbert1993,Bajer2001} and 3D strained vortex filaments \cite{Lundgren82}.  
The work of \cite{ConstantinEtAl08} considers in a more general framework the connection between the spectral properties of the background flow and the enhanced dissipation effect in the context of passive tracers using the RAGE theorem (see also \cite{BerestyckiHamelNadirashvili05} for the elliptic case). 
It is pointed out in \cite{BeckWayne11} that the effect can be linked to a form of `hypocoercivity'~\cite{villani2009}.
Moreover, it is related to Taylor dispersion \cite{Taylor1953}, as discussed in e.g. \cite{RhinesYoung83}, 
which is not present in our work due to the periodic boundary conditions. 
Certain aspects of anomalous diffusions due to underlying drift have been studied in the context of SDEs; see for example \cite{IyerNovikov14,YoungPumirPomeau89,HaynesVanneste14,CardosoTabeling88}. 

At high Reynolds number, the linear solution \eqref{eq:KelvinSoln} exhibits a constant flux of enstrophy to high frequencies, independent of the Reynolds number, for times $t\lesssim \textbf{Re}^{1/3}$. 
In \cite{BM13} it is shown that this enstrophy flux persists in the Euler equations. Here, 
 we show that this enstrophy flux is stable in the inviscid limit, and moreover gives rise to the enhanced dissipation. 
It is suggested in \cite{Lundgren82,Gilbert88,Gilbert1993,GilbertBassom98}, that the inviscid limit of this enhanced dissipation effect is important for understanding the fine scale features of turbulent flows predicted in Kolmogorov~\cite{Kolmogorov1941} and Batchelor/Kraichnan~\cite{Batchelor1969,Kraichnan67} theories. 
However, notice that the constant enstrophy flux in \eqref{eq:KelvinSoln} is {\em more efficient} than a {\em generic turbulent flow} at moving enstrophy to small scales (where it is annihilated by the diffusion).  
Indeed, one can see this by considering the associated dissipative length scale $\ell_D$. 
Since the enstrophy transfer is linear in time, we see that in fact $\ell_D \sim \RE^{-1/3}$ which is a much larger scale than the $\textbf{Re}^{-1/2}$ predicted by the Kraichnan statistical theory. Partly, this is due to the fact that the enstrophy transfer considered here is highly anisotropic, and so should not necessarily be in the same scaling regime as homogeneous, isotropic turbulence.

It is well known that studying the inviscid limit $\nu \to 0$ and setting $\nu = 0$ in \eqref{def:NSE} is not the same problem.
Theorem \ref{thm:Main} below (and \cite{BM13}) implies that the 2D Euler equations qualitatively predict the correct behavior of \eqref{def:NSE} for times $t \lesssim \textbf{Re}^{1/3}$, at which point they cease to do so.  This time interval is much longer than what follows from brute force energy estimates, but still much shorter than for completely `laminar' flows, which would behave qualitatively like 2D Euler until $t \sim \textbf{Re}$. We emphasize that, {in general}, results which hold for the Euler evolution need not necessarily hold for the Navier-Stokes equations at high Reynolds number.
 For instance, in the presence of boundaries it was shown in~\cite{GrenierGuoNguyen14a}   that 
a linearly spectrally stable Euler flow becomes unstable upon the addition of dissipation at large Reynolds number (see also~\cite{DrazinReid81}). 
Also, the recent constructions of non-unique energy dissipative solutions of the 3D Euler equations~\cite{deLSzeHoldCts,DeLellisSzekelihidiBull,isett,buckDeLSzeOnsCrit} are not known to have an analogue in the presence of dissipation~\cite{BTW12}, with or without boundaries.

One of the consequences of our work is that \eqref{def:NSE} will exhibit inviscid damping for $t \lesssim \textbf{Re}^{1/3}$. 
See \cite{BM13} and the references therein for an in-depth discussion of the phenomenon. 
The term `inviscid damping' appeared after it was noticed to be the hydrodynamic analogue of \emph{Landau damping} \cite{Landau46,Ryutov99,CagliotiMaffei98,HwangVelazquez09,MouhotVillani11,BMM13} in the kinetic description of plasmas which are sufficiently collisionless (the analogue of inviscid) to be well-described by the Vlasov equations. See e.g. \cite{BouchetMorita10,SchecterEtAl00,Briggs70,BM95,MouhotVillani11,BMT13,BM13} and the references therein for some discussion about the connections between inviscid damping and Landau damping and  
the more general concept of phase mixing. 

Orr himself pointed out the main subtlety inherent with inviscid damping and Landau damping: in \eqref{eq:KelvinSoln}, vorticity can just as easily unmix to large scales and create growth in the velocity field as it can mix to small scales and create decay; together the two effects are known as the \emph{Orr mechanism}. 
Indeed, if one takes initial data concentrated near frequencies $\eta,k > 0$ with $\eta$ large relative to $k$, one can observe from \eqref{eq:KelvinSoln} that the stream-function reaches a maximum amplitude near the \emph{critical time} $t \sim \eta/k$. 
Even on the linear level, this implies that one needs to \emph{pay regularity} to deduce decay of the velocity, in order to control the amount of information that is being unmixed. 
    See \cite{Boyd83,Lindzen88,BM13} for more discussion of the Orr mechanism. For similar reasons, the enhanced dissipation slows down near the critical times as the enstrophy passes through the $O(1)$ length scales.
 Hence, it is a serious over-simplification to imagine the enhanced dissipation to be equivalent to replacing the $\nu$ in front of the $\partial_{xx}$ on the RHS of \eqref{def:NSE} with $\nu^{1/3}$; it is acting differently for each frequency at each time.
Our methods will deal with this effect by designing a semi-norm to measure the solution which respects the frequency-by-frequency anisotropy and allows to gradually pay regularity for enhanced dissipation of the vorticity; see \S\ref{sec:Proof}. 

The Orr mechanism is known to interact poorly with the nonlinear term, creating a weakly nonlinear effect referred to as an \emph{echo}. 
 In nearly collisionless plasmas, echoes were captured experimentally in \cite{MalmbergWharton68} and in 2D Euler much later in \cite{YuDriscoll02,YuDriscollONeil} (see also \cite{VMW98,Vanneste02}). 
The echoes can potentially chain into a cascade and greatly amplify the regularity loss already present in the linear theory (see e.g. \cite{MouhotVillani11,BM13}), which is the reason that the nonlinear results in \cite{CagliotiMaffei98,HwangVelazquez09,MouhotVillani11,BM13,BMM13} all require at least Gevrey class regularity \cite{Gevrey18} (the work of \cite{FaouRousset14} is the exception, as the model studied therein does not support infinite echo cascades).
For \eqref{def:NSE}, we will not need the data to be Gevrey class in a qualitative sense (which would be somewhat non-physical for a model involving a dissipative effect), however, the high regularity requirement will come as a quantitative requirement that the data be $L^2$ close to Gevrey class (see Theorem \ref{thm:Main} below). 
Specifically, Gevrey-$\frac{1}{s}$ with {$s \in(1/2,1)$} data with a distance in $L^2$ that diminishes as a function of Reynolds number, eventually collapsing to a neighborhood in Gevrey-$\frac{1}{s}$ in the inviscid limit (consistent with \cite{BM13}). 
In view of the regularity requirement needed to obtain inviscid damping in 2D Euler,  
the type of initial data we consider is natural as the $L^2$ piece must be small enough to have negligible influence above the dissipative length scale (although it is not clear whether or not the diameter of the $L^2$ neighborhood in \eqref{ineq:IDassump} is optimal). 
The Gevrey-$\frac{1}{s}$ norm is given by
\[
\norm{f}_{\G^{\lambda;s}}^2 = \sum_{k}\int \abs{\hat{f}_k(\eta)}^2 e^{2\lambda\abs{k,\eta}^s} d\eta
\]
where $\hat{f}_k(\eta)$ is the Fourier transform of $f$ at frequency $(k,\eta)$.

\subsection{Statement of Theorem} \label{sec:Statement}
We first state the main result, and then follow the statement by remarks, a discussion and a brief summary of the proof. 
\begin{theorem} \label{thm:Main}
For all $s \in (1/2,1)$, $\lambda > \lambda^\prime > 0$, $\delta > 0$ and all integers $\alpha \geq 1$ there exists $\epsilon_0 = \epsilon_0(\alpha,s,\lambda,\lambda^\prime,\delta)$ and $K_0 = K_0(s,\lambda,\lambda^\prime,\delta)$ such that if $\omega_{in}^\nu$ is mean-zero and satisfies $\omega_{in}^{\nu} = \omega^\nu_{S,in} + \omega^\nu_{R,in}$ with 
\begin{align}
\norm{|\nabla|^{-1} \omega^\nu_{in}}_{2}   + \norm{\omega^\nu_{S,in}}_{\G^{\lambda;s}} + e^{K_0\nu^{-\frac{3(1+\delta)s}{2(1-s)}}}  \norm{\omega^\nu_{R,in}}_{2} = \epsilon \leq \epsilon_0, \label{ineq:IDassump}
\end{align} 
then for all $\nu$ sufficiently small (independent of $\epsilon$ and $\epsilon_0$) the solution $\omega^{\nu}(t)$ with initial data $\omega_{in}^\nu$ satisfies the following properties (with all implicit constants independent of $\nu$, $t$ and $\epsilon$), 
\begin{itemize} 
\item[(i)] the uniform (in $t$ and $\nu$) estimate 
\begin{equation} \label{main-omega} 
\norm{\omega^\nu(t,x + ty + \Phi(t,y),y)}_{\G^{\lambda^\prime;s}} \lesssim \epsilon
\end{equation} 
where $\Phi(t,y)$ is given explicitly by 
\begin{align} 
\Phi(t,y) = \int_0^t e^{\nu (t-\tau)\partial_{yy}} \left(\frac{1}{2\pi}\int_\T U^x(\tau,x,y) dx\right) d\tau;  \label{def:phi}
\end{align}
\item[(ii)] inviscid damping on inviscid time-scales: for $t \lesssim \nu^{-1/3}$,   
\begin{subequations} \label{main-omega-inv} 
\begin{align} 
\norm{\frac{d}{dt}\left[\omega^\nu(t,x + ty + \Phi(t,y),y)\right]}_{\G^{\lambda^\prime;s}} & \lesssim \frac{\epsilon^2}{\jap{t}^2} + \epsilon \nu \jap{t}^2, \label{ineq:scattering} \\
\norm{P_{\neq 0}U^x(t)}_2 + \jap{t}\norm{U^y(t)}_{2} & \lesssim \frac{\epsilon}{\jap{t}}, \label{ineq:ID1} \\ 
\norm{\frac{d}{dt}P_{0}U^x(t)}_{\G^{\lambda^\prime;s}} & \lesssim \frac{\epsilon^2}{\jap{t}^3} + \epsilon \nu \label{ineq:P0Ux};
\end{align} 
\end{subequations} 
\item[(iii)] enhanced dissipation for $x$-dependent modes on the fast viscous time-scale: for $t \gtrsim \nu^{-1/3}$, 
\begin{subequations} 
\begin{align} 
\norm{P_{\neq 0}\omega^\nu(t,x + ty + \Phi(t,y),y)}_{\G^{\lambda^\prime;s}} & \lesssim \frac{\epsilon}{\jap{\nu t^3}^\alpha}, \label{ineq:thmED} \\ 
\norm{P_{\neq 0}U^x(t)}_2 + \jap{t}\norm{U^y(t)}_{2} & \lesssim \frac{\epsilon}{\jap{t}\jap{\nu t^3}^\alpha}; \label{ineq:ID2}.
\end{align} 
\end{subequations}
\item[(iv)] and standard viscous decay for $x$-independent modes on the slow viscous time-scale: for $t \gtrsim \nu^{-1}$, 
\begin{align}    
\norm{P_{0}\omega^\nu(t)}_2 \lesssim \frac{\epsilon}{\jap{\nu t}^{1/4}}.  \label{ineq:thmSD}
\end{align} 
\end{itemize}
\end{theorem}
\begin{remark} 
It is useful to compare Theorem \ref{thm:Main} with the linear behavior in Proposition \ref{prop:LinBehave}. 
The qualitative behavior predicted by the nonlinear theorem is essentially the same except for two details: the algebraic rate of the enhanced dissipation and the correction to the shear flow characteristics $\Phi(t,y)$. The correction $\Phi$ is to account for the fact that the perturbation induces a small shear, $<U^x(t)>$, which varies in time and decays at a non-integrable rate on the very long $O(\textbf{Re})$ time scale (in particular, in the $\nu = 0$ case the flow never returns to Couette flow). Therefore, gradients would grow without bound in \eqref{main-omega} if we did not account for this. 
The algebraic rate of decay arises since, due to un-mixing, we pay regularity in return for 
enhanced dissipation. The current proof pays $3\alpha$ derivatives; with some additional technical effort it seems we could trade Gevrey-2 regularity (or Gevrey-$\frac{1}{\beta}$ for any $\beta < s$) and likely get a decay like $\exp\left[-(\nu t^3)^{1/6} \right]$. However, getting the decay seen on the linear level seems very difficult with the current method (we are unsure about whether or not we can expect the same linear decay rate from the nonlinear solution). 
\end{remark}
   
\begin{remark} 
Inequality \eqref{ineq:thmSD} is also true with the $L^2$ norm replaced by $\mathcal{G}^{\lambda^\prime;s}$ with a very similar proof. 
\end{remark}
\begin{remark} 
For $\nu t^3 \ll 1$, statement (ii) essentially reduces to the main results of \cite{BM13}. 
\end{remark} 
\begin{remark}
We note that if $\omega^\nu_{R,in}$ were in $H^\sigma$ with $\sigma > 5/2$, instead of merely $L^2$, then one could take $\delta = 0$ in the hypothesis \eqref{ineq:IDassump}, which is consistent with the instant regularization available from the heat equation.   
\end{remark} 
\begin{center}
\begin{figure}[hbpt] 
\begin{picture}(150,200)(-100,0)
\scalebox{.5}{\includegraphics{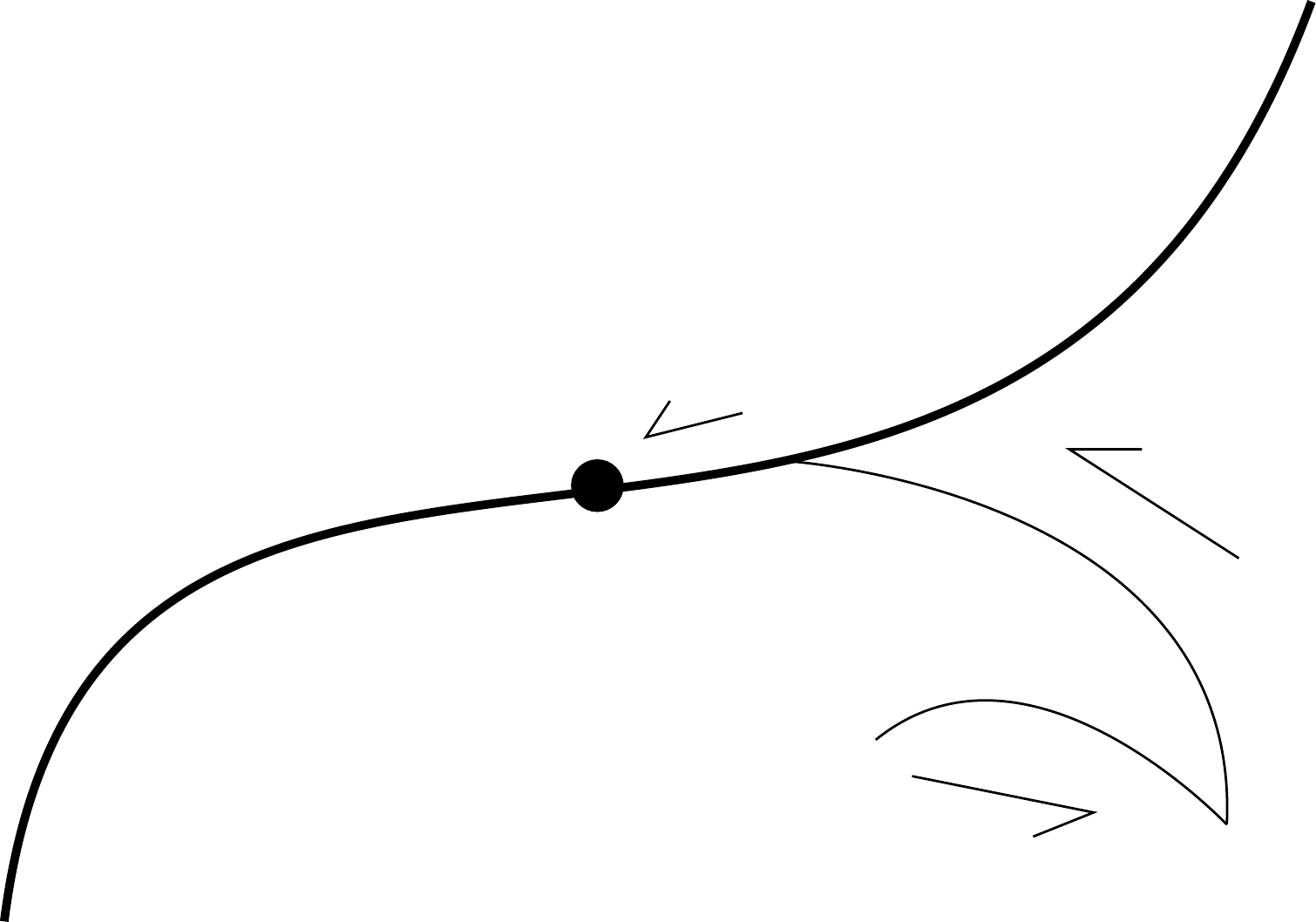}}
\put(-100,25){$t \ll 1$}
\put(-20,5){$t \approx \RE^{1/3}$}
\put(-150,80){$t \gg \RE$}
\put(-140,55){$Couette$}
\put(-100,90){$t \gg \RE^{1/3}$}
\put(0,140){$Shear flows$}
\end{picture}
\caption{This is a schematic of the qualitative behavior of $\omega^\nu(t)$ in a strong topology (e.g. $H^\sigma$ for $\sigma > 0$) as time evolves. 
The solution rapidly regularizes due to the viscosity, together with the assumption \eqref{ineq:IDassump}, and near time zero lies in a neighborhood of the Couette flow. For $t \ll \RE^{1/3}$, inviscid dynamics dominate and there is a transient growth of strong norms as the vorticity mixes (e.g. the $H^\sigma$ norm will generally grow as $\jap{t}^{\sigma}$). By times $t \gg \RE^{1/3}$, the dissipation dominates and the solution rapidly converges to a shear flow (but not necessarily the Couette flow). 
Finally, for times $t \gg \RE$, the remaining shear flow relaxes to the Couette flow. } \label{fig:Strong}
\end{figure}
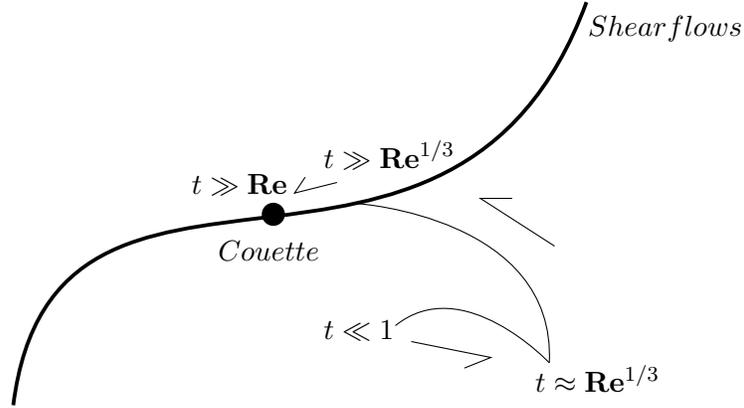
\end{center}
\begin{center}
\begin{figure}[hbpt] 
\begin{picture}(150,200)(-100,0)
\scalebox{.5}{\includegraphics{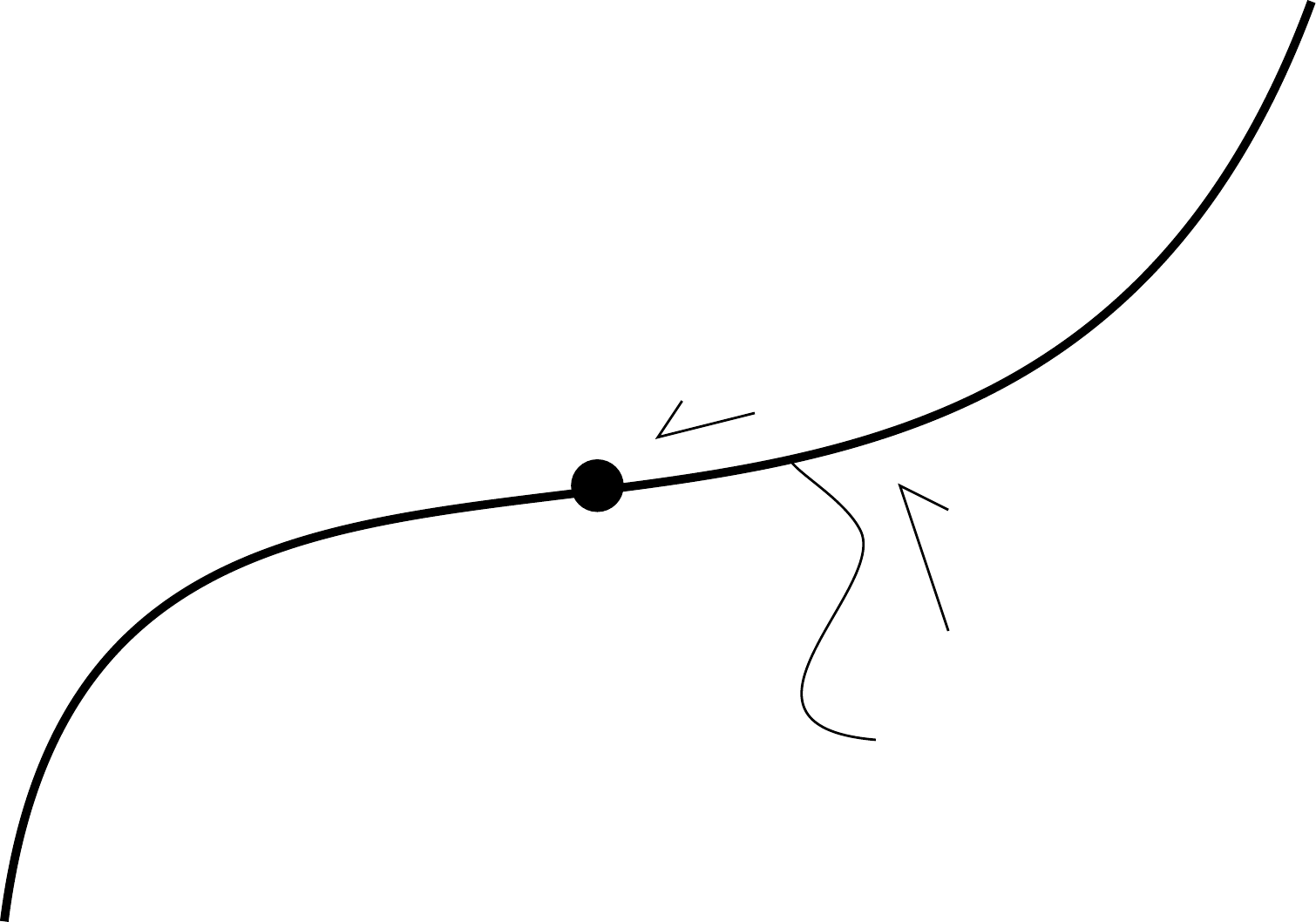}}
\put(-85,15){$t \ll 1$}
\put(-150,80){$t \gg \RE$}
\put(-140,55){$Couette$}
\put(-100,90){$t \gg 1$}
\put(0,140){$Shear flows$}
\end{picture}
\caption{This is a schematic of the qualitative behavior of $\omega^\nu(t)$ in the weak topology as time evolves. 
The solution begins in a neighborhood of the Couette flow and rapidly approaches a shear flow in the weak topology via 
inviscid damping on time scales independent of $\RE$. 
There is relatively little change until times $t \gg \RE$ when finally viscous dissipation 
relaxes the shear flow back to the Couette flow.} \label{fig:Weak} 
\end{figure}
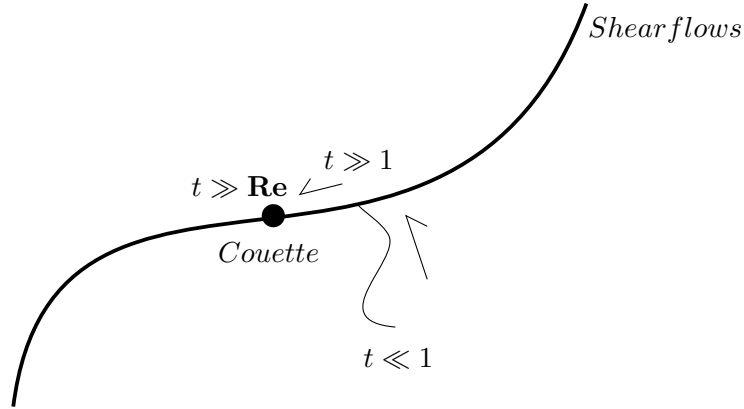
\end{center}

Schematics of the qualitative behavior predicted by Theorem \ref{thm:Main} are given in Figures \ref{fig:Strong} and \ref{fig:Weak}.
As Theorem \ref{thm:Main} predicts essentially the same behavior as the main Theorem in \cite{BM13} before the enhanced dissipation is activated, 
it is natural that the proof of Theorem \ref{thm:Main} also contains the entire proof of the main theorem in \cite{BM13}. 
To avoid being redundant, we simply quote without proof estimates that are proved in \cite{BM13} 
 and \emph{concentrate only on the steps that are specific to Navier-Stokes.} 
The proof of Theorem \ref{thm:Main} is outlined in \S\ref{sec:Proof}, however let us quickly highlight the main ideas and the main challenges. 
\begin{itemize} 
\item The coordinate system used in \cite{BM13} needs to be adapted to account for the viscous dissipation in the background shear flow (see \eqref{def:zv} below). 
In the Lagrangian view, we are accounting for the stochastic wandering of particle trajectories in the $y$ variable. In the Eulerian view, we are accounting for the momentum transport induced by viscosity. Note that this is physically relevant even for short times, as the viscous dissipation always dominates for high enough frequencies.  
Like the choice of coordinates in \cite{BM13}, the choice used here is determined essentially uniquely by the analytic properties we require. 

\item As can be seen from \eqref{eq:KelvinSoln}, the enhanced dissipation effect is highly anisotropic in time and frequency: it is only active for $x$-dependent frequencies and slows down near the critical time $t \sim \eta/k$. To correctly capture the enhanced dissipation we use a special semi-norm that is well adapted to the anisotropy (defined in \eqref{def:Anu} and analyzed in \S\ref{sec:Anu}). The anisotropy also creates a huge imbalance between $x$-dependent modes and $x$-independent modes, and an important part of the analysis in \S\ref{sec:ED} is to prove that very little enstrophy from the $x$-independent modes can be converted to $x$-dependent enstrophy at long times. 

\item The evolving background shear flow determines the rate at which dissipation occurs locally in $y$, a fully nonlinear effect as the shear flow is determined by the solution. Controlling this effect will be the main challenge to proving Theorem \ref{thm:Main} and requires a careful analysis of the interplay between the regularity and decay of the background shear flow with the regularity and decay of the $x$-dependent modes (indeed, the structure of the proof is essentially dictated by this balance; see \S\ref{sec:Proof}). 
\end{itemize}

\subsection{Notation and conventions} \label{sec:Notation}
See \S\ref{Apx:LPProduct} for the Fourier analysis conventions we are taking.
A convention we generally use is to denote the discrete $x$ (or $z$) frequencies as subscripts.   
By convention we always use Greek letters such as $\eta$ and $\xi$ to denote frequencies in the $y$ or $v$ direction and lowercase Latin characters commonly used as indices such as $k$ and $l$ to denote frequencies in the $x$ or $z$ direction (which are discrete).
Another convention we use is to denote dyadic integers by $M,N \in \mathbb{D}$ or $2^{\Integer}$ where  
\begin{align*} 
2^\Integer & = \set{...,2^{-j},...,\frac{1}{4},\frac{1}{2},1,2,...,2^j,...}, \\ 
\mathbb{D} & = \set{\frac{1}{2},1,2,...,2^j,...} = 2^\mathbb{D}  \cup {\frac12} . 
\end{align*}
This will be useful when defining Littlewood-Paley projections and paraproduct decompositions, see \S\ref{Apx:LPProduct}. 
Given a function $m \in L^\infty$, we define the Fourier multiplier $m(\grad) f$ by 
\begin{align*} 
(\widehat{m(\grad)f})_k(\eta) =  m( (ik,i\eta) ) \hat{f}_k(\eta). 
\end{align*}   

We use the notation $f \lesssim g$ when there exists a constant $C > 0$ independent of the parameters of interest 
such that $f \leq Cg$ (we analogously define $f \gtrsim g$). 
Similarly, we use the notation $f \approx g$ when there exists $C > 0$ such that $C^{-1}g \leq f \leq Cg$. 
We sometimes use the notation $f \lesssim_{\alpha} g$ if we want to emphasize that the implicit constant depends on some parameter $\alpha$.
We will denote the $l^1$ vector norm $\abs{k,\eta} = \abs{k} + \abs{\eta}$, which by convention is the norm taken in our work. 
Similarly, given a scalar or vector in $\Real^n$ we denote
\begin{align*} 
\jap{v} = \left( 1 + \abs{v}^2 \right)^{1/2}. 
\end{align*} 
We use a similar notation to denote the $x$ (or $z$) average of a function: $<f> = \frac{1}{2\pi}\int f(x,y) dx = f_0$. 
We also frequently use the notation $P_{\neq 0}f = f - f_0$. 
We denote the standard $L^p$ norms by $\norm{f}_{p}$. 
We make common use of the Gevrey-$\frac{1}{s}$ norm with Sobolev correction defined by 
\begin{align*} 
\norm{f}_{\G^{\lambda,\sigma;s}}^2 = \sum_{k}\int \abs{\hat{f}_k(\eta)}^2 e^{2\lambda\abs{k,\eta}^s}\jap{k,\eta}^{2\sigma} d\eta. 
\end{align*} 
Since most of the paper we are taking $s$ as a fixed constant, it is normally omitted. Also, if 
$\sigma =0$, it is omitted. 
We refer to this norm as the $\mathcal{G}^{\lambda,\sigma;s}$ norm and occasionally refer to the space of functions
\begin{align*} 
\mathcal{G}^{\lambda,\sigma;s} = \set{f \in L^2 :\norm{f}_{\G^{\lambda,\sigma;s}}<\infty}. 
\end{align*}
See \S\ref{apx:Gev} for a discussion of the basic properties of this norm and some related useful inequalities.

For $\eta \geq 0$, we define $E(\eta)\in \Integer$ to be the integer part.  
We define for $\eta \in \Real$ and $1 \leq \abs{k} \leq E(\sqrt{\abs{\eta}})$ with $\eta k \geq 0$, 
$t_{k,\eta} = \abs{\frac{\eta}{k}} - \frac{\abs{\eta}}{2\abs{k}(\abs{k}+1)} =  \frac{\abs{\eta}}{\abs{k}+1} +  \frac{\abs{\eta}}{2\abs{k}(\abs{k}+1)}$ and $t_{0,\eta} = 2 \abs{\eta}$ and 
the critical intervals  
\begin{align*} 
I_{k,\eta} = \left\{
\begin{array}{lr}
[t_{\abs{k},\eta},t_{\abs{k}-1,\eta}] & \textup{ if } \eta k \geq 0 \textup{ and } 1 \leq \abs{k} \leq E(\sqrt{\abs{\eta}}), \\ 
\emptyset & otherwise.
\end{array} 
\right. 
\end{align*} 
For minor technical reasons, we define a slightly restricted subset as the \emph{resonant intervals}
\begin{align*} 
\mathbf I_{k,\eta} = \left\{
\begin{array}{lr} 
I_{k,\eta} & 2\sqrt{\abs{\eta}} \leq t_{k,\eta}, \\ 
\emptyset & otherwise.
\end{array} 
\right. 
\end{align*} 
Note this is the same as putting a slightly more stringent requirement on $k$: $k \leq \frac{1}{2}\sqrt{\abs{\eta}}$.

\section{Proof of Theorem \ref{thm:Main}} \label{sec:Proof}
\subsection{Instant regularization and \texorpdfstring{$O(1)$}{O(1)} times} 
A physical aspect of our work is to verify that the instant regularization due to the viscosity
replaces the qualitative regularity requirements in \cite{BM13} with 
the more physically natural quantitative regularity requirement \eqref{ineq:IDassump}. 
Indeed, we have the following result:
\begin{proposition}[Instant regularization] \label{prop:InstantReg}
For all $\epsilon > 0$, $\lambda > \lambda^\prime > 0$, $\delta>0$, and $s \in (1/2,1)$, there exists an $\epsilon^\prime = \epsilon'(s,\lambda,\lambda^\prime,\delta)$, a $K_0 = K_0(s,\lambda,\lambda^\prime,\delta)$, and a $\nu_0 = \nu_0(s,\lambda,\lambda^\prime,\delta)$ such that if $\omega_{in}^\nu$ is mean-zero and can be written as the sum of a smooth part and a rough part $\omega_{in}^{\nu} = \omega^\nu_{S,in} + \omega^\nu_{R,in}$ with 
\begin{align}
\norm{|\nabla|^{-1} \omega^\nu_{in}}_{2} + \norm{\omega^\nu_{S,in}}_{\G^{\lambda;s}}  + e^{K_0\nu^{-\frac{(3+\delta)s}{2(1-s)}}}   \norm{\omega^\nu_{R,in}}_{2}  \leq \epsilon^\prime,  
\label{ineq:Instant}
\end{align}
then for all $\nu \leq \nu_0$  the solution $\omega^\nu$ to 2D NSE near Couette flow \eqref{def:NSE}, with initial datum $\omega_{in}^\nu$ satisfies
\begin{align} 
\label{eq:analytic:to:do}
\norm{u^\nu(T_{\lambda,\lambda^\prime})}_2 + \norm{ \omega^\nu(T_{\lambda,\lambda^\prime}) }_{\G^{\frac{3\lambda}{4} + \frac{\lambda^\prime}{4};s}} \leq \epsilon
\end{align}
where 
\begin{align}
T_{\lambda,\lambda^\prime} = \min\left\{ \frac{\lambda - \lambda'}{C_0}, 1\right\} ,
\label{eq:T:lambda}
\end{align}
for some sufficiently large universal constant $C_0 >0 $.
\end{proposition}
\begin{proof} 
The proof consists in fours steps:
\begin{itemize}
\item[(a)] We propagate the low-frequency control of the solution, which is encoded in the $L^2$ energy estimate on $u^\nu$. 
\item[(b)] We solve the Navier-Stokes equations perturbed about the Couette flow \eqref{def:NSE} with initial datum $\omega^{\nu}_{S,in}$ and establish that the solution $\omega^{\nu}_{S}(t)$ lies in $\G^{\frac{4\lambda+\lambda'}{5} ;s} \cap \G^{\sqrt{\nu} t; 1}$ for all $t \leq T_2 \lesssim \nu^{1+\delta} $,   with a norm less than $2 \epsilon^\prime$.
\item[(c)] We solve the perturbation of \eqref{def:NSE} around the solution $\omega^{\nu}_S$, with initial datum $\omega^{\nu}_{R,in}$ and show that the corresponding solution lies in $\G^{\sqrt{\nu} t; 1}$ for all $t \leq T_2$, with a suitable control on the norm. In view of the smallness of the initial datum $\omega^{\nu}_{R,in}$ in $L^2$ it follows that the $\G^{\frac{4\lambda+\lambda'}{5};s}$ norm of $\omega^{\nu}_{R}$ is also less than $2 \epsilon^\prime$.
\item[(d)] We propagate the Gevrey-class regularity of the solution to \eqref{def:NSE} with initial datum $\omega^{\nu}_{S}(T_2) + \omega^{\nu}_{R}(T_2) \in \G^{\frac{4\lambda+\lambda'}{5};s}$ of size $\leq 4 \epsilon^\prime$ and show that the corresponding solution $\omega^\nu$ at time $T_{\lambda,\lambda^\prime}$ obeys the bound \eqref{eq:analytic:to:do}.
\end{itemize}
For the remainder of the proof, for ease of notation we drop the $\nu$ upper index for $\omega$ and $u$.
Without loss of generality we may assume the ``rough'' part of the initial datum $\omega_{R,in}$ is supported on frequencies away from the set $\{ k=0, |\eta|\leq 1\}$. 
Further, we define $u_S(t) = \grad^{\perp}\Delta^{-1}\omega_S(t)$ and $u_R(t) = \grad^{\perp}\Delta^{-1}\omega_R(t)$, 
the velocities associated with the smooth and rough solutions (both in $L^2$ initially due to the assumptions on the initial data and the frequency support of $\omega_{R,in}$).  
\bigskip

\noindent {\it Proof of (a).}
The standard $L^2$ energy estimate for $u_S$ obtained from \eqref{def:NSEMomentum} yields
\[
\frac 12 \frac{d}{dt} \norm{u_S}_2^2 +  \nu \norm{\omega_S}_2^2 \leq  \left| \int u^{x}_S u^{y}_S \right| \leq  \norm{P_{\neq 0} u_S}_2^2
\]  
and therefore  
\begin{align}
\norm{u_S(t)}_2 \leq 3 \norm{u_{S,in}}_2 \leq 3 ( \norm{|\nabla|^{-1} \omega_{in}}_2 + \norm{\omega_{S,in}}_2) \leq 6 \epsilon^\prime 
\label{eq:u:S:control}
\end{align} 
for all $t \leq 1$. 
Note that the same bound for the full solution $u$ holds on $[0,1]$. \\

\noindent {\it Proof of (b).}
Consider the analytic Fourier multiplier
\begin{align}
M_A(t,k,\eta) &= \jap{k,\eta}^\sigma e^{\sqrt{\nu}\; t\; |k,\eta| }
\end{align}
for some $\sigma > 5/2$.
As in~\cite{FoiasTemam89}, the solution $\omega_S$ of \eqref{def:NSE} with initial datum $\omega_{S,in}$ obeys the a priori estimate
\begin{align}
 &\frac{1}{2} \frac{d}{dt} \norm{M_A \omega_S}_2^2 + \nu \norm{ |\nabla| M_A \omega_S}_2^2 \notag\\
 &\qquad = \sqrt{\nu} \norm{ |\nabla|^{1/2} M_A \omega_S}_2^2 + \sum_k \int_\eta  M_A(k,\eta)^2 \partial_\eta \hat{\omega}_S(k,\eta) \, k \, \overline{\hat{\omega}_S(k,\eta)}  d\eta \notag\\
 &\qquad \ \ + \sum_{k,l} \int_\eta \! \int_\xi M_A(k,\eta) \overline{\hat{\omega}_S(k,\eta)} M_A(l,\xi)  \hat{\omega}_S(l,\xi) M_A(k-l,\eta-\xi) \hat{\omega}_S(k-l,\eta-\xi)   \notag\\
 &\qquad \qquad \qquad \qquad \times  \frac{(l,\xi)^\perp \cdot (k,\eta)}{ \abs{l,\xi}^2} \frac{M_A(k,\eta)}{M_A(l,\xi) M_A(k-l,\eta-\xi)} d\xi d\eta\notag\\
 &\qquad = \sqrt{\nu} \norm{ |\nabla|^{1/2} M_A \omega_S}_2^2 + T_A + N_A
 \label{eq:step:a:1}
\end{align}
where $T_A$ is the shear transport term and $N_A$ is the nonlinear term. Upon integrating by parts in $\eta$ and bounding $|\partial_\eta M_A(k,\eta)|$ we immediately obtain
\begin{align}
|T_A| &\lesssim \sqrt{\nu} \; t \norm{ |\nabla|^{1/2} M_A \omega^\nu_S}_2^2.
\label{eq:step:a:3}
\end{align}
In order to bound the nonlinear term $N_A$ we use the triangle inequality 
\[ 
|M_A(k,\eta)| \leq |M_A(l,\xi)| |M_A(k-l,\eta-\xi)|,
\] 
followed by \eqref{ineq:L2L2L1}, which implies
\begin{align}
|N_A| 
& \leq \sum_{k,l} \int_\eta \! \int_\xi \abs{k,\eta} M_A(k,\eta) |\hat{\omega}_S(k,\eta)| M_A(k-l,\eta-\xi) |\hat{\omega}_S(k-l,\eta-\xi)| \notag\\
&\qquad \qquad \times \abs{l,\xi}^{\frac{\delta}{(2+\delta)}} M_A(l,\xi)  |\hat{\omega}_S(l,\xi)|   \frac{{\bf 1}_{|(l,\xi)|\geq 1}}{\abs{l,\xi}^{1+\frac{\delta}{(2+\delta)}}}   d\xi d\eta \notag\\
&\quad + \left|  \sum_k \int_{\eta} \int_\xi  k\,  M_A(k,\eta)  \overline{\hat{\omega}_S(k,\eta)} M_A(k,\eta-\xi) \hat{\omega}_S(k,\eta-\xi)  M_A(0,\xi)  \hat{u}_S(0,\xi)  {\bf 1}_{|\xi|\leq 1}  d\xi d\eta \right| \notag\\
&\lesssim_\delta \norm{|\nabla| M_A \omega_S}_2  \norm{M_A \omega_S}_2 \norm{|\nabla|^{\frac{\delta}{2+\delta}} M_A \omega_S}_{2} 
+ \norm{|\nabla| M_A \omega^\nu_S}_2 \norm{M_A \omega^\nu_S}_2  \norm{\hat{u}_S(0,\xi) {\bf 1}_{l=0,|\xi|\leq 1}}_{2} \notag\\
&\lesssim \norm{|\nabla| M_A \omega_S}_2^{\frac{2+2\delta}{2+\delta}}   \norm{M_A \omega_S}_2^{\frac{4+\delta}{2+\delta}}  + \norm{|\nabla| M_A \omega_S}_2 \norm{M_A \omega_S}_2  \norm{ {u}_S}_{2},
\label{eq:step:a:4}
\end{align}
for any $\delta >0$. 
In the last inequality we have used the Gagliardo-Nirenberg interpolation for the first term. In view of the bound \eqref{eq:u:S:control}, estimate \eqref{eq:step:a:4} yields
\begin{align}
|N_A|
&\leq \frac{C}{\nu} \norm{M_A \omega_S}_2^2 \left(  \frac{1}{\nu^{\delta}} \norm{M_A \omega_S}_2^{2+\delta} +\norm{ {u}_S}_{2}^2 \right) + \frac{\nu}{2}  \norm{|\nabla| M_A \omega_S}_2^2\notag\\
&\leq \frac{C}{\nu^{1+\delta}} \norm{M_A \omega_S}_2^2 (1 +  \norm{M_A \omega_S}_2)^{2+\delta} + \frac{\nu}{2}  \norm{|\nabla| M_A \omega_S}_2^2 
\label{eq:step:a:5}
\end{align}
for some positive constant $C \geq 1$, that is independent of $\nu, \epsilon^\prime \leq 1$, but depends on the choice of $\delta>0$. Combining \eqref{eq:step:a:1}, \eqref{eq:step:a:3}, and \eqref{eq:step:a:5} we arrive at 
\begin{align*} 
 \frac{d}{dt} \norm{M_A \omega_S}_2^2 + \nu \norm{ |\nabla| M_A \omega_S}_2^2 & \leq 2\sqrt{\nu} (1+ t) \norm{ |\nabla|^{1/2} M_A\omega_S}_2^2 + \frac{C}{\nu^{1+\delta}}  \norm{M_A \omega_S}_2^2 (1+ \norm{M_A \omega_S}_2)^{2+\delta} \\ 
& \leq 2\sqrt{\nu} (1+ t)\norm{ |\nabla| M_A\omega_S}_2\norm{ M_A\omega_S}_2  + \frac{C}{\nu^{1+\delta}}  \norm{M_A \omega_S}_2^2 (1+ \norm{M_A \omega_S}_2)^{2+\delta} \\ 
& \leq \frac{\nu}{2}\norm{ |\nabla| M_A\omega_S}_2^2 + 4(1+ t)^2\norm{ M_A\omega_S}_2^2 + \frac{C}{\nu^{1+\delta}}  \norm{M_A \omega_S}_2^2 (1+ \norm{M_A \omega_S}_2)^{2+\delta}, 
\end{align*} 
for any $\delta>0$ and some constant $C\geq 1$.
Therefore, for $t \leq 1$ 
\begin{align}
 \frac{d}{dt} \norm{M_A \omega_S}_2^2  \leq\frac{C_1}{\nu^{1+\delta}} \norm{M_A \omega_S}_2^2 (1 + \norm{M_A \omega_S}_2^2)^{1+\delta/2}
 \label{eq:step:a:6}
\end{align}
for some constant $C_1 \geq 1$. 
Letting 
\begin{align}
T_1 = \min \left\{ \frac{\nu^{1+\delta}}{2 C_1}, T_{\lambda,\lambda^\prime}  \right\}   = \frac{\nu^{1+\delta}}{2 C_1} \qquad \mbox{(without loss of generality, as }\nu \leq \nu_0\mbox{)}, \label{eq:T1}
\end{align}
we obtain from \eqref{eq:step:a:6} and the assumption \eqref{ineq:Instant} on the $H^\sigma$ norm of $\omega_{S,in}$ that for $\epsilon' \leq 1$ 
\begin{align}
\norm{ M_A(t) \omega_S(t)}_2\leq 2 \epsilon^\prime \qquad \mbox{for all} \qquad t \leq T_1.
 \label{eq:step:a:7}
\end{align}

The Gevrey class estimates differ slightly from the real analytic bounds as we need to employ an additional commutator estimate.
Consider the Gevrey-$\frac{1}{s}$ Fourier-multiplier
\begin{align}
M_G(t,k,\eta) &= \jap{k,\eta}^\sigma e^{\lambda(t) |k,\eta|^{s}}
\end{align}
where the decreasing function $\lambda(t)$ shall be chosen precisely later, and $\sigma>5/2$ is a fixed Sobolev correction. Similarly to \eqref{eq:step:a:1} we have
\begin{align}
 &\frac{1}{2} \frac{d}{dt} \norm{M_G \omega_S}_2^2 + \nu \norm{ |\nabla| M_G \omega_S}_2^2 = \dot{\lambda}(t) \norm{ |\nabla|^{s/2} M_G \omega_S}_2^2 + T_G + N_G
 \label{eq:step:b:1}
\end{align}
where $T_G$ and $N_G$ are defined precisely as $T_A$ and $N_A$ upon replacing all $M_A$'s by $M_G$'s. It is not hard to see that 
\begin{align}
|T_G| &\lesssim \norm{M_G \omega_S}_2^2 + \lambda(t) \norm{ |\nabla|^{s/2} M_G \omega_S}_2^2
\label{eq:step:b:2}.
\end{align}
Bounding the nonlinear term $N_G$ requires a commutator estimate which arises since $u_S$ is divergence free, in the spirit of \cite{LevermoreOliver97,KukavicaVicol09}. More precisely, we have, 
\begin{align*}
N_G&=  \sum_{k,l} \int_\eta \! \int_\xi  M_G(k,\eta)  \overline{\hat{\omega}_S(k,\eta)}  \hat{\omega}_S(l,\xi)  \hat{\omega}_S(k-l,\eta-\xi)
\notag\\
 &\qquad \qquad \quad  \times \frac{(l,\xi)^\perp \cdot (k,\eta)}{ \abs{l,\xi}^2} \left( M_G(k,\eta)- M_G(k-l,\eta-\xi)\right)  d\xi d\eta.
\end{align*}
By the mean-value theorem applied to $M_G$, we obtain that  
\begin{align}
|N_G|
&\leq \lambda(t) \sum_{k,l} \int_\eta \! \int_\xi M_G(k,\eta) \,|\hat{\omega}_S(k,\eta)| \,M_G(k-l,\eta-\xi)  \, |\hat{\omega}_S(k-l,\eta-\xi)| \, |\hat{\omega}_S(l,\xi)|   \notag\\
 &\qquad \qquad \quad  \times  \left| \frac{(l,\xi)^\perp \cdot (k,\eta)}{\abs{l,\xi}^2}\right|  \frac{ \abs{l,\xi} e^{\lambda(t) |l,\xi|^s}}{|k,\eta|^{1-s} + |k-l,\eta-\xi|^{1-s}}  d\xi d\eta \notag \\
 &\quad +  \sum_{k,l} \int_\eta \! \int_\xi M_G(k,\eta) \,|\hat{\omega}_S(k,\eta)| M_G(k-l,\eta-\xi)   \, |\hat{\omega}_S(k-l,\eta-\xi)| M_G(l,\xi) \, |\hat{\omega}_S(l,\xi)|   \notag\\
 &\qquad \qquad \quad  \times  \left| \frac{(l,\xi)^\perp \cdot (k-l,\eta-\xi)}{ \abs{l,\xi}^2}\right|  
\left(\jap{k-l,\eta-\xi}^{\sigma-1} + \jap{k,\eta}^{\sigma-1}  \right) \frac{C \sigma |l,\xi|}{\jap{k-l,\eta-\xi}^\sigma \jap{l,\xi}^\sigma}  d\xi d\eta\notag\\
 &\leq \lambda(t) \sum_{k,l} \int_\eta \! \int_\xi M_G(k,\eta) \,|\hat{\omega}_S(k,\eta)| M_G(k-l,\eta-\xi) \, |\hat{\omega}_S(k-l,\eta-\xi)|   M_G(l,\xi) \, |\hat{\omega}_S(l,\xi)|  \notag\\
 &\qquad \qquad \quad  \times \frac{|k,\eta|^{s/2} |k-l,\eta-\xi|^{s/2}}{ \jap{l,\xi}^\sigma}  d\xi d\eta \notag \\
 &\quad + C \sum_{k,l} \int_\eta \! \int_\xi M_G(k,\eta)\,|\hat{\omega}_S(k,\eta)| M_G(k-l,\eta-\xi) \, |\hat{\omega}_S(k-l,\eta-\xi)|
M_G(l,\xi) \, |\hat{\omega}_S(l,\xi)|    \notag\\
 &\qquad \qquad \quad  \times  \left(\frac{1}{\jap{l,\xi}^{\sigma}} +  \frac{1}{\jap{l,\xi} \jap{k-l,\eta-\xi}^{\sigma-1}} \right)  d\xi d\eta\notag\\
 &\leq C \lambda(t) \norm{M_G \omega_S}_2 \norm{|\nabla|^{s/2} M_G \omega_S }_2^2 + C  \norm{M_G \omega_S}_2^3
 \label{eq:step:b:3}
\end{align}
for some constant $C\geq 1$ that may depend only on $\sigma$.
Combining \eqref{eq:step:b:1}--\eqref{eq:step:b:3} we arrive at
\begin{align}
 \frac 12 \frac{d}{dt} \norm{M_G \omega_S}_2^2 \leq \left(\dot{ \lambda}(t) + C_2 \lambda(t)  (1 + \norm{M_G \omega_S}_2) \right) \norm{|\nabla|^{s/2} M_G \omega_S}_2^2 + C_2 \norm{M_G \omega_S}_2^2 (1+ \norm{M_G \omega_S}_2)
 \label{eq:step:b:4}
\end{align}
for some constant $C_2 \geq 1$.
The initial datum by the assumption \eqref{ineq:Instant} obeys 
\[
\norm{M_G(0) \omega_{S,in}}_2 \leq \epsilon^\prime
\]
whenever $\lambda(0) \leq \lambda$. Letting 
\[
\lambda(t) = \frac{9 \lambda + \lambda'}{10} e^{-3 C_2 t}
\]
we obtain that (recall $T_1$ is defined in \eqref{eq:T1}), 
for $\epsilon'$ sufficiently small, depending on $\sigma,\lambda,\lambda^\prime$, 
for all 
\begin{align}
t \leq  T_2 = \min\left\{ T_1, \frac{1}{3C_2} \left( 1 + \log \frac{9 \lambda + \lambda'}{8 \lambda + 2 \lambda'} \right) \right\} = \frac{\nu^{1+\delta}}{2C_1} \qquad \mbox{(without loss of generality, as }\nu \leq \nu_0\mbox{)}, 
\label{eq:T2}
\end{align}
the following holds: 
\begin{align}
\norm{M_G(t) \omega_{S}(t)}_2 \leq 2 \epsilon^\prime \qquad \mbox{and} \qquad \lambda(t) \geq \frac{4 \lambda + \lambda'}{5}.
\label{eq:step:b:5}
\end{align}
We have now completed the proof of step (b), and shown that $\omega_S(t) \in \G^{\sqrt{\nu} t;1} \cap G^{\frac{4 \lambda + \lambda'}{5};s}$, with norm less than $2 \epsilon^\prime$, but only for times $t \leq T_2 \lesssim \nu^{1+\delta}$. \\

\noindent {\it Proof of (c).}
We now wish to show that the ``rough'' perturbation $\omega_R = \omega - \omega_S$, which in view of \eqref{def:NSE}, obeys
\begin{equation} \label{def:NSE:rough}
\left\{
\begin{array}{l}
  \partial_t \omega_R+ y \partial_x \omega_R + u_R \cdot \nabla \omega_R + u_S \cdot \nabla \omega_R + u_R \cdot \nabla \omega_S= \nu \Delta \omega_R, \\ 
  u_S  = \grad^{\perp}(\Delta)^{-1} \omega_S, u_R  = \grad^{\perp}(\Delta)^{-1} \omega_R \\\omega_R(0) = \omega_{R,in},
\end{array}
\right.
\end{equation}
is sufficiently small in $\G^{\sqrt{\nu} t;1}$ at time $t =T_2$.
The proof is very much similar to the first part of step (b), and we thus omit the redundant details. 
Before estimating the analytic norm of $\omega_R$ we estimate the $L^2$ norm of $u_R$, which solves \eqref{def:NSEMomentum} perturbed about $u_S$. From the energy estimate we have
\begin{align}
\frac 12 \frac{d}{dt} \norm{u_R}_2^2 + \nu \norm{\omega_R}_2^2 
&= - \int u^{y}_R u^{x}_R - \int u_R \cdot \nabla u_S u_R \notag\\
&\leq \norm{u_R}_2^2 (1 + \|\nabla u_S\|_{L^\infty}) \leq 3 \norm{u_R}_2^2
\end{align}
for all $t \leq T_2$, by the Sobolev embedding, which holds because $\sigma>5/2$, and we have used the bound \eqref{eq:step:b:5}. In particular, in view of the assumption \eqref{ineq:Instant} on the initial datum $u_{R,in}$ which is supported away from frequencies less than $1$,  we obtain that
\begin{align}
\norm{u_R(t)}_2 \leq 21 \norm{u_{R,in}}_2 \leq 21 \norm{\omega_{R,in}}_2 \leq 21 e^{-K_0 \nu^{-\frac{(3+2\delta)s}{2(1-s)}}} \epsilon^\prime, 
\label{eq:U:rough:low:freq}
\end{align}
for all $t \leq T_2$.

Since $\omega_{R,in} \in L^2$, and not necessarily in $H^\sigma$, we need to consider a new analytic Fourier-multiplier
\begin{align}
M_R(t,k,\eta) &=  e^{ \sqrt{\nu}\; t\; |k,\eta| }.
\end{align}
Then the solution $\omega_R$ of \eqref{def:NSE:rough}  obeys the a priori estimate
\begin{align}
 &\frac{1}{2} \frac{d}{dt} \norm{M_R \omega_R}_2^2 + \nu \norm{ |\nabla| M_R \omega_R}_2^2 \notag\\
 &\qquad = \sqrt{\nu}\norm{ |\nabla|^{1/2} M_R \omega_R}_2^2 
 + \sum_k \int_\eta M_R(k,\eta)^2 \, k \, \overline{\hat{\omega}_R(k,\eta)} \partial_\eta \hat{\omega}_R(k,\eta)  d\eta \notag\\
 &\qquad \ \ + \sum_{k,l} \int_\eta \! \int_\xi  M_R(k,\eta)  \overline{\hat{\omega}_R(k,\eta)} M_R(l,\xi) \hat{\omega}_R(l,\xi) M_R(k-l,\eta-\xi) \hat{\omega}_R(k-l,\eta-\xi)  \notag\\
 &\qquad \qquad \qquad \qquad \times \frac{(l,\xi)^\perp \cdot (k,\eta)}{ \abs{l,\xi}^2} \frac{M_R(k,\eta)}{M_R(l,\xi) M_R(k-l,\eta-\xi)}   d\eta d\xi\notag\\
 &\qquad \ \ + \sum_{k,l} \int_\eta \! \int_\xi M_R(k,\eta) \overline{\hat{\omega}_R(k,\eta)}  M_R(l,\xi)  \hat{\omega}_S(l,\xi) M_R(k-l,\eta-\xi)  \hat{\omega}_R(k-l,\eta-\xi)  \notag\\
 &\qquad \qquad \qquad \qquad \times \frac{(l,\xi)^\perp \cdot (k,\eta)}{ \abs{l,\xi}^2} \frac{M_R(k,\eta)}{M_R(l,\xi) M_R(k-l,\eta-\xi)}   d\xi d\eta\notag\\
&\qquad \ \ + \sum_{k,l} \int_\eta \! \int_\xi M_R(k,\eta) \overline{\hat{\omega}_R(k,\eta)}  M_R(l,\xi)   \hat{\omega}_R(l,\xi) M_R(k-l,\eta-\xi)  \hat{\omega}_S(k-l,\eta-\xi)  \notag\\
 &\qquad \qquad \qquad \qquad \times \frac{(l,\xi)^\perp \cdot (k,\eta)}{ \abs{l,\xi}^2} \frac{M_R(k,\eta)}{M_R(l,\xi) M_R(k-l,\eta-\xi)}  d\xi d\eta\notag\\
 &\qquad = \sqrt{\nu} \norm{ |\nabla|^{1/2} M_R \omega_S}_2^2 + T_R + N_R + L_{S,R} + L_{R,S}
 \label{eq:step:c:1}.
\end{align}
As before, similarly to \eqref{eq:step:a:3} we have
\begin{align}
|T_R| &\lesssim \sqrt{\nu} \; t \norm{ |\nabla|^{1/2} M_R \omega_S}_2^2 
\label{eq:step:c:2}
\end{align}
and similarly to \eqref{eq:step:a:5} we have
\begin{align}
|N_R| &\leq \frac{C}{\nu^{1+\delta}} \norm{M_R \omega_R}_2^2 (1 + \norm{M_R \omega_R}_2)^{2+\delta} + \frac{\nu}{2}  \norm{|\nabla| M_R \omega_R}_2^2
\label{eq:step:c:3}
\end{align}
for some constant $C\geq 1$ which depends on $\delta>0$.
 Bounding the first linear term $L_{S,R}$  is done exactly as in \eqref{eq:step:a:5} by splitting according to the relative size of $|(l,\xi)|$ and $1$, and using the already established bound  \eqref{eq:step:a:7}. We obtain
\begin{align}
|L_{S,R}| 
&\lesssim   \norm{|\nabla|  M_R \omega_S}_{2} \norm{M_R \omega_R}_2 \norm{|\nabla| M_R \omega_R}_2 +  \norm{u_S}_{2} \norm{M_R \omega_R}_2 \norm{|\nabla| M_R \omega_R}_2 \notag\\
&\leq \frac{C( \epsilon^\prime)^2}{\nu} \norm{M_R \omega_R}_2^2 + \frac{\nu}{4} \norm{|\nabla| M_R \omega_R}_2^2
\label{eq:step:c:4}
\end{align}
for all $t\leq T_2$. 
For the second linear term $L_{R,S}$  we recall that for the analytic norm of $\omega_S$ we have a Sobolev $H^\sigma$ correction which allows us to deal with the derivative loss in $u_R \cdot \nabla \omega_S$, and in addition we have a good control on the low frequencies in view of \eqref{eq:U:rough:low:freq}. We then obtain
\begin{align}
|L_{R,S}| 
&\lesssim  \norm{|\nabla| M_R \omega_R}_2 \norm{M_R \omega_R}_2 \norm{|\nabla| M_R \omega_S}_2  +  \norm{u_R}_2 \norm{M_R \omega_R}_2 \norm{|\nabla| M_R \omega_S}_2 \notag\\
&\leq \frac{C (\epsilon^\prime)^2}{\nu} \norm{M_R \omega_R}_2^2 + \frac{\nu}{4} \norm{|\nabla| M_R \omega_R}_2^2, 
\label{eq:step:c:5}
\end{align}
for some constant $C\geq1$. 
Combining \eqref{eq:U:rough:low:freq}--\eqref{eq:step:c:5}, and assuming without loss of generality that $\epsilon^\prime$ is sufficiently small (in terms of $C_1,C_2$), we arrive at
\begin{align}
\frac{d}{dt} \norm{M_R \omega_R}_2^2 \leq \frac{3C_1}{2\nu^{1+\delta}} \norm{M_R \omega_R}_2^2 (1 + \norm{M_R \omega_R}_2^2)^{1+\delta/2} + C_1  \epsilon'  \norm{\omega_{R,in}}_2^2
\label{eq:step:c:6}
\end{align}
for all $t\leq T_2$, and $\nu \leq 1$, where $C_1$ is the same constant as in \eqref{eq:step:a:6}. Since by definition (cf.~\eqref{eq:T1}) we have $T_2 \leq T_1 \leq \nu^{1+\delta}/ (2C_1)$, and $M_R(0) =1$, 
we obtain from \eqref{eq:step:c:6} that
\[
\norm{M_R(t) \omega_R(t)}_2^2 \leq 9 \norm{\omega_{R,in}}_2^2.
\]
for all $t \in [0,T_2]$. It thus follows from the assumption \eqref{ineq:Instant} that 
\begin{align}
\label{eq:step:c:7}
\norm{M_R(T_2) \omega_R(T_2)}_2 \leq 3 \epsilon^\prime e^{-K_0 \nu^{-\frac{(3+2\delta) s}{2(1-s)}}}.
\end{align}
We need the above estimate as it yields a Gevrey-$\frac 1s$ bound on $\omega^\nu_R$ at time $T_2$.
Indeed, by \eqref{ineq:SobExp} and \eqref{ineq:IncExp}, 
\begin{align}
\jap{k,\xi}^\sigma e^{\frac{4\lambda + \lambda'}{5} |k,\xi|^s}  
&\leq \jap{k,\xi}^\sigma e^{(\frac{4\lambda + \lambda'}{5})^{\frac{1}{1-s}} (2\sqrt{\nu} T_2)^{-\frac{s}{1-s}}} e^{\frac{\sqrt{\nu}}{2} T_2 |k,\xi|}\notag\\
&= \jap{k,\xi}^\sigma e^{- \frac{\nu^{\frac{3+2\delta}{2}}}{4C_1}  |k,\xi|}\;
 e^{(\frac{4\lambda + \lambda'}{5})^{\frac{1}{1-s}} (4C_1)^{\frac{s}{1-s}} \nu^{-\frac{(3+2\delta)s}{2(1-s)}}} e^{\sqrt{\nu}T_2 |k,\xi|} \notag\\
&\leq  e^{K_0 \nu^{-\frac{(3+2\delta)s}{2(1-s)}}} M_R(T_2)
\label{eq:step:c:8}
\end{align}
by letting $K_0$ be sufficiently large, {\em independently} of $\nu \leq 1$. Here we have recalled the definition of $T_2$ in \eqref{eq:T2}.
Combining \eqref{eq:step:c:7}--\eqref{eq:step:c:8}, we thus have shown that for $\epsilon^\prime$ sufficiently small, depending only on $s,\sigma,\lambda, \lambda^\prime$, and a sufficiently large $K_0$, depending on the same parameters, that 
\begin{align}
\norm{M_G(T_2) \omega_R(T_2)}_2 \leq 3 \epsilon^\prime, \qquad \mbox{and} \qquad \lambda(T_2) = \frac{4 \lambda+\lambda^\prime}{5}
\label{eq:step:c:9}
\end{align}
which concludes the proof of step (c).\\ 

\noindent {\it Proof of (d).} To conclude the proof of the proposition, we now mutatis-mutandi run the Gevrey-class regularity propagation on the time interval $[T_2, T_{\lambda,\lambda^\prime}]$, on the equation \eqref{def:NSE}, with initial datum $\omega_S(T_2) + \omega_R(T_2)$. The estimate \eqref{eq:step:b:4} holds on this interval, and thus, if we let
\begin{align}
\lambda(t) = \frac{4 \lambda + \lambda^\prime}{5} e^{-3C_2 (t-T_2)}
\end{align}
we see that 
\begin{align}
\norm{M_G(t)\omega(t)}_2 \leq \epsilon
\end{align}
for all $t \leq T_{\lambda,\lambda^\prime}$. Moreover, by letting the constant $C_0$ in the definition of $T_{\lambda,\lambda^\prime}$ be sufficiently large, we can ensure that 
\begin{align}
\lambda(t) \geq \lambda(T_{\lambda,\lambda^\prime}) \geq \frac{3 \lambda + \lambda^\prime}{4}
\end{align}
as claimed in \eqref{eq:analytic:to:do}.  To conclude the proof of Proposition  \ref{prop:InstantReg}, we note that the bound $\norm{u(T_{\lambda,\lambda^\prime})}_2 \leq \epsilon$ follows from a bound similar to \eqref{eq:u:S:control}, which holds for an $O(1)$ time interval. 
\end{proof}

\subsection{Coordinate transform} \label{sec:CoordinateTrans}
One of the main insights in \cite{BM13} was the use of a very specific change of coordinates
that allows to simultaneously `mod out' by the evolution of the time-dependent background shear flow  
and treat the mixing of this background shear as a perturbation of the Couette flow (in particular, to understand the nonlinear effect of the Orr mechanism). 
In this work we will adjust the change of coordinates used in \cite{BM13} to account for the viscous dissipation in the following manner: 
\begin{subequations}  \label{def:zv}
\begin{align} 
z(t,x,y) &= x - tv(t,y), \\ 
\frac{d}{dt}\left(t (v(t,y)-y) \right) & = U_0^x(t,y) + \nu t\partial_{yy}v(t,y), \label{def:v}
\end{align}
\end{subequations}
supplemented with the initial condition
\begin{align*} 
\lim_{t \rightarrow 0} t \left(v(t,y) - y\right) = 0. 
\end{align*}
Note that when $\nu = 0$, \eqref{def:zv} reduces to the change of coordinates used in \cite{BM13} for 2D Euler. 
As briefly discussed in \S\ref{sec:Statement}, the coordinate transform is adapting to the momentum transport due to the viscosity in the background shear flow. 
It is important to note that the shear flow driving the mixing, and hence the enhanced dissipation, is the unknown, evolving flow $(y + U_0^x(t,y),0)$. 

Define the following quantities
\begin{subequations} \label{def:coordchg}
\begin{align} 
C(t,v(t,y)) & = v(t,y) - y, \\
v^\prime(t,v(t,y)) & = (\partial_y v)(t,y), \\
v^{\prime\prime}(t,v(t,y)) & = (\partial_{yy}v)(t,y), \\
[\partial_t v](t,v(t,y)) &= (\partial_t v)(t,y), \\
f(t,z(t,x,y),v(t,y))  & = \omega(t,x,y), \\  
\phi(t,z(t,x,y),v(t,y)) & = \psi(t,x,y),  \\
\tilde u(t,z(t,x,y),v(t,y)) & = U^x(t,x,y). 
\end{align}
\end{subequations}
Denoting $\Delta_t$ as the Laplacian in the coordinates \eqref{def:zv}, we derive the new Biot-Savart law: 
\begin{align} 
\Delta_t \phi = \partial_{zz}\phi  + (v^\prime)^2\left( \partial_v - t \partial_z \right)^2\phi + v^{\prime\prime}\left(\partial_{v} - t \partial_z\right) \phi = f.  \label{def:Deltat}
\end{align}
For future notational convenience denote the `linear' Laplacian $\Delta_L$ (the Laplacian associated with the pure Couette flow that arises in the coordinates $z = x-tv, v=y$):
\begin{align} 
\Delta_L \phi = \partial_{zz}\phi  + \left( \partial_v - t \partial_z \right)^2\phi.  \label{def:DeltaL}
\end{align}
The vorticity equation \eqref{def:NSE} in the new variables becomes
\begin{align} 
\partial_t f + [\partial_t v]\partial_v f - \nu v^{\prime\prime} t\partial_z f + v^\prime \grad^{\perp}_{z,v} P_{\neq 0} \phi \cdot \grad_{z,v} f = \nu \Delta_t f. \label{def:vortPDE}
\end{align} 
Note the derivatives are with respect to $z$ and $v$; we will henceforth drop the subscripts and always assume this is the case unless otherwise indicated. Taking the $x$ average of the first equation in \eqref{def:NSEMomentum} and transforming to the coordinates defined by \eqref{def:zv} (using the notations \eqref{def:coordchg}) gives
\begin{align} 
\partial_t \tilde u_0 +  [\partial_t v]\partial_v \tilde u_0 + <v^\prime \grad^\perp P_{\neq 0} \phi \cdot \grad \tilde u > = \nu \Delta_t \tilde u_0. \label{def:moment1}   
\end{align}

Next, we derive expressions for the coordinate system that are amenable to energy estimates and derive a 
useful cancellation in \eqref{def:vortPDE} and \eqref{def:moment1}.
The parabolic aspect to the coordinate system in \eqref{def:v} is crucial for both of these. 
Define the auxiliary function  
\begin{align} 
g(t,v) = \frac{1}{t}\left(\tilde u_0(t,v) - C(t,v)\right),  \label{def:g}
\end{align} 
which roughly measures how the background shear flow is converging in time. 
In 2D Euler \cite{BM13}, this quantity is exactly $[\partial_t v]$, however due to the presence of the dissipation in \eqref{def:v} here we instead have the relation 
\begin{align}
[\partial_t v] = g + \nu v^{\prime\prime}. \label{eq:awesome}
\end{align} 
We also derive from the chain rule: 
\begin{subequations}
\begin{align} 
v^\prime \partial_v C(t,v)  & = v^\prime(t,v) - 1, \\ 
\partial_t C + [\partial_t v]\partial_v C & = [\partial_t v], \label{ineq:dtc} \\ 
v^\prime \partial_v v^\prime & = v^{\prime\prime}, \label{def:vpp} \\
v^{\prime\prime} & = \Delta_t C. \label{eq:vppDc}
\end{align}
\end{subequations}
Computing the time derivative of $g$ from \eqref{def:g} using \eqref{ineq:dtc}, \eqref{eq:awesome}, \eqref{def:moment1} and \eqref{eq:vppDc}, yields
\begin{align} 
\partial_t g &  = -\frac{g}{t} + \frac{1}{t}\left( -[\partial_t v]\partial_v \tilde u_0  - v^\prime <\grad^\perp P_{\neq 0} \phi\cdot \grad \tilde u> + \nu \Delta_t \tilde u_0 + [\partial_t v]\partial_v C - g - \nu v^{\prime\prime}\right) \nonumber \\ 
& = -\frac{2g}{t} - [\partial_t v]\partial_v g - \frac{v^\prime}{t} <\grad^\perp P_{\neq 0} \phi \cdot \grad \tilde u> + \nu \Delta_t g. \label{def:Intermedg}
\end{align}
The form of the dissipation in \eqref{def:v} was made precisely so that a natural dissipation arises on the LHS of \eqref{def:Intermedg}, as opposed to a dangerous forcing term such as $\Delta_t \tilde u_0$. 
The coordinate system is properly adapting itself to the effect of the dissipation that $\tilde u_0$ is experiencing.  
This will also manifest in another way.
Indeed, by \eqref{eq:awesome} and the definition of $\Delta_t$ \eqref{def:Deltat}, we reduce \eqref{def:Intermedg} to
\begin{align} 
\partial_t g + \frac{2g}{t} + g\partial_v g = -\frac{v^\prime}{t} <\grad^\perp P_{\neq 0} \phi \cdot \grad \tilde u> + \nu (v^\prime)^{2} \partial_{vv} g. \label{def:PDEg}
\end{align}
The cancellation 
will turn out to be quite useful, as while $[\partial_t v]$ will decay slowly for long times, $g$ will decay essentially like $O(t^{-2})$ uniformly in $\nu$ (the same rate as $[\partial_t v]$ in Euler \cite{BM13}).
This cancellation in \eqref{def:PDEg} occurs throughout several equations so define: 
\begin{align} 
\tilde{\Delta_t} = \partial_{zz} + (v^\prime)^{2}\left(\partial_{v} - t\partial_z \right)^2.
\end{align}
From \eqref{def:vortPDE}  and \eqref{eq:awesome}, we derive a similar cancellation on the vorticity evolution:  
\begin{align} 
\partial_t f + u \cdot \grad f = \nu \tilde{\Delta_t} f, \label{def:vortNice}
\end{align}
where we are defining the relative velocity field
\begin{align}
u(t,z,v) & = \begin{pmatrix} 
0 \\ g  
\end{pmatrix} 
+ v^\prime \grad^\perp P_{\neq 0}\phi. 
\end{align}
Similarly, if we denote $h = v^\prime - 1$, we derive from \eqref{def:v} that
\begin{align} 
\partial_t h + g\partial_v h & = \frac{-f_0 - h}{t} + \nu \tilde{\Delta_t} h. \label{def:PDEh} 
\end{align} 
Next, we define the unknown $\bar{h} =\frac{-f_0 - h}{t}$ (which quantifies the convergence of the background vorticity) and derive from \eqref{def:vortNice} and \eqref{def:PDEh} 
\begin{align} 
\partial_t \bar{h} + g \partial_v \bar{h} & = -\frac{2}{t}\bar{h} + \frac{v^\prime}{t}<\grad^\perp P_{\neq 0} \phi \cdot \grad f> + \nu \tilde{\Delta_t} \bar{h}. \label{def:barh}
\end{align}
Finally, note from the Biot-Savart law and the definitions of $\bar{h}$ and $g$ we have
\begin{align} 
v^\prime \partial_v g = \bar{h}. \label{eq:barhgRelat}
\end{align} 
The cancellations in \eqref{def:PDEg}, \eqref{def:vortNice}, \eqref{def:PDEh} and \eqref{def:barh} imply that we have effectively modded out the slow decay of $[\partial_t v]$ -- or more precisely, we are moving the coordinate system in such a way so that the slow part of $[\partial_t v]$ is balancing the lower order drift term coming from the $\Delta_t$ in the dissipation. 
The choice of \eqref{def:zv} appears to be the unique coordinate transformation that simultaneously satisfies this property as well as the properties used to obtain inviscid damping in \cite{BM13}. 

Note that from \eqref{def:zv}, the conservation laws of \eqref{def:NSE} and Jensen's inequality, we have
\begin{align} 
\norm{<f>}_p \lesssim \norm{f}_p \lesssim \norm{(v^\prime)^{-1}}^{1/p}_\infty \norm{\omega}_p \lesssim \epsilon\norm{(v^\prime)^{-1}}^{1/p}_\infty, \label{ineq:fLp}
\end{align}
for all $1 \leq p \leq \infty$.
Similarly, from \eqref{def:zv} it follows that 
\begin{align}
\norm{v^\prime(t) - 1}_1 \lesssim \norm{(v^\prime)^{-1}}_\infty \norm{v_y(t)-1}_1 \lesssim \norm{(v^\prime)^{-1}}_\infty\left(\sup_{\tau \leq t}\norm{\omega_0(\tau)}_1\right) \lesssim \epsilon\norm{(v^\prime)^{-1}}_\infty. \label{ineq:L1h}
\end{align}

\subsection{Energy Estimates}
In this section we describe the bootstrap argument that propagates estimates from $O(1)$ times to infinite times. 
The energy estimates will involve two main multipliers: 
\begin{subequations} \label{def:Norms}
\begin{align}
A(t,k,\eta) & = e^{\lambda(t)\abs{k,\eta}^s}\jap{k,\eta}^\sigma J_k(t,\eta), \\
A^\nu(t,k,\eta) & = e^{\lambda(t)\abs{k,\eta}^s}\jap{k,\eta}^{\beta}  \jap{D(t,\eta)}^\alpha \mathbf{1}_{k \neq 0}.  \label{def:Anu}
\end{align}
\end{subequations}
The properties and detailed definitions of both multipliers can be found in \S\ref{sec:PropNorm} below. 
The index $\lambda(t)$ is the main Gevrey$-\frac{1}{s}$ regularity and is chosen to satisfy (recall definition \eqref{eq:T:lambda}),
\begin{subequations} \label{def:lambdat}
\begin{align} 
\lambda(t) & = \frac{3}{4}\lambda + \frac{1}{4}\lambda^\prime, \quad t \leq T_{\lambda,\lambda^\prime} \label{def:lambdashort} \\   
\dot\lambda(t) & = -\frac{\delta_\lambda}{\jap{t}^{2\tilde q}}(1 + \lambda(t)), \quad t > T_{\lambda,\lambda^\prime}
\end{align} 
\end{subequations}
where $\delta_\lambda \approx \lambda - \lambda^\prime$ is a small parameter that ensures $\lambda(t) > \lambda/2 + \lambda^{\prime}/2$ and $\tilde q \in (1/2, s/8 + 7/16)$ is a parameter chosen by the proof in \cite{BM13}. 
Moreover, as in \cite{BM13}, we will assume without loss of generality that $s < 2/3$; for larger $s$, an additional 
Gevrey regularity correction $1/2 < \gamma < 2/3$ can be added to replace the role of $s$ in the ensuing proof.  

The high norm defined by $A$ is the same used in the work on the Euler equations \cite{BM13} (and in particular is independent of $\nu$). 
The corrector $J$ is defined in \S3 of \cite{BM13} and recalled here along with the relevant properties below in \S\ref{sec:InviscidNorm}. 
We will also need the stronger variants $A^R \geq A$, $J^R \geq J$ defined in \eqref{def:AR} below. 
The second multiplier, $A^\nu$, quantifies the enhanced dissipation effect through the inclusion of $D$, defined by 
\begin{align} 
D(t,\eta) & = \frac{1}{3\alpha}\nu \abs{\eta}^3 + \frac{1}{24 \alpha} \nu\left(t^3 - 8\abs{\eta}^3\right)_+. \label{def:D}
\end{align} 
We will choose $\beta + 3\alpha + 8 < \sigma$, so that a sizable regularity gap is maintained between $A$ and $A^\nu$ 
and $\beta > 3\alpha + 2$ (for less crucial technical convenience). 
Hence, for sufficiently high frequencies (depending on $t$), $A$ will define a stronger norm, whereas for lower frequencies $A^\nu$ is stronger (for $k \neq 0$).  
The point of introducing $D$ is so that a uniform bound expresses a powerful enhanced dissipation effect: 
\begin{align}
\norm{P_{\neq 0} f(t)}_{\mathcal{G}^{\lambda(t),\beta}} \lesssim_\alpha \jap{\nu t^3}^{-\alpha} \norm{A^\nu f(t)}_2. \label{ineq:AnuDecay}
\end{align} 
The role this estimate plays in the proof is something like a parabolic analogue of the lossy elliptic Lemma \ref{lem:LossyElliptic}.  

The proof is based on a bootstrap argument. 
Let $[T_{\lambda,\lambda^\prime},T^\star]$ be the largest connected, closed interval with left endpoint $T_{\lambda,\lambda^{\prime}}$ such that the following \emph{bootstrap hypotheses} hold for some fixed, positive constants $K_v$ and $K_D$ depending only on $\lambda,\lambda^\prime,\alpha$ and $s$ determined by the proof:
\begin{itemize} 
\item The `high frequency' controls (the first on the vorticity and the latter four on the coordinate system)
\begin{subequations} \label{ineq:Boot_HiFreq}
\begin{align} 
\norm{A f(t)}_2^2 + \frac{\nu}{10}\int_{T_{\lambda,\lambda^\prime}}^t\norm{\sqrt{-\Delta_L} Af(\tau)}_2^2 d\tau + \int_{T_{\lambda,\lambda^\prime}}^t CK_\lambda + CK_w d\tau  & \leq 4\epsilon^2, \label{ineq:Boot_energyHi} \\
t^{2+2s}\norm{\frac{A}{\jap{\partial_v}^s} \bar{h}}_2^2 + \frac{\nu}{10}\int_{T_{\lambda,\lambda^\prime}}^t \tau^{2+2s}\norm{\partial_v\frac{A}{\jap{\partial_v}^s} \bar{h}}_2^2 d\tau + \int_{T_{\lambda,\lambda^\prime}}^t CK^{v,2}_\lambda + CK^{v,2}_w d\tau  & \leq 4\epsilon^2,  \label{ineq:Boot_bhc} \\ 
t^{4-K_D\epsilon }\norm{g}_{\G^{\lambda,\sigma-6}}^2 + \nu\int_{T_{\lambda,\lambda^\prime}}^t \tau^{4-K_D\epsilon}\norm{\partial_vg}_{\G^{\lambda(\tau),\sigma-6}}^2 d\tau  & \leq 4\epsilon^2, \label{ineq:Boot_gc} \\ 
\int_{T_{\lambda,\lambda^\prime}}^t CK^{v,1}_\lambda + CK^{v,1}_w d\tau & \leq 4K_v\epsilon^2, \label{ineq:Boot_CK1}\\ 
\norm{A^R (v^\prime - 1)}_2^2 + \int_{T_{\lambda,\lambda^\prime}}^t \sum_{i = 1}^2 CCK^i_{\lambda} + CCK^i_{w} d\tau & \leq 4K_v\epsilon^2  \label{ineq:Boot_ARh};
\end{align}
\end{subequations}
\item the enhanced dissipation estimate
\begin{align}
\norm{A^\nu f}_2^2 + \frac{\nu}{10}\int_{T_{\lambda,\lambda^\prime}}^t\norm{\sqrt{-\Delta_L} A^\nu f(\tau)}_2^2 d\tau + \int_{T_{\lambda,\lambda^\prime}}^t CK^{\nu}_\lambda d\tau  & \leq 4\epsilon^2;  \label{ineq:Boot_energyLo} 
\end{align}
\item the low frequency controls
\begin{subequations} \label{ineq:Boot_Low}
\begin{align}
\norm{v^\prime - 1}_\infty & < \frac{3}{4}, \label{ineq:Boot_vprime} \\ 
\norm{v^\prime - 1}_2 + \jap{\nu t}^{1/2}\norm{\partial_v\left(v^\prime - 1\right)}_2 & \leq 4 K_v \epsilon\jap{\nu t}^{-1/4}, \label{ineq:Boot_hdecay} \\ 
\norm{\tilde u_0(t)}_2^2 + \frac{\nu}{10}\int_0^t \norm{v^\prime \partial_v \tilde u_0(\tau) }_2^2 d\tau & \leq 4K_v\epsilon^2. \label{ineq:Boot_KE}
\end{align}
\end{subequations}
\end{itemize}
The numerous `CK' terms (for \emph{Cauchy-Kovalevskaya} since the most important of these terms arise naturally from the weakening of the norms in time)
are given by 
\begin{subequations} \label{def:CKterms}
\begin{align}
CK_\lambda(t) & = -\dot{\lambda}(t)\norm{\abs{\grad}^{s/2}Af}_2^2, \\ 
CK_{w}(t) & = \sum_k\int\frac{\partial_t w_k(t,\eta)}{w_k(t,\eta)} e^{\lambda(t)\abs{k,\eta}^{s}}\jap{k,\eta}^\sigma \frac{e^{\mu\abs{\eta}^{1/2}}}{w_k(t,\eta)}A_{k}(t,\eta)\abs{\hat{f}_k(t,\eta)}^2 d\eta,  \label{def:CKw} \\ 
CK^\nu_\lambda(t) & = -\dot{\lambda}(t)\norm{\abs{\grad}^{s/2}A^\nu f(t)}_2^2, \\ 
CK_{w}^{v,2}(t) & = \jap{t}^{2+2s} \norm{\sqrt{\frac{\partial_t w}{w}} \frac{A}{\jap{\partial_v}^s}  \bar{h}(t) }_2^2, \\ 
CK_{\lambda}^{v,2}(t) & = \jap{t}^{2+2s} (-\dot{\lambda}(t))\norm{\abs{\partial_v }^{s/2}\frac{A}{\jap{\partial_v}^s} \bar{h}(t) }_2^2, \\ 
CK_{w}^{v,1}(t) & = \jap{t}^{2+2s} \norm{\sqrt{\frac{\partial_t w}{w}} \frac{A}{\jap{\partial_v}^s}g(t)}_2^2, \\ 
CK_{\lambda}^{v,1}(t) & = \jap{t}^{2+2s} (-\dot{\lambda}(t))\norm{\abs{\partial_v }^{s/2}\frac{A}{\jap{\partial_v}^s}g(t)}_2^2, 
\end{align}   
\end{subequations}
along with the `coefficient Cauchy-Kovalevskaya' terms
\begin{subequations} \label{def:QL}
\begin{align}
CCK_\lambda^1(t) & =  -\dot\lambda(t)\norm{\abs{\partial_v}^{s/2}A^R\left(1-(v^{\prime})^2\right)(t) }_2^2, \\ 
CCK_w^1(t) & = \norm{\sqrt{\frac{\partial_t w}{w}}A^R\left(1-(v^{\prime})^2\right)(t)}_2^2,  \\
CCK_\lambda^2(t) & = -\dot{\lambda}(t)\norm{\abs{\partial_v}^{s/2}\frac{A^R}{\jap{\partial_v}} v^{\prime\prime}(t) }_2^2, \\
CCK_w^2(t) & = \norm{\sqrt{\frac{\partial_t w}{w}}\frac{A^R}{\jap{\partial_v}} v^{\prime\prime}(t) }_2^2. 
\end{align}
\end{subequations}
For future convenience (see \eqref{def:CKw} above) we also define 
\begin{subequations}
\begin{align} 
\tilde{J}_k(t,\eta) & = \frac{e^{\mu\abs{\eta}^{1/2}}}{w_k(t,\eta)}, \label{def:tildeJB} \\ 
\tilde{A}_k(t,\eta) & = e^{\lambda(t)\abs{k,\eta}^{s}}\jap{k,\eta}^\sigma  \tilde J_k(t,\eta). 
\end{align} 
\end{subequations}
Note that $\tilde A \leq A$ and it turns out that if $\abs{k} \leq  \frac{1}{4}\abs{\eta}$ then $A \lesssim \tilde A$ (see Lemma \ref{basic} below). 

It is easy to show that the quantities on the left hand sides of \eqref{ineq:Boot_HiFreq}, \eqref{ineq:Boot_energyLo} and \eqref{ineq:Boot_Low} take values continuously in time for $t > 0$ due to the analyticity of solutions to \eqref{def:NSE} and 
similarly, using Proposition \ref{prop:InstantReg}, that $T^\star > T_{\lambda,\lambda^\prime}$. 
Note by \eqref{ineq:Boot_hdecay}, \eqref{ineq:Boot_vprime} and \eqref{def:vpp} we have
\begin{align}
\norm{v^{\prime\prime}}_2^2 = \norm{v^\prime \partial_v v^{\prime}}_2^2 \lesssim \norm{v^\prime}_\infty^2\norm{\partial_v v^{\prime}}_2^2 \lesssim \frac{\epsilon^2}{\jap{\nu t}^{3/2}}. \label{ineq:vppdecay}
\end{align} 

The primary step to the proof of Theorem \ref{thm:Main} is the following, which shows that the bootstrap controls can be propagated indefinitely.  

\begin{proposition} \label{prop:BootProp}
For $\epsilon \leq \epsilon_0$ chosen sufficiently small depending only on $s,\lambda,\lambda^\prime$ and $\alpha$,  
the following estimates hold uniformly on $t \in [T_{\lambda,\lambda^\prime},T^\star]$: 
\begin{itemize} 
\item The `high frequency' controls
\begin{subequations} \label{ineq:HiFreq}
\begin{align} 
\norm{A f(t)}_2^2 + \frac{\nu}{10}\int_{T_{\lambda,\lambda^\prime}}^t\norm{\sqrt{-\Delta_L} Af(\tau)}_2^2 d\tau + \int_{T_{\lambda,\lambda^\prime}}^t CK_\lambda + CK_w d\tau  & \leq 2\epsilon^2 \label{ineq:energyHi} \\
t^{2+2s}\norm{\frac{A}{\jap{\partial_v}^s} \bar{h}}_2^2 + \frac{\nu}{10}\int_{T_{\lambda,\lambda^\prime}}^t \tau^{2+2s}\norm{\partial_v\frac{A}{\jap{\partial_v}^s} \bar{h}}_2^2 d\tau + \int_{T_{\lambda,\lambda^\prime}}^t CK^{v,2}_\lambda + CK^{v,2}_w d\tau  & \leq 2\epsilon^2  \label{ineq:bhc} \\ 
t^{4-K_D\epsilon }\norm{g}_{\G^{\lambda,\sigma-6}}^2 + \frac{\nu}{10}\int_{T_{\lambda,\lambda^\prime}}^t \tau^{4-K_D\epsilon}\norm{\partial_vg}_{\G^{\lambda(\tau),\sigma-6}}^2 d\tau  & \leq 2\epsilon^2 \label{ineq:gc} \\ 
\int_{T_{\lambda,\lambda^\prime}}^t CK^{v,1}_\lambda + CK^{v,1}_w d\tau & \leq 2K_v\epsilon^2 \label{ineq:CK1}\\ 
\norm{A^R (v^\prime - 1)}_2^2 + \int_{T_{\lambda,\lambda^\prime}}^t \sum_{i = 1}^2 CCK^i_{\lambda} + CCK^i_{w} d\tau & \leq 2K_v\epsilon^2  \label{ineq:ARh};
\end{align}
\end{subequations}
\item the enhanced dissipation estimate
\begin{align}
\norm{A^\nu f}_2^2 + \frac{\nu}{10}\int_{T_{\lambda,\lambda^\prime}}^t\norm{\sqrt{-\Delta_L} A^\nu f(\tau)}_2^2 d\tau + \int_{T_{\lambda,\lambda^\prime}}^t CK^{\nu}_\lambda d\tau  & \leq 2\epsilon^2;  \label{ineq:energyLo} 
\end{align}
\item the low frequency controls
\begin{subequations} \label{ineq:Low}
\begin{align}
\norm{v^\prime - 1}_\infty & < \frac{5}{8} \label{ineq:vprime} \\ 
\norm{v^\prime - 1}_2 + \jap{\nu t}^{1/2}\norm{\partial_v\left(v^\prime - 1\right)}_2 & \leq 2 K_v \epsilon\jap{\nu t}^{-1/4}, \label{ineq:hdecay} \\ 
\norm{\tilde u_0(t)}_2^2 + \frac{\nu}{10}\int_0^t \norm{v^\prime \partial_v \tilde u_0(\tau) }_2^2 d\tau & \leq 2\epsilon^2, \label{ineq:KE} \\ 
\norm{f_0}_2 & \leq \frac{2K_v \epsilon}{\jap{\nu t}^{1/4}}.  \label{ineq:LoDecay}  
\end{align}
\end{subequations}
\end{itemize}
It therefore follows that $T^\star = \infty$. 
\end{proposition}

The proof of Proposition \ref{prop:BootProp} constitutes the majority of our work; let us briefly outline it here. 
Except for the presence of the dissipation integrals, the estimates \eqref{ineq:HiFreq} are essentially the same as those propagated in the bootstrap argument for 2D Euler in \cite{BM13} (along with \eqref{ineq:vprime}).
Since we want to be able to send $\nu \rightarrow 0$ independently of $t$ and $\epsilon$, the dissipation cannot help to get these estimates, and so the proof of \eqref{ineq:HiFreq} must also contain all the work done on the inviscid problem in \cite{BM13} (we will try to repeat as little as possible).  
The main new challenge to deducing \eqref{ineq:HiFreq} for the Navier-Stokes equations \eqref{def:NSE} is ensuring that the presence of the variable coefficients in $\tilde{\Delta_t}$ does not create growth or arrest the decay estimates. 
In \eqref{ineq:energyHi}, one could worry that having rapid variations in the $v$-dependent enhanced dissipation rate could amplify gradients of $f$. 
Moreover, the fact that the enhanced dissipation slows down near the critical times will cause additional problems here.
We will use \eqref{ineq:Boot_energyLo} to provide the crucial decay necessary to handle the effect of high frequencies in the coefficients of $\tilde{\Delta_t}$ strongly forcing non-zero modes in $z$. 
This is detailed in \S\ref{sec:HighNorm} below.  
The NSE-specific challenge to deducing \eqref{ineq:bhc}, \eqref{ineq:gc} and \eqref{ineq:ARh} is that having a $v$-dependent diffusivity in $\tilde{\Delta_t}$ can induce effective bulk transport in $v$, a phenomenon that is mostly associated with low frequencies in the derivative of the coefficients. The effect is nonlinear, as the coefficients are coupled to $f$, but is controlled by the low frequency decay estimate \eqref{ineq:hdecay}. This is detailed in \S\ref{sec:CoordHiReg}. 

The proof of the enhanced dissipation estimate \eqref{ineq:energyLo} can be found in \S\ref{sec:ED}.
The main issue confronted in \S\ref{sec:ED} is that there are terms in \eqref{def:vortNice} which are linear in the $k\neq 0$ frequencies. For these terms, we cannot use the enhanced dissipation and so we must still carefully use some of the structural properties of the Euler nonlinearity (e.g. the transport structure), as if $A^\nu f$ were the high norm. As a result, the chosen $A^\nu$ still needs to satisfy certain amenable properties. Moreover, these linear terms also express the strong forcing of non-zero frequencies in $z$ by the vastly larger zero frequencies in $z$. 
The regularity gap between $A$ and $A^\nu$ was chosen so that we could use \eqref{ineq:Boot_energyHi} to control this effect. 

The low frequency controls \eqref{ineq:Low} are relatively straightforward consequences of the slow viscous dissipation of $z$-independent (equivalently $x$-independent) quantities and are proved in \S\ref{sec:LoFreq} using the a priori high norm controls from \eqref{ineq:Boot_energyHi} and \eqref{ineq:Boot_energyLo} and standard Moser iteration techniques.

\subsection{Conclusion of the proof} \label{sec:ProofConcl}
To conclude the proof of Theorem \ref{thm:Main} from Proposition \ref{prop:BootProp} we need to translate the information back
into information on the original unknowns, $\omega$ and $\psi$. 
This is not trivial, but we may follow closely the corresponding steps in \cite{BM13} with some minor alterations. 
By Proposition \ref{prop:BootProp}, we have the global uniform bounds 
\begin{subequations} \label{ineq:apriori}
\begin{align} 
\norm{f(t)}_{\G^{\lambda(t),\sigma}}^2 + \jap{\nu t^3}^{2\alpha}\norm{P_{\neq 0} f(t)}_{\G^{\lambda(t),\beta}}^2 + \jap{\nu t}^{1/2}\norm{f_0(t)}_2^2  & \lesssim \epsilon^2  \\ 
\norm{\tilde u_0}^2_{\G^{\lambda(t),\sigma}} + \jap{\nu t^3}^{2\alpha} \jap{t}^{4}\norm{P_{\neq 0}\phi}^2_{\G^{\lambda(t),\beta-3}} 
 + \norm{v^\prime - 1}_{\G^{\lambda(t),\sigma}}^2 & \lesssim \epsilon^2. 
\end{align} 
\end{subequations}
Define $\lambda_\infty = \lim_{t \rightarrow \infty} \lambda(t)$. 
In order to complete the proof of Theorem \ref{thm:Main}, we undo the change of coordinates in $v$, switching to the coordinates $(z,y)$.
Writing $\tilde{\omega}(t,z,y) = f(t,z,v) = \omega(t,x,y)$ and $\tilde\psi(t,z,y) = \phi(t,z,v) = \psi(t,x,y)$, one derives from \eqref{def:NSE}, as in \S\ref{sec:CoordinateTrans}, that 
\begin{align}
\partial_t \tilde{\omega} + \grad_{z,y}^\perp P_{\neq 0}\tilde\psi \cdot \grad_{z,y}\tilde{\omega} = \nu \partial_{zz}\tilde{\omega} + \nu\left(\partial_y - t \partial_y U^x_0\partial_z\right)^2\tilde{\omega}. \label{ineq:2DhNSE}
\end{align}
Similarly, denoting $\tilde U(t,z,y) = \tilde u(t,z,v) = U^x(t,x,y)$, we deduce from \eqref{def:NSEMomentum}
\begin{align}
\partial_t U_0^x + <\grad^\perp_{z,y} P_{\neq 0} \tilde \psi \cdot \grad_{z,y} \tilde U> = \nu \partial_{yy} U^x_0. \label{def:Umomen}
\end{align}
We may follow the argument in \S2.3 of \cite{BM13}, apply the appropriate Gevrey inverse function theorem~\cite[Lemma A.5]{BM13} and the Gevrey composition inequality~\cite[Lemma A.4]{BM13} to deduce from \eqref{ineq:apriori} and \eqref{ineq:gradulossy} below that 
the estimates on $f,\phi$ and $\tilde u$ 
imply estimates on $\tilde{\omega},\tilde \psi$ and $\tilde U$.   
The main issue is inverting the coordinate transform $y = y(t,v)$ and deducing good Gevrey regularity estimates on $v(t,v)-y$ and $y(t,v)-v$. 
The only difference from \cite{BM13} is in establishing the $L^2$ control on $v(t,y) - y$, which here requires a straightforward estimate 
on the forced heat equation \eqref{def:v}.
We omit the details for brevity and conclude that, for $\epsilon$ sufficiently small, there exists some $\lambda_\infty^{\prime\prime\prime} \in (\lambda^\prime,\lambda_\infty)$ such that 
\begin{subequations} \label{ineq:hpsitildeest}
\begin{align} 
\norm{\tilde{\omega}(t)}^2_{\G^{\lambda_\infty^{\prime\prime\prime}}} + \jap{\nu t^3}^{2\alpha}\norm{P_{\neq 0}\tilde{\omega}(t)}^2_{\G^{\lambda_\infty^{\prime\prime\prime}}} + \jap{\nu t^3}^{2\alpha}\jap{t}^{4}\norm{P_{\neq 0}\tilde \psi(t)}^2_{\G^{\lambda_\infty^{\prime\prime\prime}}} + \jap{\nu t}^{1/2}\norm{\tilde{\omega}(t)}^2_2  &\lesssim \epsilon^2 \\
\norm{\tilde U_0}^2_{\G^{\lambda_\infty^{\prime\prime\prime}}} + \jap{t}^2 \jap{\nu t^3}^{2\alpha} \norm{P_{\neq 0} \grad_{z,y} \tilde U(t)}^2_{\G^{\lambda_\infty^{\prime\prime\prime}}} & \lesssim \epsilon^2. 
\end{align} 
\end{subequations}
In view of the definitions \eqref{def:zv} and \eqref{def:phi}, 
estimate \eqref{ineq:hpsitildeest} above implies the claimed bounds on the vorticity
\eqref{main-omega}, \eqref{ineq:thmED} and \eqref{ineq:thmSD}.  
The $L^2$ estimates \eqref{ineq:ID1} and \eqref{ineq:ID2} follow from shifting in the $z$ variable and using the decay estimates on $\psi$ in \eqref{ineq:hpsitildeest}.
From \eqref{ineq:2DhNSE} and \eqref{ineq:hpsitildeest} together, we further deduce \eqref{ineq:scattering}. 
Similarly, \eqref{def:Umomen} and \eqref{ineq:hpsitildeest} imply \eqref{ineq:P0Ux}. 
This completes the proof of Theorem \ref{thm:Main}.

\section{Properties of the norms} \label{sec:PropNorm}
In this section we detail the properties of the two multipliers $A$ and $A^\nu$. 
\subsection{Inviscid multiplier \texorpdfstring{$A$}{A}} \label{sec:InviscidNorm}
In \S3 of \cite{BM13}, the corrector $J$ is designed to deal with a possible frequency cascade caused by hydrodynamic echoes,
the weakly nonlinear manifestation of the Orr mechanism \cite{Boyd83,Orr07}.  
We briefly recall the definition here. 
For a constant $\mu > 0$ fixed below, we define 
\begin{align}
J_k(t,\eta) = \frac{e^{\mu \abs{\eta}^{1/2}}}{w_k(t,\eta)} + e^{\mu \abs{k}^{1/2}}, \label{def:J}
\end{align}  
for $w_k(t,\eta)$ specified next as an estimated `worst possible growth'. 
First recall the notation for critical intervals and times in \S\ref{sec:Notation}. 
For the remainder of \S\ref{sec:InviscidNorm} assume for notational simplicity that $\eta > 0$.
Finally, define $\kappa \in (0, 1/2)$ a fixed constant. 

Let $w_{NR}$ be a non-decreasing function of time with $w_{NR}(t,\eta) = 1  $ for $t \geq  2\eta $. 
Further, set $w_{NR}(t,\eta) = w_{NR}(t,10)$ for $\abs{\eta} < 10$. 
For $ k \geq 1$, we assume that $w_{NR} (t_{k-1,\eta})  $ was computed and for $k=1,2,3,..., E(\sqrt{\eta}) $, we define 
\begin{subequations} \label{def:wNR}
\begin{align} 
w_{NR}(t,\eta) &=   \Big( \frac{k^2}{\eta}   \left[ 1 + b_{k,\eta} |t-\frac{\eta}k | \right]   \Big)^{C\kappa}  w_{NR} (t_{k-1,\eta}),  \quad& 
  \quad  \forall t \in  I^R_{k,\eta} =  \left[ \frac{\eta}k ,t_{k-1,\eta}  \right], \\ 
w_{NR}(t,\eta) &=   \Big(1 + a_{k,\eta} |t-\frac{\eta}k |   \Big)^{-1-C\kappa}  w_{NR} \left(\frac{\eta}k\right),  \quad& 
  \quad \forall  t \in  I^L_{k,\eta} =  \left[ t_{k,\eta}  , \frac{\eta}k   \right].  
\end{align} 
\end{subequations}  
The constant $b_{k,\eta}  $   is chosen to ensure that $ \frac{k^2}{\eta}   \left[ 1 + b_{k,\eta} |t_{k-1,\eta}-\frac{\eta}k | \right]  =1$. Hence for $k \geq2$, we have 
\begin{align} \label{bk} 
 b_{k,\eta} = \frac{2(k-1)}{k} \left(1 - \frac{k^2}{\eta} \right)
\end{align} 
and $b_{1,\eta} = 1 - 1/\eta$. 
Similarly,  $a_{k,\eta}$ is chosen to ensure $ \frac{k^2}{\eta}\left[ 1 + a_{k,\eta} |t_{k,\eta}-\frac{\eta}k | \right] = 1$, which implies
\begin{align} \label{ak} 
 a_{k,\eta} = \frac{2(k+1)}{k} \left(1 - \frac{k^2}{\eta} \right). 
\end{align} 

Finally, we take $ w_{NR}$ to be constant on the interval $[0, t_{E(\sqrt{\eta}),\eta }] $, namely 
 $ w_{NR}(t,\eta)   = w(t_{E(\sqrt{\eta}),\eta } ,\eta)  $  for $ t \in [0, t_{E(\sqrt{\eta}),\eta }].$

On each interval $I_{k,\eta}  $, we define $w_R(t,\eta) $ by 
\begin{subequations} \label{def:wR}
\begin{align} 
w_R(t,\eta) &=   \frac{k^2}{\eta}   \left( 1 + b_{k,\eta} \abs{t-\frac{\eta}k} \right)w_{NR}(t,\eta),   \quad& 
  \quad \forall  t \in  I^R_{k,\eta} =  \left[ \frac{\eta}k ,t_{k-1,\eta}  \right], \\ 
w_R(t,\eta) &=   \frac{k^2}{\eta}   \left( 1 + a_{k,\eta} \abs{t-\frac{\eta}k} \right)   w_{NR}(t,\eta),    \quad& 
  \quad \forall  t \in  I^L_{k,\eta} =  \left[ t_{k,\eta}  , \frac{\eta}k   \right].   
\end{align}
\end{subequations}    
Due to the choice of $ b_{k,\eta}$  and $a_{k,\eta}$, we get that 
$w_R( t_{k,\eta}  ,\eta) =w_{NR}( t_{k,\eta}  ,\eta)      $  and $ w_R(  \frac{\eta}{k}  ,\eta) 
 =\frac{k^2}{\eta} 
 w_{NR}(  \frac{\eta}{k}   ,\eta)   $.

To define the full $w_k(t,\eta)$,  we then have 
\begin{align} 
w_k(t,\eta) =
\left\{
\begin{array}{ll} 
w_k(t_{E(\sqrt{{\eta}}),\eta},\eta) & t < t_{E(\sqrt{\eta}),\eta} \\
w_{NR}(t,\eta) & t \in [t_{E(\sqrt{\eta}),\eta},2\eta]\setminus I_{k,\eta} \\
w_{R}(t,\eta) & t \in I_{k,\eta} \\
1 & t \geq 2\eta. 
\end{array}
\right. \label{def:wk}
\end{align} 
Note that $w_k(t,\eta)$ is Lipschitz continuous in time. 
This completes the construction of the  $w_k(t,\eta)$ which appears in the $J$ defined above in \eqref{def:J}.  

We also define $J^R(t,\eta)$ and $A^R(t,\eta)$ to assign resonant regularity at \emph{every} critical time: 
\begin{align} 
J^R(t,\eta) & =
\left\{
\begin{array}{ll} e^{\mu\abs{\eta}^{1/2}} w^{-1}_R(t_{E(\sqrt{{\eta}}),\eta},\eta) + 1 & t < t_{E(\sqrt{\eta}),\eta} \\
  e^{\mu\abs{\eta}^{1/2}}w^{-1}_{R}(t,\eta) + 1 & t \in [t_{E(\sqrt{\eta}),\eta},2\eta] \\
e^{\mu\abs{\eta}^{1/2}} + 1 & t \geq 2\eta,
\end{array}
\right. \nonumber \\ 
A^R(t,\eta) & = e^{\lambda(t)\abs{\eta}^s}\jap{\eta}^\sigma J^R(t,\eta). \label{def:AR}
\end{align}
From \eqref{def:wR} we get that $A^R(t,\eta) \geq A_0(t,\eta)$ and since the zero frequency is always non-resonant from \eqref{def:wk}, near the critical times, $A^R$ can be as much as a factor of $\abs{\eta}$ larger. 

We next recall basic properties of $w$, $J$ and $A$. 
The proofs can be found in \S3 of \cite{BM13}. 
The first shows the origin of the requirement $s > 1/2$. 
\begin{lemma}[From \cite{BM13}] \label{basic}
There exists some constants $\mu > 0$, $C_0 >0$ such that for all $\abs{\eta} \geq 1$ there holds
\begin{align}  \label{w-grwth}
\frac1{w_k(0,\eta)}  =  \frac1{w_k( t_{E(\sqrt{\eta}),\eta},\eta)} 
  = \frac{C_0}{\eta^{\mu/8}}  e^{\frac{\mu}{2} \sqrt{\eta} } + o_{\abs{\eta} \rightarrow \infty}\left(\frac{C_0}{\eta^{\mu/8}}  e^{\frac{\mu}{2} \sqrt{\eta} }\right). 
\end{align} 
\end{lemma}

The next lemma emphasizes the technical advantage of having well-separated critical times.  

\begin{lemma}[From \cite{BM13}]  \label{lem:wellsep}
Let $\xi,\eta$ be such that there exists some $K \geq 1$ with $\frac{1}{K}\abs{\xi} \leq \abs{\eta} \leq K\abs{\xi}$ and let $k,n$ be such that $t \in I_{k,\eta}$ and $t \in I_{n,\xi}$  (note that $k \approx n$).  
Then at least one of the following holds:
\begin{itemize} 
\item[(a)] $k = n$ (almost same interval); 
\item[(b)] $\abs{t - \frac{\eta}{k}} \geq \frac{1}{10 K}\frac{\abs{\eta}}{k^2}$ and $\abs{t - \frac{\xi}{n}} \geq \frac{1}{10 K}\frac{\abs{\xi}}{n^2}$ (far from resonance);
\item[(c)] $\abs{\eta - \xi} \gtrsim_K \frac{\abs{\eta}}{\abs{n}}$ (well-separated). 
\end{itemize}
\end{lemma}

The next two fundamental lemmas allow first to compare $\partial_t w/w$ at 
different frequencies and then $J$ at different frequencies in nonlinear terms.  

\begin{lemma}[From \cite{BM13}] \label{lem:WtFreqCompare}
\begin{itemize}
\item[(i)] For $t \geq 1$, and $k,l,\eta,\xi$ such that  $ \max(2\sqrt{\abs{\xi}}, \sqrt{\abs{\eta}}) < t < 2  \min( \abs{\xi},  \abs{\eta})    $,
\begin{align} \label{dtw-xi} 
 \frac{\partial_t w_k(t,\eta)}{w_k(t,\eta)}\frac{w_l(t,\xi)}{\partial_t w_l(t,\xi)}
 \lesssim \jap{\eta - \xi} 
\end{align}
\item[(ii)] For all $t \geq 1$, and $k,l,\eta,\xi$, such that for some $K\geq1$, $\frac{1}{K}\abs{\xi} \leq \abs{\eta} \leq K\abs{\xi}$,  
\begin{align}
\sqrt{\frac{\partial_t w_l(t,\xi)}{w_l(t,\xi)}} \lesssim_K \left[\sqrt{\frac{\partial_t w_k(t,\eta)}{w_k(t,\eta)}} + \frac{\abs{\eta}^{s/2}}{\jap{t}^{s}}\right]\jap{\eta-\xi}. \label{ineq:partialtw_endpt}  
\end{align}
\item[(iii)] For $t \in I_{k,\eta}$ and $t> 2  \sqrt{\abs{\eta}}$, we have the following with $\tau = t - \frac\eta{k}$ and all $l \in \Integer$, 
\begin{align}  \label{dtw}
\frac{\partial_t w_{l}(t,\eta)}{w_{l}(t,\eta)}  
\approx  \frac{1}{1+\abs{\tau}}  \approx 
  \frac{\partial_t w_k(t,\eta)}{w_k(t,\eta)}. 
\end{align} 
\end{itemize}
\end{lemma}

\begin{lemma}[From \cite{BM13}] \label{lem:Jswap} 
In general we have 
\begin{align}
\frac{J_k(t,\eta)}{J_l(t,\xi)} \lesssim \frac{\abs{\eta}}{k^2\left( 1+ \abs{t - \frac{\eta}{k}} \right)} e^{9{\mu}\abs{k-l,\eta - \xi}^{1/2}}. \label{ineq:WFreqCompRes}
\end{align}  
If any one of the following holds: ($t \not\in I_{k,\eta}$) or ($k = l$) or ($t \in I_{k,\eta}$, $t \not\in I_{k,\xi}$ and $\frac{1}{K}\abs{\xi} \leq \abs{\eta} \leq K\abs{\xi}$ for some $K \geq 1$) or ($t \in I_{l,\xi}$) then we have the improved estimate  
\begin{align} 
\frac{J_k(t,\eta)}{J_l(t,\xi)} \lesssim e^{10 {\mu}\abs{k-l,\eta - \xi}^{1/2}}. \label{ineq:BasicJswap} 
\end{align} 
Finally if $t \in I_{l,\xi}$, $t \not\in I_{k,\eta}$ and $\frac{1}{K}\abs{\xi} \leq \abs{\eta} \leq K\abs{\xi}$ for some $K \geq 1$ then 
\begin{align} 
\frac{J_k(t,\eta)}{J_l(t,\xi)} \lesssim \frac{l^2\left(1 + \abs{t - \frac{\xi}{l}}\right)}{\abs{\xi}}e^{11{\mu}\abs{k-l,\eta - \xi}^{1/2}}. \label{ineq:WFreqCompNRGain}
\end{align}
\begin{remark} \label{rmk:GainLoss}
If $t \in \mathbf I_{k,\eta} \cap \mathbf I_{k,\xi}, k\neq l$, then
\begin{align} 
\frac{J_k(t,\eta)}{J_l(t,\xi)} \lesssim \frac{\abs{\eta}}{k^2}\sqrt{\frac{\partial_t w_k(t,\eta)}{w_k(t,\eta)}}\sqrt{\frac{\partial_t w_l(t,\xi)}{w_l(t,\xi)}}e^{20\mu\abs{k-l,\eta-\xi}^{1/2}}. \label{ineq:RatJ2partt}
\end{align}
\end{remark} 
\end{lemma} 

While $A$ does not define an algebra (by design), it does when restricted to the zero frequency in $z$, as is proved in the next lemma from \cite{BM13}.   
\begin{lemma}[Product lemma (from \cite{BM13})] \label{lem:ProdAlg}
For some $c \in (0,1)$, all $\sigma > 1$, all $\beta > -\sigma + 1$ and $\alpha \geq 0$, the following inequalities hold for all $f,g$ which depend only on $v$, 
\begin{subequations} \label{ineq:AProdProps}
\begin{align} 
\norm{\abs{\partial_v}^{\alpha}\jap{\partial_v}^\beta A(fg)}_2 & \lesssim \norm{f}_{\G^{c\lambda,\sigma}} \norm{\abs{\partial_v}^{\alpha}\jap{\partial_v}^\beta Ag}_2 + \norm{g}_{\G^{c\lambda,\sigma}} \norm{\abs{\partial_v}^{\alpha}\jap{\partial_v}^\beta Af}_2 \label{ineq:AProd} \\
\norm{\sqrt{\frac{\partial_tw}{w}}\jap{\partial_v}^\beta A(fg)}_2 & \lesssim \norm{g}_{\G^{c\lambda,\sigma}}\norm{\left(\sqrt{\frac{\partial_tw}{w}} + \frac{\abs{\partial_v}^{s/2}}{\jap{t}^s} \right)\jap{\partial_v}^\beta Af}_2 \nonumber \\ & \quad + \norm{f}_{\G^{c\lambda,\sigma}}\norm{\left(\sqrt{\frac{\partial_tw}{w}} + \frac{\abs{\partial_v}^{s/2}}{\jap{t}^s} \right)\jap{\partial_v}^\beta Ag}_2.   \label{ineq:wtAProd}
\end{align}
\end{subequations}
We also have for $\beta > -\sigma+1$ the algebra property,
\begin{align} 
\norm{\jap{\partial_v}^\beta A(fg)}_2 & \lesssim \norm{\jap{\partial_v}^\beta Af}_2 \norm{\jap{\partial_v}^\beta Ag}_2. \label{ineq:Aalg}
\end{align} 
Moreover, \eqref{ineq:AProdProps} and \eqref{ineq:Aalg} both hold with $A$ replaced by $A^R$. 
\end{lemma} 
\begin{remark} 
Writing $(v^\prime)^2 - 1 = (v^\prime - 1)^2 + 2(v^\prime - 1)$ and $v^{\prime\prime} = \partial_v(v^\prime - 1) + (v^\prime - 1)\partial_v(v^\prime - 1)$, by \eqref{ineq:Boot_ARh} combined with \eqref{ineq:Aalg} we have,
\begin{subequations} \label{ineq:coefbds}
\begin{align}
\norm{A^R\left(1 - (v^\prime)^2\right)}_2 \lesssim \norm{A^R\left(1 - v^\prime\right)}_2 + \norm{A^R\left(1 - v^\prime\right)}_2^2 & \lesssim \epsilon \label{ineq:vp2m1bd} \\ 
\norm{\frac{A^R}{\jap{\partial_v}} v^{\prime\prime}}_2 = \norm{\frac{A^R}{\jap{\partial_v}} \left(v^\prime \partial_v v^\prime \right)}_2 \lesssim \norm{A^R\left(1 - v^\prime\right)}_2 +  \norm{A^R\left(1 - v^\prime\right)}_2^2 & \lesssim \epsilon. 
\end{align}
\end{subequations}
\end{remark}

\subsection{Enhanced dissipation multiplier \texorpdfstring{$A^\nu$}{A nu}} \label{sec:Anu}
In this section we focus on the relevant properties of the enhanced dissipation multiplier and the associated semi-norm. 
The following lemma summarizes the properties of $D(t,\eta)$ (defined in \eqref{def:D}).  
The point is that $D(t,\eta)$ essentially trades regularity for time-decay when $t \geq 2\abs{\eta}$ but we need to ensure that it does so in a way that will define a norm suitable for use in nonlinear estimates. 

\begin{lemma}[Properties of $D(t,\eta)$] \label{lem:propD}
Uniformly in $\eta,\xi,t$ and $\nu$ we have:
\begin{itemize} 
\item[(a)] the lower bound
\begin{align} 
\max\left( \nu \abs{\eta}^3, \nu t^3\right) \lesssim \alpha D(t,\eta); \label{ineq:DLowB}
\end{align}
\item[(b)] the ratio estimate: 
\begin{align}   
\frac{\jap{D(t,\eta)}}{\jap{D(t,\xi)}} \lesssim \jap{\eta-\xi}^3; \label{ineq:DRat}
\end{align}
\item[(c)] the difference estimate: if there is some $K \geq 1$, such that $\frac{1}{K}\abs{\xi} \leq \abs{\eta} \leq K\abs{\xi}$, then
\begin{align} 
\abs{\jap{D(t,\eta)}^\alpha - \jap{D(t,\xi)}^\alpha} \lesssim_{\alpha,K}  \frac{\jap{D(t,\xi)}^{\alpha}}{\jap{\xi}}\jap{\eta-\xi}^{3\alpha}. \label{ineq:Dcomm}
\end{align}
\end{itemize}
\end{lemma}

\begin{proof} 
First, \eqref{ineq:DLowB} is immediate from \eqref{def:D}.

For \eqref{ineq:DRat}, write
\begin{align*}
\frac{\jap{D(t,\eta)}}{\jap{D(t,\xi)}} & = \frac{\jap{\frac{1}{3\alpha}\nu \abs{\eta}^3 + \frac{1}{24\alpha} \nu\left(t^3 - 8\abs{\eta}^3\right)_+}}{ \jap{\frac{1}{3\alpha}\nu \abs{\xi}^3 + \frac{1}{24 \alpha} \nu\left(t^3 - 8\abs{\xi}^3\right)_+}} \\ 
& \lesssim \jap{\eta-\xi}^3 + \frac{\jap{\frac{1}{24\alpha} \nu\left(t^3 - 8\abs{\eta}^3\right)_+}}{ \jap{\frac{1}{3\alpha}\nu \abs{\xi}^3 + \frac{1}{24\alpha} \nu\left(t^3 - 8\abs{\xi}^3\right)_+}}. 
\end{align*} 
If $t \leq 2\abs{\eta}$ then we have shown \eqref{ineq:DRat}, so assume otherwise. Then, 
 \begin{align*} 
\frac{\jap{D(t,\eta)}}{\jap{D(t,\xi)}} & \lesssim \jap{\eta-\xi}^3 + \frac{\jap{\frac{1}{24\alpha} \nu\left(t^3 - 8\abs{\xi}^3 + 8\abs{\xi}^3 - 8\abs{\eta}^3\right)}}{ \jap{\frac{1}{3\alpha} \nu \abs{\xi}^3 + \frac{1}{24\alpha} \nu\left(t^3 - 8\abs{\xi}^3\right)_+}} \\ 
& \lesssim \jap{\eta-\xi}^3 + \frac{\jap{\frac{1}{24\alpha} \nu\left(t^3 - 8\abs{\xi}^3\right)}}{ \jap{\frac{1}{3\alpha} \nu \abs{\xi}^3 + \frac{1}{24\alpha} \nu\left(t^3 - 8\abs{\xi}^3\right)_+}}.   
\end{align*} 
If $t \geq 2\abs{\xi}$ then \eqref{ineq:DRat} follows, so assume otherwise. Hence, the only case left is $2\abs{\eta} \leq t \leq 2\abs{\xi}$. However in this case, we get that \eqref{ineq:DRat} follows by $\abs{t-2\abs{\xi}} \leq 2\abs{\abs{\eta} - \abs{\xi}} \leq 2\abs{\eta-\xi}$. This covers all cases. 

Finally turn to \eqref{ineq:Dcomm}. First, from the mean value theorem and \eqref{ineq:DRat}, 
\begin{align*} 
\abs{\jap{D(t,\eta)}^\alpha - \jap{D(t,\xi)}^\alpha} & = \abs{\jap{\frac{1}{3\alpha} \nu \abs{\eta}^3 + \frac{1}{24\alpha} \nu\left(t^3 - 8\abs{\eta}^3\right)_+}^{\alpha} - \jap{\frac{1}{3\alpha} \nu \abs{\xi}^3 + \frac{1}{24\alpha} \nu\left(t^3 - 8\abs{\xi}^3\right)_+}^{\alpha}} \\ 
& \hspace{-4cm} \lesssim \alpha \sup_{\xi^\star \in [\eta,\xi]}\jap{D(t,\xi^\star)}^{\alpha - 1} \abs{\jap{\frac{1}{3\alpha} \nu \abs{\eta}^3 + \frac{1}{24\alpha} \nu\left(t^3 - 8\abs{\eta}^3\right)_+} - \jap{\frac{1}{3 \alpha} \nu \abs{\xi}^3 + \frac{1}{24 \alpha} \nu\left(t^3 - 8\abs{\xi}^3\right)_+}} \\ 
 & \hspace{-4cm} \lesssim \alpha \jap{\eta-\xi}^{3\alpha-3} \jap{D(t,\xi)}^{\alpha-1} \abs{\jap{\frac{1}{3\alpha} \nu \abs{\eta}^3 + \frac{1}{24\alpha} \nu\left(t^3 - 8\abs{\eta}^3\right)_+} - \jap{\frac{1}{3 \alpha} \nu \abs{\xi}^3 + \frac{1}{24 \alpha} \nu\left(t^3 - 8\abs{\xi}^3\right)_+}}.
\end{align*} 
If $t \geq 2\max(\abs{\eta},\abs{\xi})$ then we just have
\begin{align*} 
\abs{\jap{\frac{1}{3 \alpha} \nu \abs{\eta}^3 + \frac{1}{24 \alpha} \nu\left(t^3 - 8\abs{\eta}^3\right)_+} - \jap{\frac{1}{3 \alpha} \nu \abs{\xi}^3 + \frac{1}{24\alpha} \nu\left(t^3 - 8\abs{\xi}^3\right)_+}} = 0.
\end{align*}
If $t \leq 2\min(\abs{\eta},\abs{\xi})$ then the factor reduces to 
\begin{align*} 
\abs{\jap{\frac{1}{3 \alpha} \nu \abs{\eta}^3 + \frac{1}{24\alpha} \nu\left(t^3 - 8\abs{\eta}^3\right)_+} - \jap{\frac{1}{3\alpha} \nu \abs{\xi}^3 + \frac{1}{24\alpha} \nu\left(t^3 - 8\abs{\xi}^3\right)_+}} = \abs{\jap{\frac{1}{3 \alpha} \nu \abs{\eta}^3} - \jap{\frac{1}{3 \alpha} \nu \abs{\xi}^3}}.
\end{align*} 
This is essentially a statement about Sobolev regularity, and it follows from \eqref{ineq:DLowB} that
\begin{align*}
\abs{\jap{\frac{1}{3\alpha} \nu \abs{\eta}^3} - \jap{\frac{1}{3 \alpha} \nu \abs{\xi}^3}}&  \lesssim \left(\jap{\nu \abs{\eta}^2} + \jap{\nu \abs{\xi}^2}\right) \abs{\eta-\xi} \\ 
& \lesssim  \frac{\jap{D(t,\xi)}}{\jap{\xi}} \jap{\eta-\xi}^2 \abs{\eta-\xi}.
\end{align*}
If $2\abs{\xi} < t < 2\abs{\eta}$ then by \eqref{ineq:DLowB} and \eqref{ineq:DRat} we have
\begin{align*} 
\abs{\jap{\frac{1}{3\alpha} \nu \abs{\eta}^3 + \frac{1}{24 \alpha} \nu\left(t^3 - 8\abs{\eta}^3\right)_+} - \jap{\frac{1}{3 \alpha} \nu \abs{\xi}^3 + \frac{1}{24\alpha} \nu\left(t^3 - 8\abs{\xi}^3\right)_+}} & \\ 
& \hspace{-4cm} = \abs{\jap{\frac{1}{3\alpha} \nu \abs{\eta}^3} - \jap{\frac{1}{3\alpha} \nu t^3}} \\
& \hspace{-4cm}  \lesssim \frac{1}{\alpha}\left( \jap{\nu \abs{\eta}^2} + \jap{\nu t^2}\right) \abs{\eta - t} \\ 
& \hspace{-4cm}  \lesssim \frac{\jap{D(t,\xi)}}{\jap{\xi}} \abs{\eta - \xi}.
\end{align*}
Since \eqref{ineq:DRat} is symmetric in $\xi$ and $\eta$, this completes the proof of \eqref{ineq:DRat}.  
\end{proof}

From Lemma \ref{lem:propD}, and the choice of $\beta$ and $\sigma$, we may deduce the following important product inequality about $A^\nu$. 
The lemma allows us to deduce the correct time decay from products of functions, only one of which might be decaying. 
The important detail to notice is the loss of $3\alpha$ derivatives on the first factor in \eqref{ineq:Ddistri}.
The latter inequality, \eqref{ineq:DdistriDecay}, is not directly used in the proof of Theorem \ref{thm:Main} but regardless, the fact that it holds is important to consider when examining the main difficulties in the proof of \eqref{ineq:energyLo} in \S\ref{sec:ED}.

\begin{lemma}[$A^\nu$ Product Lemma] \label{lem:AnuProd} 
The following holds for all $q^1$ and $q^2$ such that $P_{\neq 0} q^2 = q^2$, 
\begin{align} 
\norm{A^\nu(q^1 q^2)}_2 & \lesssim \norm{q^1}_{\G^{\lambda,\beta + 3\alpha}}\norm{A^\nu q^2}_2 \label{ineq:Ddistri}  
\end{align}
If in addition we have $P_{\neq 0} q^1 = q^1$ then it follows that 
\begin{align} 
\norm{A^\nu(q^1 q^2)}_2 \lesssim \frac{1}{\jap{\nu t^3}^\alpha}\norm{A^\nu q^1}_{2}\norm{A^\nu q^2}_2. \label{ineq:DdistriDecay}
\end{align}
\end{lemma}
\begin{proof}
First we prove \eqref{ineq:Ddistri}. 
Expand via an inhomogeneous paraproduct
\begin{align*} 
\widehat{A^\nu(q^1q^2)}(k,\eta) & = \frac{1}{2\pi} \sum_{N \geq 8} \sum_{l \in \Integer} \int_{\xi} A_k^\nu(\eta) \widehat{q^1}_{l}(\xi)_{<N/8} \widehat{q^2}_{k-l}(\eta-\xi)_{N} d\xi \\ 
& \quad + \frac{1}{2\pi}\sum_{N \geq 8} \sum_{l \in \Integer} \int_{\xi} A_k^\nu(\eta) \widehat{q^1}_{l}(\xi)_{N} \widehat{q^2}_{k-l}(\eta-\xi)_{<N/8} d\xi \\ 
 & \quad + \frac{1}{2\pi}\sum_{N \in \mathbb D} \sum_{N/8 \leq N^\prime \leq 8N} \sum_{l \in \Integer} \int_{\xi} A_k^\nu(\eta) \widehat{q^1}_{l}(\xi)_{N^\prime} \widehat{q^2}_{k-l}(\eta-\xi)_{N} d\xi \\ 
& = T_{LH} + T_{HL} + T_{\R}.
\end{align*} 

Consider first $T_{LH}$. 
Since on the support of the integrand (see \S\ref{Apx:LPProduct}),  
\begin{subequations} \label{ineq:TLHFreqCon}
\begin{align} 
\abs{\abs{k,\eta} - \abs{k-l,\eta-\xi}} & \leq \abs{l,\xi} \leq \frac{3}{16}\abs{k-l,\eta-\xi}, \\
\frac{13}{16}\abs{k-l,\eta-\xi} &\leq \abs{k,\eta} \leq \frac{19}{16}\abs{k-l,\eta-\xi},  
\end{align} 
\end{subequations}
inequalities \eqref{lem:scon} and \eqref{ineq:DRat} imply that for some $c \in (0,1)$ (depending only on $s$ and our Littlewood-Paley conventions), 
\begin{align*} 
\abs{T_{LH}} & \lesssim \sum_{l \in \Integer: l \neq k} \sum_{N \geq 8} \int_\xi \jap{l,\xi}^{3\alpha} e^{c\lambda\abs{l,\xi}^s} \abs{\widehat{q^1}_l(\xi)_{<N/8}}  \jap{k-l,\eta-\xi}^\beta e^{\lambda\abs{k-l,\eta-\xi}^s}\jap{D(\eta-\xi)}^\alpha \abs{ \widehat{q^2}_{k-l}(\eta-\xi)_{N} } d\xi. 
\end{align*} 
Hence, by \eqref{ineq:L2L1}, $\beta > 1$ and the almost orthogonality \eqref{ineq:GeneralOrtho}, we have
\begin{align} 
\norm{T_{LH}}^2_2 & \lesssim \sum_{N \geq 8} \norm{q^1_{<N/8}}_{\G^{c\lambda,3\alpha+\beta}}^2 \norm{A^\nu q^2_{N}}_2^2 \nonumber \\ 
& \lesssim \norm{q^1}^2_{\G^{\lambda,3\alpha+\beta}} \norm{A^\nu q^2}^2_2.  \label{ineq:TLH}
\end{align}

Turn next to $T_{HL}$. On the support of the integrand we have, 
\begin{subequations} \label{ineq:THLFreqCon}
\begin{align} 
\abs{\abs{k,\eta} - \abs{l,\xi}} & \leq \abs{k-l,\eta-\xi} \leq \frac{3}{16}\abs{l,\xi}, \\
\frac{13}{16}\abs{l,\xi} &\leq \abs{k,\eta} \leq \frac{19}{16}\abs{l,\xi},  
\end{align} 
\end{subequations}
which implies that, by \eqref{lem:scon} and \eqref{ineq:DRat}, for some $c \in (0,1)$, there holds
\begin{align*} 
\abs{T_{HL}} 
& \lesssim \sum_{l \in \Integer} \sum_{N \geq 8} \int_\xi \abs{\jap{l,\xi}^{\beta + 3\alpha} e^{\lambda\abs{l,\xi}^s} \widehat{q^1}_l(\xi)_{N} \jap{D(\eta-\xi)}^\alpha \widehat{q^2}_{k-l}(\eta-\xi)_{<N/8} e^{c\lambda\abs{k-l,\eta-\xi}^s}} d\xi. 
\end{align*} 
Therefore, by \eqref{ineq:GeneralOrtho}, \eqref{ineq:L2L1} and $\sigma \geq \beta + 3\alpha$, $\beta > 1$, 
\begin{align} 
\norm{T_{HL}}^2_2 & \lesssim \sum_{N \geq 8}\norm{q^1_{N}}_{\G^{\lambda,3\alpha+\beta}}^2 \norm{\jap{D(t,\partial_v)}^\alpha q^2_{<N/8}}_{\G^{c\lambda,\beta}}^2 \nonumber \\ 
& \lesssim \norm{q^1}^2_{\G^{\lambda,3\alpha+\beta}} \norm{A^\nu q^2}^2_2.  \label{ineq:THL}
\end{align} 
This completes the treatment of $T_{HL}$. 

Finally turn to the remainder term $T_{\R}$. 
On the support of the integrand, there holds $\abs{l,\xi} \approx \abs{k-l,\eta-\xi}$ 
and therefore by \eqref{lem:strivial} and \eqref{ineq:DRat}, there is some $c \in (0,1)$ such that
\begin{align*}  
\abs{T_{\R}} & \lesssim \sum_{N \in \mathbb D} \sum_{N \approx N^\prime} \sum_{l \in \Integer} \int_{\xi} \jap{k,\eta}^\beta \jap{\xi}^{3\alpha} e^{c \abs{l,\xi}^s} \abs{\widehat{q^1}_{l}(\xi)_{N^\prime} e^{c \abs{k-l,\eta-\xi}^s} \jap{D(\eta-\xi)}^\alpha \widehat{q^2}_{k-l}(\eta-\xi)_{N}} d\xi \\ 
& \lesssim \sum_{N \in \mathbb D} \sum_{N \approx N^\prime} \sum_{l \in \Integer} \int_{\xi} N^{-1} \jap{k-l,\eta-\xi} \jap{l,\xi}^{3\alpha + \beta} e^{c \abs{l,\xi}^s} \abs{\widehat{q^1}_{l}(\xi)_{N^\prime} e^{c \abs{k-l,\eta-\xi}^s} \jap{D(\eta-\xi)}^\alpha \widehat{q^2}_{k-l}(\eta-\xi)_{N}} d\xi
\end{align*}
Therefore, by \eqref{ineq:L2L1} ($\beta > 2$),  
\begin{align} 
\norm{T_{\R}}_2 & \lesssim \sum_{N \in \mathbb D} N^{-1} \norm{q^1}_{\G^{c \lambda,3\alpha + \beta}} \norm{A^\nu q^2}_2 \nonumber \\ 
& \lesssim \norm{q^1}_{\G^{\lambda,3\alpha + \beta}} \norm{A^\nu q^2}_2.\label{ineq:TR}
\end{align}    
Upon taking square roots in \eqref{ineq:TLH} and \eqref{ineq:THL} and combining with \eqref{ineq:TR}, this completes the proof of \eqref{ineq:Ddistri}.

The proof of \eqref{ineq:DdistriDecay} is a slight variant of the proof of \eqref{ineq:Ddistri} except now the multiplier $A^\nu$ is always passed to the `high frequency' factor in the paraproduct and \eqref{ineq:AnuDecay} is used on the `low frequency' factor to introduce the additional decay. 
As \eqref{ineq:DdistriDecay} is not actually used in the proof of Theorem \ref{thm:Main}, we omit the details for brevity. 
\end{proof} 

\section{Elliptic estimates} \label{sec:Elliptic}

The following easy, but fundamental, lemma from \cite{BM13} shows that by paying regularity, one can still deduce
the same decay from $\Delta_t^{-1}$ as from $\Delta_L^{-1}$. 
The loss of three derivatives comes from the presence of $v^{\prime\prime} = v^\prime \partial_v v^\prime$ 
in the coefficients of $\Delta_t$. 
\begin{lemma}[Lossy elliptic estimate (from \cite{BM13})] \label{lem:LossyElliptic}
Under the bootstrap hypotheses, for $\epsilon$ sufficiently small, 
\begin{align} 
\norm{P_{\neq 0}\phi(t)}_{\G^{\lambda(t), \sigma - 3}} \lesssim \frac{\norm{f(t)}_{\G^{\lambda(t),\sigma-1}}}{\jap{t}^2}. \label{ineq:lossyelliptic}
\end{align} 
\end{lemma}

The following lemma quantifies the rapid decay of the $z$-dependent velocity field due 
to the enhanced dissipation, crucial to the proof of \eqref{ineq:energyLo} in \S\ref{sec:ED}. 
The lemma is `lossy' for two reasons. First, due to the presence of $A$ on the RHS of \eqref{ineq:LossyEllip2},
and second due to the hidden loss of $3\alpha$ derivatives on the coefficients of $\Delta_t$ induced by \eqref{ineq:Ddistri}. 
From \eqref{def:D} we see that for high frequencies $D$ simply appears as differentiation 
and Lemma \ref{lem:LossyEllip2} is reduced to Lemma \ref{lem:LossyElliptic}. 
For a given frequency, by the time $D$ begins to increase, it is already past the critical time and we do not need to pay regularity for decay in $\Delta_t^{-1}$. 

\begin{lemma}[Lossy elliptic estimate for $A^\nu$] \label{lem:LossyEllip2}
Under the bootstrap hypotheses for $\epsilon$ sufficiently small, for $\sigma \geq \beta + 3\alpha + 4$ there holds
\begin{align} 
\norm{A^\nu (\grad^\perp P_{\neq 0}\phi)}_2 \lesssim \frac{1}{\jap{t}^2}\left(\norm{A^\nu f}_2 + \norm{Af}_2\right). \label{ineq:LossyEllip2}
\end{align} 
\end{lemma}
\begin{proof}
We define the Fourier multipliers
\begin{align} 
\widehat{\left(\mathbf{1}_{t \leq 2\abs{\partial_v}} f\right)}_k(\eta) &= \mathbf{1}_{t \leq 2\abs{\eta}}\hat{f}_k(\eta), \\
\widehat{\left(\mathbf{1}_{t > 2\abs{\partial_v}} f\right)}_k(\eta) &= \mathbf{1}_{t > 2\abs{\eta}}\hat{f}_k(\eta), 
\end{align} 
and write 
\begin{align} 
\norm{A^\nu \grad^\perp P_{\neq 0}\phi }_2 \leq \norm{\mathbf{1}_{t \leq 2\abs{\partial_v}}A^\nu \grad^\perp P_{\neq 0}\phi}_2 + \norm{\mathbf{1}_{t > 2\abs{\partial_v}} A^\nu \grad^\perp P_{\neq 0}\phi}_2. \label{ineq:AnuLossyDecomp}
\end{align} 
Notice that if $t \leq 2\abs{\eta}$ there holds for $\sigma \geq \beta + 3\alpha + 4$,  
\begin{align*}
A^\nu_k(t,\eta) \approx e^{\lambda\abs{k,\eta}^s} \jap{k,\eta}^{\beta + 3\alpha} \lesssim e^{\lambda\abs{k,\eta}^s} \jap{k,\eta}^{\sigma - 4},
\end{align*}
and therefore, by Lemma \ref{lem:LossyElliptic},  
\begin{align} 
\norm{\mathbf{1}_{t \leq 2 \abs{\partial_v}} A^\nu \grad^\perp P_{\neq 0}\phi}_2 & \lesssim \norm{\grad^\perp P_{\neq 0} \phi}_{\G^{\lambda,\sigma- 4}} \lesssim \jap{t}^{-2}\norm{Af}_{2}. \label{ineq:AnuLossyHiFreq}
\end{align}
Next, focus on lower frequencies: 
\begin{align} 
\norm{\mathbf{1}_{t > 2\abs{\partial_v}}A^\nu P_{\neq 0}\phi}^2_{2} & = \sum_{k \neq 0}\int_\eta \mathbf{1}_{t > 2\abs{\eta}} e^{2\lambda \abs{k,\eta}^s}\jap{k,\eta}^{2\beta} \jap{D(\eta)}^\alpha  \abs{\hat{\phi}_k(\eta)}^2 d\eta \nonumber \\ 
& = \sum_{k \neq 0}\int_\eta \mathbf{1}_{t > 2\abs{\eta}} e^{2\lambda \abs{(k,\eta)}^s}\frac{\jap{k,\eta}^{2\beta} \jap{D(\eta)}^\alpha}{(k^2 + \abs{\eta - kt}^2)^2} (k^2 + \abs{\eta - kt}^2)^2 \abs{\hat{\phi}_k(\eta)}^2 d\eta \nonumber \\ 
& \lesssim \frac{1}{\jap{t}^4}\norm{\Delta_L P_{\neq 0}A^\nu \phi}^2_{2}. \label{ineq:DeltaLphi2}
\end{align}
As in the proof of Lemma \ref{lem:LossyElliptic} (see \cite{BM13}), we write $\Delta_t$ as a perturbation of $\Delta_L$ via 
\begin{align*}  
\Delta_LP_{\neq 0}\phi = P_{\neq 0}f + (1 - (v^{\prime})^2)(\partial_y - t\partial_z)^2P_{\neq 0}\phi - v^{\prime\prime}(\partial_y - t\partial_z)P_{\neq 0}\phi. 
\end{align*} 
Applying \eqref{ineq:Ddistri} implies
\begin{align*} 
\norm{\Delta_L A^\nu P_{\neq 0}\phi}_{2} \lesssim \norm{A^\nu f}_{2} + \norm{1 - (v^\prime)^2}_{\G^{\lambda,\beta + 3\alpha}}\norm{\Delta_L P_{\neq 0}A^\nu \phi}_{2} + \norm{v^{\prime\prime}}_{\G^{\lambda,\beta + 3\alpha}} \norm{(\partial_y - t\partial_z)P_{\neq 0}A^\nu \phi}_{2}. 
\end{align*}
Therefore, \eqref{ineq:coefbds} and $\sigma \geq \beta + 3\alpha + 1$ imply
\begin{align*} 
\norm{\Delta_L A^\nu P_{\neq 0}\phi}_{2} \lesssim \norm{A^\nu f}_{2} + \epsilon\norm{\Delta_L P_{\neq 0}A^\nu \phi}_{2}. 
\end{align*} 
 Together with \eqref{ineq:DeltaLphi2}, for $\epsilon$ is sufficiently small we get
\begin{align*} 
\norm{\mathbf{1}_{t > 2\abs{\partial_v}}A^\nu P_{\neq 0}\phi}_{2} & \lesssim \jap{t}^{-2} \norm{A^\nu f}_2, 
\end{align*} 
which, with \eqref{ineq:AnuLossyHiFreq} and \eqref{ineq:AnuLossyDecomp}, completes the proof of Lemma \ref{lem:LossyEllip2}. 
\end{proof}

\section{High norm vorticity estimate \eqref{ineq:energyHi}} \label{sec:HighNorm}
We are now ready to begin the proof of \eqref{ineq:energyHi}. 
Computing the evolution of $Af$ from \eqref{def:vortNice} gives, 
\begin{align} 
\frac{d}{dt}\frac{1}{2}\norm{Af}_2^2 = -CK_\lambda - CK_w - \int Af A(u\cdot \grad f)dv dz + \nu\int Af A \left(\tilde{\Delta_t} f\right) dv dz. \label{eq:AfEvo}
\end{align} 
Treating the third term, the Euler nonlinearity, comprises the majority of the work in \cite{BM13}. 
Here, we need only replace the role of $[\partial_t v]$ in \cite{BM13} with $g$ here and proceed in the same manner as 
used in \S5,\S6 and \S7 of \cite{BM13} to deduce (under the bootstrap hypotheses for $\epsilon$ sufficiently small) that
\begin{align} 
- \int Af A(u\cdot \grad f)dv dz &  \lesssim \epsilon CK_\lambda + \epsilon CK_w + \epsilon CK^{v,1}_{\lambda} + \epsilon CK^{v,1}_{w} \nonumber \\ & \quad + \epsilon^3\left(\sum_{i=1}^2 CCK_\lambda^i + CCK_w^i\right) + \frac{\epsilon^3}{\jap{t}^{1+s}}. \label{ineq:EulerHiNrm}
\end{align}
The new difficulty in deducing \eqref{ineq:energyHi} from \eqref{eq:AfEvo} is commuting $A$ and $\tilde{\Delta_t}$ in the last term. 
We follow an approach which is consistent with the proofs of the elliptic estimates in \S\ref{sec:Elliptic} and \cite{BM13}.
Following those arguments we write,  
\begin{align} 
\nu\int Af A(\tilde{\Delta_t} f)  & = \nu\int Af A(\Delta_L f) - \nu\int Af A\left[\left(1- (v^\prime)^2\right) (\partial_v - t\partial_z)^2 f\right] dv dz \nonumber \\  
& = -\nu\norm{\sqrt{-\Delta_L} A f}_2^2 - \nu\sum_{k \neq 0}\int Af_k A\left[\left(1- (v^\prime)^2\right) (\partial_v - t\partial_z)^2 f_k\right] dv dz \nonumber \\ & \quad - \nu\int Af_0 A\left[\left(1- (v^\prime)^2\right) \partial_v^2 f_0\right] dv dz \nonumber \\ 
& = -\nu\norm{\sqrt{-\Delta_L} A f}_2^2 + E^{\neq} + E^0. \label{eq:HiNrmDiss}
\end{align} 
Note that while the first term looks like a powerful dissipation term ($\sqrt{-\Delta_L}$ contains powers of $t$), since we are sending $\nu \rightarrow 0$ it can only really be useful for controlling the error terms $E^{\neq}$ and $E^0$.
Especially $E^{\neq}$ is dangerous as while the leading order dissipation degenerates near the critical times, we will see that $E^{\neq}$ can still be large.   
Since the operator $\tilde{\Delta_t}$ is very anisotropic between non-zero frequencies in $z$ and the zero frequencies, the two cases, $E^0$ and $E^{\neq}$, are treated separately by different methods in \S\ref{sec:DissErrNZHi} and \S\ref{sec:ZeroModeDiss} respectively. 

\subsection{Dissipation error term: non-zero frequencies} \label{sec:DissErrNZHi}
For future notational convenience, use the short hand 
\begin{align} 
G(t,v) & = \left(1-(v^\prime)^2\right)(t,v).  \label{def:Gshort}
\end{align} 
We begin the treatment of $E^{\neq}$ by decomposing with an inhomogeneous paraproduct in $v$ only: 
\begin{align*} 
E^{\neq} & =   \frac{\nu}{2\pi}\sum_{M \geq 8} \sum_{k \neq 0} \int A\overline{\hat{f}}_k(t,\eta) A_k(t,\eta)\hat{G}(\eta - \xi)_{M} \abs{\xi - tk}^2 \hat{f}_k(\xi)_{<M/8} d\xi d\eta \\ 
& \quad + \frac{\nu}{2\pi} \sum_{M \geq 8} \sum_{k \neq 0} \int A\overline{\hat{f}}_k(t,\eta) A_k(t,\eta)\hat{G}(\eta - \xi)_{<M/8} \abs{\xi - tk}^2 \hat{f}_k(\xi)_M d\xi d\eta \\ 
& \quad + \frac{\nu}{2\pi}\sum_{M \in \mathbb D} \sum_{M/8 \leq M^\prime \leq 8M} \sum_{k \neq 0} \int A\overline{\hat{f}}_k(t,\eta) A_k(t,\eta)\hat{G}(\eta - \xi)_{M} \abs{\xi - tk}^2 \hat{f}_k(\xi)_{M^\prime} d\xi d\eta \\ 
& = E^{\neq}_{HL} + E^{\neq}_{LH} + E^{\neq}_{\R}.  
\end{align*}  

The $E^{\neq}_{LH}$ term is the easier one, so we will treat this one first. 
The goal is to pass the $A_k(\eta)$ onto $\hat{f}_k(\xi)$ 
and split the $(\partial_v - t\partial_z)^{2}$ between the two factors of $f$.  
The latter will use the following triangle inequality, which holds specifically for $k \neq 0$,
\begin{align} 
\abs{\xi - kt} \leq \abs{\eta - \xi} + \abs{\eta - kt} \leq \jap{\eta - \xi}\left(\abs{k} + \abs{\eta - kt}\right). \label{ineq:triTriv}
\end{align}
This allows to split the derivatives by losing regularity on the coefficient, which is fine since it is in `low frequency' (we are essentially integrating by parts). 
Note that this will not be possible in treating $E^{\neq}_{HL}$. 
An important point is that \eqref{ineq:BasicJswap} (as opposed to \eqref{ineq:WFreqCompRes}) applies to transfer $A_k(\eta)$ to $A_k(\xi)$ - resonant vs non-resonant losses can only occur when comparing different frequencies in $z$. 
Since on the support of the integrand there holds 
\begin{subequations} \label{ineq:EneqFreqCon}
\begin{align} 
\abs{\abs{k,\eta} - \abs{k,\xi}} & \leq \abs{\xi - \eta} \leq \frac{3}{16}\abs{\xi} \leq \frac{3}{16}\abs{k,\xi}, \\
\frac{13}{16}\abs{k,\xi} &\leq \abs{k,\eta} \leq \frac{19}{16}\abs{k,\xi}, 
\end{align} 
\end{subequations}
which implies by \eqref{lem:scon} we have for some $c \in (0,1)$, 
\begin{align*}
\abs{E^{\neq}_{LH}} \lesssim \nu\sum_{M \geq 8} \sum_{k \neq 0} \int \abs{A\hat{f}_k(t,\eta)e^{c\lambda\abs{\eta-\xi}^s}\hat{G}(\eta - \xi)_{<M/8} \frac{J_k(\eta)}{J_k(\xi)} \abs{\xi - tk}^2 A\hat{f}_k(\xi)_M} d\xi d\eta. 
\end{align*} 
Applying \eqref{ineq:triTriv} together with \eqref{ineq:BasicJswap} followed by \eqref{ineq:IncExp} implies
\begin{align*} 
\abs{E^{\neq}_{LH}} \lesssim \nu\sum_{M \geq 8} \sum_{k \neq 0} \int \abs{ \sqrt{-\Delta_L}  A\hat{f}_k(t,\eta)e^{\lambda\abs{\eta-\xi}^s}\jap{\eta-\xi}\hat{G}(\eta - \xi)_{<M/8} \abs{\xi - tk} A\hat{f}_k(\xi)_M} d\xi d\eta,  
\end{align*} 
where we are slightly abusing notation by using $\sqrt{-\Delta_L} \hat{f}_k(\eta) = \left(\abs{k} + \abs{\eta-kt}\right) \hat{f}_k(\eta)$. 
Therefore, by \eqref{ineq:L2L2L1} ($\sigma > 1$), almost orthogonality \eqref{ineq:GeneralOrtho} and \eqref{ineq:coefbds} (recall the shorthand \eqref{def:Gshort}),
\begin{align} 
\abs{E^{\neq}_{LH}} & \lesssim \nu\sum_{M \geq 8} \sum_{k \neq 0} \norm{(1-(v^\prime)^2)_{<M/8}}_{\G^{\lambda,\sigma}} \norm{\sqrt{-\Delta_L}(Af_k)_{\sim M}}_2^2 \nonumber \\
& \lesssim \epsilon\nu \norm{\sqrt{-\Delta_L}Af}_2^2. \label{ineq:EneqLH}
\end{align} 
This term is then absorbed by the leading order dissipation term in \eqref{eq:HiNrmDiss} for $\epsilon$ sufficiently small.

Now, let us turn to the more delicate HL term. 
Since the paraproduct decomposition was with respect to $v$ but the norm depends on both $z$ and $v$, 
we will need to divide into separate contributions corresponding to when $k$ is large compared to $\eta$ and vice-versa. 
A similar issue arose, for example, in \cite[Proposition 2.5]{BM13}. 
Hence,  
\begin{align*} 
\abs{E^{\neq}_{HL}} & \lesssim \nu\sum_{M \geq 8} \int\left[\mathbf{1}_{\abs{k} \geq \frac{1}{16}\abs{\eta}} + \mathbf{1}_{\abs{k} < \frac{1}{16}\abs{\eta}}  \right] \abs{A\hat{f}_k(\eta)_{\sim M} A_k(t,\eta)\hat{G}(\eta - \xi)_{M} \abs{\xi - tk}^2\hat{f}_k(\xi)_{<M/8} }d\xi d\eta \\ 
& = E_{HL}^{\neq,z} + E_{HL}^{\neq,v}. 
\end{align*} 
In the `$z$' case,  we can assume that  the `derivatives' are still landing on $f_k(\xi)$, 
and hence we may treat this term in a manner similar to $E^{\neq}_{LH}$.
 On the support of the integrand, we claim that there is some $c \in (0,1)$ such that,
\begin{align} 
\abs{k,\eta}^s \leq \abs{k,\xi}^s + c\abs{\eta - \xi}^s. \label{ineq:lgainScon}
\end{align} 
To see \eqref{ineq:lgainScon}, one can consider separately the cases $\frac{1}{16}\abs{\eta} \leq \abs{k} \leq 16\abs{\eta}$ and $\abs{k} > 16\abs{\eta}$, applying \eqref{lem:strivial}  and \eqref{lem:scon} respectively.
Therefore, by \eqref{ineq:lgainScon} we have
\begin{align*}
\abs{E_{HL}^{\neq,z}} & \lesssim \nu\sum_{M \geq 8} \int \mathbf{1}_{\abs{k} \geq \frac{1}{16}\abs{\eta}} \abs{A\hat{f}_k(\eta)_{\sim M} e^{c\lambda\abs{\eta-\xi}} \hat{G}(\eta - \xi)_{M} J_k(\eta)  \abs{\xi - tk}^2 \jap{k,\eta}^\sigma e^{\lambda\abs{k,\xi}^s}\hat{f}_k(\xi)_{<M/8} }d\xi d\eta. 
\end{align*} 
Moreover, on the support of the integrand we have $\jap{k,\eta} \lesssim \jap{k} \leq \jap{k,\xi}$. 
Therefore, by \eqref{ineq:triTriv} and \eqref{ineq:BasicJswap} followed by \eqref{ineq:IncExp} ($c < 1$ and $s > 1/2$) we get
\begin{align*} 
\abs{E_{HL}^{\neq,z}} & \lesssim \nu\sum_{M \geq 8} \int \mathbf{1}_{\abs{k} \geq \frac{1}{16}\abs{\eta}} \abs{\sqrt{-\Delta_L}A\widehat{f}_k(\eta)_{\sim M} e^{\lambda\abs{\eta-\xi}^s} \jap{\eta-\xi} \hat{G}(\eta - \xi)_{M} \abs{\xi - tk} A\hat{f}_k(\xi)_{<M/8} }d\xi d\eta. 
\end{align*} 
By \eqref{ineq:L2L2L1} ($\sigma > 2$), Cauchy-Schwarz in $M$ and $k$, almost orthogonality \eqref{ineq:GeneralOrtho} and \eqref{ineq:coefbds} we get
\begin{align} 
\abs{E_{HL}^{\neq,z}} & \lesssim \nu\sum_{M \geq 8} \sum_{k \neq 0} \norm{(1-(v^\prime)^2)_{M}}_{\G^{\lambda,\sigma}} \norm{\sqrt{-\Delta_L}(Af_k)_{\sim M}}_2  \norm{\sqrt{-\Delta_L}(Af_k)_{<M/8}}_2 \nonumber \\
& \lesssim \epsilon\nu \norm{\sqrt{-\Delta_L}Af}_2^2 \label{ineq:EneqHLz}
\end{align} 
This term is then absorbed by the leading order dissipation term in \eqref{eq:HiNrmDiss} for $\epsilon$ sufficiently small.

Turn now to the more challenging $v$ case, $E_{HL}^{\neq,v}$, which corresponds to $\abs{k,\eta} \approx \abs{\eta-\xi} \gg \abs{k,\xi}$. 
The challenge here is that we cannot simply use \eqref{ineq:triTriv}, as this requires more regularity on the coefficients then we have to spend. 
Hence we will have to find another way of controlling this term and will find it most challenging near the critical times, as there the leading order dissipation could become weaker than the error term. 
Indeed, using that $\abs{\xi} \leq \frac{3}{16}\abs{\eta-\xi}$ on the support of the integrand, analogous to \eqref{ineq:EneqFreqCon} there holds on the support of the integrand: 
\begin{align}
 \abs{\abs{\eta - \xi}-\abs{k,\eta}} \leq \abs{k,\xi} \leq \frac{1}{16}\abs{\eta} + \abs{\xi} \leq \frac{1}{16}\abs{\eta-\xi} + \frac{17}{16}\abs{\xi} \leq \frac{67}{256}\abs{\eta-\xi}. \label{ineq:EneqHLvreqLoc}
\end{align}
Hence we may apply \eqref{lem:scon} to show there exists some $c \in (0,1)$ such that 
\begin{align*}
\abs{E_{HL}^{\neq,v}} & \lesssim \nu\sum_{M \geq 8} \int\mathbf{1}_{\abs{k} < \frac{1}{16}\abs{\eta}} \abs{A\hat{f}_k(\eta)_{\sim M} A\widehat{G}(\eta - \xi)_{M} \frac{J_k(\eta)}{J_0(\eta-\xi)} e^{c\lambda\abs{k,\xi}^s}\abs{\xi - tk}^2 \hat{f}_k(\xi)_{<M/8}}  d\xi d\eta. 
\end{align*} 
Near the resonant times, the treatment will be different, so divide further into resonant and non-resonant contributions: 
\begin{align*} 
\abs{E_{HL}^{\neq,v}} & \lesssim \nu\sum_{M \geq 8} \int \left[\chi^R + \chi^\ast \right]
 \abs{A\hat{f}_k(\eta)_{\sim M}  A\widehat{G}(\eta - \xi)_{M} \frac{J_k(\eta)}{J_0(\eta-\xi)} e^{c\lambda\abs{k,\xi}^s} \abs{\xi - tk}^2 \hat{f}_k(\xi)_{<M/8}} d\xi d\eta \\ 
& = E_{HL}^{\neq;v,R} + E_{HL}^{\neq;v,\ast},
\end{align*} 
where $\chi^R = \mathbf{1}_{t \in \mathbf{I}_{k,\eta-\xi} \cap \I_{k,\eta}}\mathbf{1}_{\abs{k} < \frac{1}{16}\abs{\eta}}$ and $\chi^\ast = (1-\mathbf{1}_{t \in \mathbf{I}_{k,\eta-\xi} \cap \I_{k,\eta}})\mathbf{1}_{\abs{k} < \frac{1}{16}\abs{\eta}}$. 

Turn first to the `R' contribution. 
Applying \eqref{ineq:WFreqCompRes} followed by \eqref{ineq:IncExp} implies (recall definitions \eqref{def:AR} and \eqref{def:wR}),  
\begin{align*} 
E_{HL}^{\neq;v,R} & \lesssim \nu \sum_{k \neq 0}\sum_{M \geq 8} \int \chi^R \abs{A \hat{f}_k(\eta)_{\sim M} A^R(\eta-\xi) \widehat{G}(\eta - \xi)_{M} \abs{\xi - tk}^2 e^{\lambda \abs{k,\xi}^s} \hat{f}_k(\xi)_{<M/8}} d\xi d\eta. 
\end{align*}  
Since we are near the resonant time, the leading dissipation term in \eqref{eq:HiNrmDiss} is very weak, and in particular, it cannot directly recover the $t^2$ present in the low frequency factor. 
In order to deal with this fundamental difficulty, we will use both the enhanced dissipation estimate \eqref{ineq:Boot_energyLo} as well as the $CCK$ control on $G$ expressed in \eqref{ineq:Boot_ARh}.  
Using 
\begin{align} 
\abs{k} + \abs{\xi - tk} & \lesssim \jap{t}\jap{k,\xi}, \label{ineq:TrivDeltaLBd}
\end{align} 
as well as \eqref{dtw}, \eqref{ineq:partialtw_endpt} and the definition of $\chi^R$ we have
\begin{align*} 
E_{HL}^{\neq;v,R}& \lesssim \nu \sum_{k \neq 0}\sum_{M \geq 8} \int \chi^R \left(\abs{k} + \abs{\eta - kt}\right)^{1/2} \abs{A\hat{f}_k(\eta)_{\sim M}} A^R(\eta-\xi) \left(\frac{1}{\sqrt{1 + \abs{\frac{\eta}{k} - t}}}\right) \\ &\quad \quad \times \widehat{G}(\eta - \xi)_{M}\left(\abs{k} + \abs{\xi - tk}\right)^{1/2}  \jap{t}^{3/2}\jap{k,\xi}^{3/2}e^{\lambda \abs{k,\xi}^s} \abs{\hat{f}_k(\xi)_{<M/8}} d\xi d\eta \\ 
& \lesssim \nu \sum_{k \neq 0}\sum_{M \geq 8} \int \chi^R \left(\abs{k} + \abs{\eta - kt}\right)^{1/2} \abs{A\hat{f}_k(\eta)_{\sim M} \left( \sqrt{\frac{\partial_t w(\eta-\xi)}{w(\eta - \xi)}} + \frac{\abs{\eta-\xi}^{s/2}}{\jap{t}^s}\right) A^R(\eta-\xi)} \\ &\quad \quad \times \widehat{G}(\eta - \xi)_{M}\left(\abs{k} + \abs{\xi - tk}\right)^{1/2}  \jap{t}^{3/2}\jap{k,\xi}^{5/2}e^{\lambda \abs{k,\xi}^s} \abs{\hat{f}_k(\xi)_{<M/8}} d\xi d\eta. 
\end{align*} 
Therefore, by \eqref{ineq:SobExp}, \eqref{ineq:L2L2L1} (using $\beta > 4$) and Cauchy-Schwarz,
\begin{align*} 
E_{HL}^{\neq;v,R} & \lesssim \nu t^{3/2} \sum_{k \neq 0} \sum_{M \geq 8} \norm{(-\Delta_L)^{1/4}A(f_k)_{\sim M}}_2 \norm{(-\Delta_L)^{1/4} f_k}_{\G^{\lambda,\beta}} \\ & \quad\quad \times \norm{\left(\sqrt{\frac{\partial_t w}{w}} + \frac{\abs{\partial_v}^{s/2}}{\jap{t}^s} \right)A^R \left(1- (v^\prime)^2\right)_M }_2 \\ 
& \lesssim \nu t^{3/2} \sum_{M \geq 8} \left( \sum_{k \neq 0} \norm{(-\Delta_L)^{1/4}A(f_k)_{\sim M}}^2_2\right)^{1/2} \left( \sum_{k \neq 0}\norm{(-\Delta_L)^{1/4} f_k}^2_{\G^{\lambda,\beta}} \right)^{1/2} \\ & \quad\quad \times \norm{\left(\sqrt{\frac{\partial_t w}{w}} + \frac{\abs{\partial_v}^{s/2}}{\jap{t}^s} \right)A^R \left(1- (v^\prime)^2\right)_M }_2 \\ 
& \lesssim \nu t^{3/2} \sum_{M \geq 8} \norm{Af_{\sim M}}_2^{1/2} \norm{f}_{\G^{\lambda,\beta}}^{1/2} \norm{(-\Delta_L)^{1/2}Af_{\sim M}}^{1/2}_2 \norm{(-\Delta_L)^{1/2}Af}^{1/2}_2 \\ 
& \quad\quad \times \norm{ \left(\sqrt{\frac{\partial_t w}{w}} + \frac{\abs{\partial_v}^{s/2}}{\jap{t}^s} \right) A^R \left(1- (v^\prime)^2\right)_M }_2. 
\end{align*}
Now we crucially apply the enhanced dissipation by using \eqref{ineq:AnuDecay} and \eqref{ineq:Boot_energyLo} to kill the extra powers of time, 
\begin{align*}  
E_{HL}^{\neq;v,R} & \lesssim  \frac{\nu t^{3/2}}{\jap{\nu t^3}^{\alpha/2}} \sum_{M \geq 8}\norm{Af_{\sim M}}_2^{1/2} \norm{A^\nu f}_2^{1/2} \norm{(-\Delta_L)^{1/2}Af_{\sim M}}^{1/2}_2 \norm{(-\Delta_L)^{1/2}Af}^{1/2}_2 \\ 
& \quad\quad \times \norm{ \left(\sqrt{\frac{\partial_t w}{w}} + \frac{\abs{\partial_v}^{s/2}}{\jap{t}^s} \right) A^R \left(1- (v^\prime)^2\right)_M }_2 \\ 
& \lesssim  \sqrt{\nu \epsilon} \sum_{M \geq 8}\norm{Af_{\sim M}}_2^{1/2} \norm{(-\Delta_L)^{1/2}Af_{\sim M}}^{1/2}_2 \norm{(-\Delta_L)^{1/2}Af}^{1/2}_2 \\ & \quad\quad \times \norm{ \left(\sqrt{\frac{\partial_t w}{w}} + \frac{\abs{\partial_v}^{s/2}}{\jap{t}^s} \right) A^R \left(1- (v^\prime)^2\right)_M }_2. 
\end{align*}
Therefore, by \eqref{ineq:Boot_energyLo}, \eqref{ineq:Boot_energyHi}, Cauchy-Schwarz in $M$ and almost orthogonality \eqref{ineq:GeneralOrtho},
\begin{align} 
E_{HL}^{\neq;v,R}  & \lesssim \nu \sum_{M \geq 8} \norm{Af_{\sim M}} \norm{(-\Delta_L)^{1/2}Af_{\sim M}}_2 \norm{(-\Delta_L)^{1/2}Af}_2 \nonumber \\ 
& \quad + \epsilon \sum_{M \geq 8}\norm{ \left(\sqrt{\frac{\partial_t w}{w}} + \frac{\abs{\partial_v}^{s/2}}{\jap{t}^s} \right) A^R \left(1- (v^\prime)^2\right)_M }_2^2  \nonumber \\ 
& \lesssim  \epsilon \nu \norm{(-\Delta_L)^{1/2}Af}_2^2 + \epsilon\norm{ \left(\sqrt{\frac{\partial_t w}{w}} + \frac{\abs{\partial_v}^{s/2}}{\jap{t}^s} \right) A^R \left(1- (v^\prime)^2\right)}_2^2  . \label{ineq:T1Hard}
\end{align} 
The first term is absorbed by the dissipation term in \eqref{eq:HiNrmDiss} and the latter term is controlled by the bootstrap control on the $CCK^1$ terms in \eqref{ineq:Boot_ARh}. 

The treatment of $E_{HL}^{\neq;v,\ast}$ is easier as the leading order dissipation is much stronger at this set of frequencies (although we still cannot apply \eqref{ineq:triTriv}).
By \eqref{ineq:BasicJswap} followed by \eqref{ineq:IncExp} we have
\begin{align*}
E_{HL}^{\neq;v,\ast} & \lesssim \nu\sum_{M \geq 8} \int A\abs{\hat{f}_k(\eta)_{\sim M}} \chi^\ast A_0(\eta-\xi)\abs{\hat{G}(\eta - \xi)_{M}} \abs{\xi - tk}^2 e^{\lambda \abs{k,\xi}^s} \abs{\hat{f}_k(\xi)_{<M/8}} d\xi d\eta. 
\end{align*}
We will introduce $\sqrt{-\Delta_L}$ on the first factor with the goal of directly using the dissipation: 
\begin{align*} 
E_{HL}^{\neq;v,\ast} & \lesssim \nu\sum_{M \geq 8} \int \sqrt{-\Delta_L} A\abs{\hat{f}_k(\eta)_{\sim M}} \chi^\ast A_0(\eta-\xi)\abs{\hat{G}(\eta - \xi)_{M}} \frac{\abs{\xi - tk}^2}{\sqrt{k^2 + \abs{\eta - kt}^2}} e^{\lambda \abs{k,\xi}^s} \abs{\hat{f}_k(\xi)_{<M/8}} d\xi d\eta.
\end{align*} 
By the definition of $\chi^\star$ and Lemma \ref{lem:wellsep} (using that $\eta \approx \eta-\xi$), on the support of the integrand we have one of two possibilities: either $\abs{\eta - kt} \gtrsim t$ or $\abs{\xi} \gtrsim t$ 
Therefore, using also \eqref{ineq:TrivDeltaLBd},  
\begin{align*} 
E_{HL}^{\neq;v,\ast} & \lesssim \nu\sum_{M \geq 8} \int \sqrt{-\Delta_L} A\abs{\hat{f}_k(\eta)_{\sim M}} \chi^\ast A_0(\eta-\xi)\abs{\hat{G}(\eta - \xi)_{M}} \jap{t}^{-1}\abs{\xi - tk}^2 \jap{\xi} e^{\lambda \abs{k,\xi}^s} \abs{\hat{f}_k(\xi)_{<M/8}} d\xi d\eta.
\end{align*} 
Therefore, by \eqref{ineq:L2L2L1} ($\beta > 3$), Cauchy-Schwarz (in $M$ and $k$) and \eqref{ineq:vp2m1bd} we have, 
\begin{align} 
E_{HL}^{\neq;v,\ast} &  \lesssim \nu\sum_{k \neq 0}\sum_{M \geq 8} \int \abs{\sqrt{-\Delta_L} A\hat{f}_k(\eta)_{\sim M}\chi^\ast A_0(\eta-\xi) \hat{G}(\eta - \xi)_{M}} \nonumber \\ 
& \quad\quad\quad \times \jap{k,\xi}^2 \abs{\xi - tk} e^{\lambda \abs{k,\xi}^s} \abs{\hat{f}_k(\xi)_{<M/8}} d\xi d\eta \nonumber \\ 
& \lesssim \nu\sum_{M \geq 8} \norm{A \left(1- (v^\prime)^2\right)_M}_2\norm{\sqrt{-\Delta_L}Af_{\sim M}}_2\norm{\sqrt{-\Delta_L}f_{<M/8}}_{\G^{\lambda,\beta}} \nonumber \\ 
& \lesssim \epsilon \nu\norm{\sqrt{-\Delta_L}Af}_2^2. \label{ineq:EHLneqAst}
\end{align}
Hence, for $\epsilon$ chosen sufficiently small, this is absorbed by the dissipation term in \eqref{eq:HiNrmDiss}. 

Finally, turn to the remainder term $E^{\neq}_\R$. 
Dividing into two contributions depending on the relative size of the $z$ vs $v$ frequencies, 
\begin{align*}
E^{\neq}_{\R} & \lesssim \nu \sum_{M \in \mathbb D} \sum_{M/8 \leq M^\prime \leq 8M} \sum_{k \neq 0} \int \left[ \mathbf{1}_{\abs{k} > 100\abs{\xi}} + \mathbf{1}_{\abs{k} \leq 100\abs{\xi}}\right] \abs{A \hat{f}_k(t,\eta) A_k(t,\eta)\widehat{G}(\eta - \xi)_{M} \abs{\xi - tk}^2 \hat{f}_k(\xi)_{M^\prime}} d\xi d\eta \\ 
& = E^{\neq;z}_{\R} + E^{\neq;v}_{\R}. 
\end{align*}  
Consider $E^{\neq;z}_{\R}$ first. On the support of the integrand,
\begin{align} 
\abs{\abs{k,\eta} - \abs{k,\xi}} \leq \abs{\eta - \xi} \leq \frac{3M^\prime}{2} \leq 12 M \leq 24\abs{\xi} \leq \frac{24}{100}\abs{k,\xi},\label{ineq:M1TR}
\end{align}
and hence inequality \eqref{lem:scon} implies there is some $c \in (0,1)$ such that
\begin{align*}
\abs{E^{\neq;z}_{\R}} & \lesssim \nu \sum_{M \in \mathbb D} \sum_{M \approx  M^\prime} \sum_{k \neq 0} \int \mathbf{1}_{\abs{k} > 100\abs{\xi}} \abs{A \hat{f}_k(t,\eta) e^{c\lambda\abs{\eta-\xi}^s}\widehat{G}(\eta - \xi)_{M} \abs{\xi - tk}^2 \frac{J_k(\eta)}{J_k(\xi)} A\hat{f}_k(\xi)_{M^\prime}} d\xi d\eta.  
\end{align*} 
Therefore, by \eqref{ineq:triTriv}, \eqref{ineq:BasicJswap}, \eqref{ineq:IncExp} and \eqref{ineq:SobExp} (using $c < 1$ and $s > 1/2$) we have
\begin{align*}
\abs{E^{\neq;z}_{\R}} & \lesssim \nu \sum_{M \in \mathbb D} \sum_{M \approx  M^\prime} \sum_{k \neq 0} \int \mathbf{1}_{\abs{k} > 100\abs{\xi}} \abs{\sqrt{-\Delta_L}A\widehat{f}_k(t,\eta) e^{\lambda\abs{\eta-\xi}^s}\widehat{G}(\eta - \xi)_{M} \abs{\xi - tk} A\hat{f}_k(\xi)_{M^\prime}} d\xi d\eta.  
\end{align*} 
Therefore, by \eqref{ineq:L2L2L1} followed by Cauchy-Schwarz in $k$ and $M$, almost orthogonality \eqref{ineq:GeneralOrtho} and \eqref{ineq:coefbds}, 
\begin{align} 
\abs{E^{\neq;z}_{\R}} & \lesssim \nu \sum_{M \in \mathbb D} \sum_{k \neq 0} \norm{\sqrt{-\Delta_L}A f_k}_2 \norm{\left(1 - (v^\prime)^2 \right)_{M}}_{\G^{\lambda,\sigma}} \norm{\sqrt{-\Delta_L}A (f_k)_{\sim M}}_2 \nonumber \\ 
& \lesssim \epsilon \nu \norm{\sqrt{-\Delta_L}A f }^2_2,  \label{ineq:EneqRz}
\end{align} 
which can then be absorbed by the leading dissipation term in \eqref{eq:HiNrmDiss} for $\epsilon$ sufficiently small.  

Finally turn to $E^{\neq;v}_{\R}$. 
On the support of the integrand in this case, 
\begin{align*} 
\abs{\eta - \xi} & \leq 24\abs{\xi} \leq 24\abs{k,\xi} \\ 
\abs{k,\xi} & \leq 101\abs{\xi} \leq 2424\abs{\eta - \xi}. 
\end{align*}
Therefore by \eqref{lem:strivial} there exists a $c \in (0,1)$ such that
\begin{align*} 
\abs{E^{\neq;v}_{\R}} & \lesssim \nu \sum_{M \in \mathbb D} \sum_{M \approx  M^\prime} \sum_{k \neq 0} \int \mathbf{1}_{\abs{k} \leq 100\abs{\xi}} \abs{A \hat{f}_k(\eta) e^{c\lambda\abs{\eta-\xi}^s}\jap{\xi - \eta}^{\sigma/2}\widehat{G}(\eta - \xi)_{M}} \\ & \quad\quad \times \abs{\xi - tk}^2 J_k(\eta) e^{c\lambda\abs{k,\xi}^s} \jap{k,\xi}^{\sigma/2} \abs{\hat{f}_k(\xi)_{M^\prime}} d\xi d\eta.   
\end{align*} 
By Lemma \ref{basic}, \eqref{ineq:IncExp}, \eqref{ineq:triTriv} and \eqref{ineq:SobExp}, 
\begin{align*} 
\abs{E^{\neq;v}_{\R}} & \lesssim \nu \sum_{M \in \mathbb D} \sum_{M \approx  M^\prime} \sum_{k \neq 0} \int \mathbf{1}_{\abs{k} \leq 100\abs{\xi}} \abs{\sqrt{-\Delta_L} A \widehat{f}_k(t,\eta) e^{\lambda\abs{\eta-\xi}^s} \widehat{G}(\eta - \xi)_{M}} \\ & \quad\quad \times \abs{\xi - tk} e^{\lambda\abs{k,\xi}^s} \abs{\hat{f}_k(\xi)_{M^\prime}} d\xi d\eta. 
\end{align*}  
Finally, by \eqref{ineq:L2L2L1}, Cauchy-Schwarz in $M$ and $k$, almost orthogonality \eqref{ineq:GeneralOrtho} and \eqref{ineq:coefbds}, 
\begin{align} 
\abs{E^{\neq;v}_{\R}} & \lesssim \nu \sum_{M \in \mathbb D} \sum_{k \neq 0} \norm{\sqrt{-\Delta_L}A f_k}_2 \norm{\left(1 - (v^\prime)^2 \right)_{M}}_{\G^{\lambda,\sigma}} \norm{\sqrt{-\Delta_L}A (f_k)_{\sim M}}_2 \nonumber \\ 
& \lesssim \epsilon \nu \norm{\sqrt{-\Delta_L}A f}^2_2,  \label{ineq:EneqRv}
\end{align} 
which can then be absorbed by the leading dissipation term in \eqref{eq:HiNrmDiss} for $\epsilon$ sufficiently small.  
This completes the treatment of $E^{\neq;v}_{\R}$ and of the entire non-zero frequency error term $E^{\neq}$. 

\subsection{Dissipation error term: zero frequencies} \label{sec:ZeroModeDiss}
For treating $E^0$ in \eqref{eq:HiNrmDiss}, the main challenge is dealing with the contributions of low frequencies, 
where the leading dissipation term in \eqref{eq:HiNrmDiss} will not directly control $E^0$. 
Physically, there is an effective transport of enstrophy due to the motion of the coordinate system inducing a variable dissipation coefficient. 
This low frequency effect is controlled by \eqref{ineq:Boot_hdecay}, which ensures that the coordinate system 
is relaxing sufficiently fast.
Using again the shorthand \eqref{def:Gshort}, begin by decomposing $E^0$ with a \emph{homogeneous} paraproduct: 
\begin{align*} 
E^0 & = \nu \sum_{M \in 2^\Integer} \int A f_0 A\left( G(t,v)_{M} \partial_{vv}(f_0)_{< M/8}\right) dv + \nu \sum_{M \in 2^\Integer} \int A f_0 A\left( G(t,v)_{<M/8} \partial_{vv}(f_0)_{M}\right) dv \\  & \quad  + \nu \sum_{M \in 2^\Integer} \sum_{M/8 \leq M^\prime \leq 8M} \int A f_0 A\left( G(t,v)_{M^\prime} \partial_{vv}(f_0)_{M}\right) dv \\ 
& = E^0_{HL} + E^0_{LH} + E^0_{\mathcal{R}}. 
\end{align*} 
Consider 
\begin{align*} 
\abs{E^0_{HL}} & \lesssim \nu \sum_{M \in 2^\Integer} \int \abs{A \hat{f_0}(\eta)_{\sim M} A_0(\eta) \widehat{G}(\eta-\xi)_{M}  \xi^2 \hat{f_0}(\xi)_{< M/8} } d\eta.  
\end{align*}
On the support of the integrand, 
\begin{subequations} \label{ineq:E0HLFreqLoc}
\begin{align} 
\abs{\abs{\eta} - \abs{\eta-\xi}} & \leq \abs{\xi} \leq \frac{3}{16}\abs{\eta-\xi}, \\
\frac{13}{16}\abs{\eta-\xi} &\leq \abs{\eta} \leq \frac{19}{16}\abs{\eta-\xi},  
\end{align} 
\end{subequations}
which implies that by \eqref{lem:scon} and \eqref{ineq:BasicJswap} followed by \eqref{ineq:IncExp} we have, 
\begin{align*} 
\abs{E^0_{HL}} & \lesssim \nu \sum_{M \in 2^\Integer} \int \abs{A \hat{f_0}(\eta)_{\sim M} A_0(\eta-\xi) \widehat{G}(\eta-\xi)_{M} \xi^2 \hat{f_0}(\xi)_{< M/8} e^{\lambda \abs{\xi}^s}} d\eta.  
\end{align*} 
Due to \eqref{ineq:E0HLFreqLoc}, on the support of the integrand there always holds $\abs{\xi} \lesssim \abs{\eta}$ and therefore by \eqref{ineq:L2L2L1}, Cauchy-Schwarz in $M$ and \eqref{ineq:vp2m1bd} (also almost orthogonality \eqref{ineq:GeneralOrtho} and $\sigma > 1$)
\begin{align} 
\abs{E^0_{HL}} & \lesssim \nu \sum_{M \in 2^\Integer} \int \abs{\eta A \hat{f_0}(\eta)_{\sim M} A_0(\eta-\xi) \widehat{G}(\eta-\xi)_{M}\xi \hat{f_0}(\xi)_{< M/8} e^{\lambda \abs{\xi}^s}} d\eta \nonumber \\ 
& \lesssim  \nu\sum_{M \in 2^\Integer} \norm{\partial_v A (f_0)_{\sim M}}_2\norm{\partial_v (f_0)_{<M/8}}_{\G^{\lambda,\sigma}} \norm{A\left(1 - (v^\prime)^2\right)_M}_2 \nonumber \\ 
& \lesssim  \nu \epsilon \norm{\partial_v A f_0}_2^2, \label{ineq:E0HL}
\end{align} 
which is then absorbed by the dissipation term in \eqref{eq:HiNrmDiss} for $\epsilon$ sufficiently small.   

The $E^0_{LH}$ term is treated similarly. 
The analogue of \eqref{ineq:E0HLFreqLoc} (with the role of $\xi$ and $\eta-\xi$ swapped) holds on the support of the 
integrand and hence from \eqref{lem:scon} and \eqref{ineq:BasicJswap} followed by \eqref{ineq:IncExp} ($s > 1/2$), we have, 
\begin{align*} 
\abs{E^0_{LH}} & \lesssim \nu \sum_{M \in 2^\Integer} \int \abs{A \hat{f_0}(\eta)_{\sim M} e^{\lambda\abs{\eta-\xi}^s} \widehat{G}(\eta-\xi)_{<M/8} \xi^2 A_0(\xi) \hat{f_0}(\xi)_{M} } d\eta \\ 
& \lesssim \nu \sum_{M \in 2^\Integer} \int \abs{\eta A \hat{f_0}(\eta)_{\sim M} e^{\lambda\abs{\eta-\xi}^s} \widehat{G}(\eta-\xi)_{<M/8} \xi A_0(\xi) \hat{f_0}(\xi)_{M} } d\eta. 
\end{align*}
Then by \eqref{ineq:L2L2L1}, almost orthogonality \eqref{ineq:GeneralOrtho}, Cauchy-Schwarz in $M$, $\sigma > 1$ and \eqref{ineq:vp2m1bd} we get,
\begin{align} 
\abs{E^0_{LH}} & \lesssim \nu\sum_{M \in 2^\Integer} \norm{\partial_v A (f_0)_{\sim M}}_2\norm{\partial_v A (f_0)_M}_2 \norm{\left(1 - (v^\prime)^2\right)_{<M/8}}_{\G^{\lambda,\sigma}} \nonumber \\ 
& \lesssim  \nu \epsilon \norm{\partial_v A f_0}_2^2, \label{ineq:E0LH}
\end{align} 
which is again absorbed by the dissipation term in \eqref{eq:HiNrmDiss} for $\epsilon$ sufficiently small.   

Finally, consider the remainder term and divide into separate cases based on the output frequency:  
\begin{align*} 
E^0_{\mathcal{R}} & = \nu \sum_{M \in 2^\Integer} \sum_{M/8 \leq M^\prime \leq 8M} \int \left(A f_0 \right)_{\leq 1} A\left(G_{M} (\partial_{vv}f_0)_{M^\prime} \right) dv  \\ 
& \quad + \sum_{M \in 2^\Integer} \sum_{M/8 \leq M^\prime \leq 8M} \int \left(A f_0 \right)_{> 1} A\left( G_{M} (\partial_{vv}f_0)_{M^\prime} \right) dv  \\ 
& = E^{0;L}_{\mathcal{R}} +  E_{\mathcal{R}}^{0;H}.   
\end{align*} 
The $E_{\mathcal{R}}^{0;H}$ term is relatively straightforward. Indeed, on the support of the integrand we may apply \eqref{lem:strivial} to deduce for some $c \in (0,1)$ that, 
\begin{align*} 
\abs{E_{\mathcal{R}}^{0;H}} & \lesssim \sum_{M \in 2^\Integer} \sum_{M/8 \leq M^\prime \leq 8M} \int_{\eta,\xi} \mathbf{1}_{\abs{\eta} > 1}\abs{A \widehat{f_0}(\eta) J_0(\eta)e^{\lambda\abs{\eta}^s} \jap{\eta-\xi}^{\sigma/2} \widehat{G}(\eta-\xi)_M  \jap{\xi}^{\sigma/2+1}  \widehat{\partial_v f_0}(\xi)_{M^\prime}} d\xi d\eta \\ 
& \lesssim \sum_{M \in 2^\Integer} \sum_{M/8 \leq M^\prime \leq 8M} \int_{\eta,\xi}  \mathbf{1}_{\abs{\eta} > 1}\abs{A \widehat{f_0}(\eta) J_0(\eta)e^{c \lambda\abs{\eta-\xi}^s} \jap{\eta-\xi}^{\sigma/2} \widehat{G}(\eta-\xi)_M} \\ & \quad\quad \times  \abs{\jap{\xi}^{\sigma/2+1} e^{c \lambda\abs{\xi}^s} \widehat{\partial_v f_0}(\xi)_{M^\prime}} d\xi d\eta. 
\end{align*} 
By Lemma \ref{basic} followed by \eqref{ineq:IncExp} and \eqref{ineq:SobExp} we have
\begin{align*} 
\abs{E_{\mathcal{R}}^{0;H}} & \lesssim \nu\sum_{M \in 2^\Integer} \sum_{M/8 \leq M^\prime \leq 8M} \int_{\eta,\xi}  \mathbf{1}_{\abs{\eta} > 1}\abs{A \widehat{f_0}(\eta) e^{\lambda\abs{\eta-\xi}^s} \widehat{G}(\eta-\xi)_M e^{\lambda\abs{\xi}^s} \widehat{\partial_v f_0}(\xi)_{M^\prime}} d\xi d\eta \\ 
& \lesssim \nu\sum_{M \in 2^\Integer} \sum_{M/8 \leq M^\prime \leq 8M} \int_{\eta,\xi}  \mathbf{1}_{\abs{\eta} > 1}\abs{\eta A \widehat{f_0}(\eta) e^{\lambda\abs{\eta-\xi}^s} \widehat{G}(\eta-\xi)_M e^{\lambda\abs{\xi}^s} \widehat{\partial_v f_0}(\xi)_{M^\prime}} d\xi d\eta, 
\end{align*}  
where in the last line we used $\abs{\eta} > 1$ on the support of the integrand. 
Therefore, by \eqref{ineq:L2L2L1} followed by Cauchy-Schwarz in $M$ and \eqref{ineq:coefbds} (recall \eqref{def:Gshort}), 
\begin{align} 
\abs{E_{\mathcal{R}}^{0;H}} & \lesssim \nu\sum_{M \in 2^\Integer} \norm{\partial_v A f_0}_2 \norm{\left(1-(v^\prime)^2\right)_M}_{\G^{\lambda,2}}\norm{\partial_v A (f_0)_{\sim M}}_{\G^{\lambda,0}} \nonumber \\ 
& \lesssim \epsilon \nu \norm{\partial_v Af_0}_2^2, \label{ineq:E0RH}
\end{align}
which is then absorbed by the dissipation in \eqref{eq:HiNrmDiss} for $\epsilon$ sufficiently small.  

Next we treat $E_{\R}^{0;L}$, which requires more care than $E_\R^{0;H}$. 
 By Cauchy-Schwarz and Bernstein's inequalities, 
\begin{align*} 
\abs{E_{\mathcal{R}}^{0;L}} & \leq \nu \sum_{M \in 2^\Integers} \sum_{M^\prime \approx M} \norm{P_{\leq 1}A f_0}_{2}\norm{P_{\leq 1} A\left(G_{M^\prime} (\partial_{vv}f_0)_{M}\right)}_{2} \\ 
 & \lesssim \nu \sum_{M \in 2^\Integers} \sum_{M^\prime \approx M} \norm{f_0}_{2}\norm{G_{M^\prime} (\partial_{vv}f_0)_{M}}_{2} \\ 
 & \lesssim \nu \sum_{M \in 2^\Integers} \sum_{M^\prime \approx M} \norm{f_0}_{2} \norm{\left(1-(v^\prime)^2\right)_{M^\prime}}_{\infty}\norm{(\partial_{vv}f_0)_{M}}_{2} \\ 
 & \lesssim \nu \sum_{M \in 2^\Integers} \norm{f_0}_{2} M^{3/2}  \norm{\left(1-(v^\prime)^2\right)_{\sim M}}_2 \norm{(\partial_{v}f_0)_{M}}_{2} \\
 & \lesssim \nu \sum_{M \in 2^\Integers : M \leq 1} \norm{f_0}_{2} M^{1/2}  \norm{\partial_v\left(1-(v^\prime)^2\right)_{\sim M}}_{2} \norm{(\partial_{v}f_0)_{M}}_{2} 
\\ & \quad + \nu \sum_{M \in 2^\Integers: M > 1} \norm{f_0}_{2} M^{1/2}  \norm{\partial_v\left(1-(v^\prime)^2\right)_{\sim M}}_{2} \norm{(\partial_{v}f_0)_{M}}_{2}. 
\end{align*}
The first (low frequency) term is summed by Cauchy-Schwarz in $M$ 
whereas the second term is summed by paying additional derivatives on the last factor to reduce the power of $M$. 
Therefore by \eqref{ineq:vppdecay} and \eqref{ineq:Boot_energyHi}, 
\begin{align} 
\abs{E_{\mathcal{R}}^{0;L}} & \lesssim \nu\norm{f_0}_2\norm{\partial_v\left(1 - (v^\prime)^2\right)}_2 \norm{\partial_v f_0}_2 + \nu\norm{f_0}_2\norm{\partial_v\left(1 - (v^\prime)^2\right)}_2 \norm{\partial_v A f_0}_2 \nonumber \\ 
& \lesssim \nu\norm{f_0}_2 \norm{\partial_v A f_0}_2^2 + \nu\norm{f_0}_2 \norm{\partial_v\left(1 - (v^\prime)^2\right)}_2^2 \nonumber \\ 
& \lesssim \epsilon \nu \norm{\partial_v A f_0}_2^2 + \epsilon^3 \frac{\nu}{\jap{\nu t}^{3/2}}. \label{ineq:ER0L}
\end{align}
For $\epsilon$ small, the first term is absorbed by the dissipation term in \eqref{eq:HiNrmDiss} while the
latter is integrable in time uniformly in $\nu$ and cubic in $\epsilon$.  

Putting together all the contributions from \eqref{eq:AfEvo}: the Euler nonlinear bound \eqref{ineq:EulerHiNrm}, the dissipation terms \eqref{eq:HiNrmDiss}, the non-zero frequency dissipation error terms \eqref{ineq:EneqLH}, \eqref{ineq:EneqHLz}, \eqref{ineq:T1Hard}, \eqref{ineq:EHLneqAst}, \eqref{ineq:EneqRz} and \eqref{ineq:EneqRv} with the zero frequency dissipation error terms \eqref{ineq:E0LH}, \eqref{ineq:E0HL}, \eqref{ineq:E0RH} and \eqref{ineq:ER0L} and integrating in time gives the bound \eqref{ineq:energyHi} for $\epsilon$ chosen sufficiently small.

\section{Enhanced dissipation estimate \eqref{ineq:energyLo}} \label{sec:ED}
Up to an adjustment of the constants in the bootstrap argument, it suffices to consider only $t$ such that $\nu t^3 \geq 1$ (say), as otherwise the decay estimate \eqref{ineq:energyLo} follows trivially from the inviscid energy estimate \eqref{ineq:energyHi}. 

Computing the time evolution of $\norm{A^\nu f}_2$ from \eqref{def:vortNice}, 
\begin{align} 
\frac{1}{2}\frac{d}{dt} \norm{A^\nu f}_2^2 & = -CK_\lambda^\nu + \alpha\int A^\nu f e^{\lambda(t)\abs{\grad}^s}\jap{\grad}^\beta\jap{D(t,\partial_v)}^{\alpha-1} \frac{D(t,\partial_v)}{\jap{D(t,\partial_v)}} \partial_t D(t,\partial_v) P_{\neq 0}f dv dz   \nonumber \\ 
& \quad - \int A^\nu f A^\nu \left( u\cdot \grad f\right) dv dz + \nu\int A^\nu f A^{\nu} \left(\tilde{\Delta_t} f\right) dv dz \nonumber \\ 
& \leq -CK_\lambda^\nu + \frac{1}{8}\nu t^2 \norm{\mathbf{1}_{t \geq 2 \abs{\partial_v}} A^\nu P_{\neq 0} f}_2^2 \nonumber    \\ 
& \quad - \int A^\nu f A^\nu \left( u\cdot \grad f\right) dv dz + \nu\int A^\nu f A^{\nu} \left(\tilde{\Delta_t} f\right) dv dz. \label{ineq:AnuEvo} 
\end{align} 
As in \eqref{eq:HiNrmDiss} above, we write the dissipation term as a perturbation of $\Delta_L$, 
\begin{align} 
\nu \int A^\nu f A^{\nu} \Delta_t f dv dz & = -\nu\norm{\sqrt{-\Delta_L}A^\nu f}_2^2 - \sum_{k \neq 0}\int A^\nu f_k A^\nu\left( (1-(v^\prime)^2) (\partial_{v} - t\partial_z)^2 f_k \right) dv dz \nonumber \\
& = -\nu\norm{\sqrt{-\Delta_L}A^\nu f}_2^2 + E^{\nu}. \label{ineq:LoNrmDiss}
\end{align} 
First, we need to cancel the growing term in \eqref{ineq:AnuEvo} that involves $\nu t^2$ with part of the leading order dissipation in \eqref{ineq:LoNrmDiss}.   
Indeed, we have
\begin{align*}
\frac{1}{8} \nu t^2 \norm{\mathbf{1}_{t \geq 2 \abs{\partial_v}}A^\nu P_{\neq 0}f}_2^2-\nu\norm{\sqrt{-\Delta_L}A^\nu P_{\neq 0} f}_2^2 & \\
& \hspace{-6cm}  = \nu \sum_{k \neq 0}\int \left(\frac{1}{8}t^2\mathbf{1}_{t \geq 2 \abs{\eta}} - \abs{k}^2 - \abs{\eta-kt}^2\right)   \abs{A^\nu \hat{f}_k(\eta)}^2 d\eta \\ 
& \hspace{-6cm} \leq -\frac{\nu}{8}\norm{\sqrt{-\Delta_L}A^\nu P_{\neq 0} f}_2^2, 
\end{align*} 
which implies from \eqref{ineq:AnuEvo},  
\begin{align} 
\frac{1}{2}\frac{d}{dt} \norm{A^\nu f}_2^2 & \leq -CK_\lambda^\nu - \int A^\nu f A^\nu \left( u\cdot \grad f\right) dv dz -\frac{\nu}{8}\norm{\sqrt{-\Delta_L}A^\nu f}_2^2 + E^{\nu}. \label{ineq:EDdiss} 
\end{align} 
There are two challenges here: the Euler nonlinearity (the second term) and the error from the dissipation (the final term).

\subsection{Euler nonlinearity}
We first divide into zero and non-zero frequency contributions, as they will be treated differently: 
\begin{align*} 
-\int A^\nu f A^\nu \left( u\cdot \grad f\right) dv dz & = -\int A^\nu f A^\nu\left(g \partial_v f\right) dv dz - \int A^\nu f A^\nu\left(v^\prime \grad^{\perp}P_{\neq 0} \phi \cdot \grad f\right) dv dz \\ 
& = \mathcal{E}_1 + \mathcal{E}_2.
\end{align*} 
The reason for this is the large disparity that $A^\nu$ imposes between the zero-in-$z$ mode and the non-zero-in-$z$ modes, 
which are mixed in $\mathcal{E}_2$ due to the $z$ dependence of the velocity field $\grad^\perp P_{\neq 0} \phi$. \
Indeed, the commutator trick that is used to recover part of the derivative in the treatment of the Euler nonlinearity in \cite{LevermoreOliver97,BM13} normally requires that the norm not vary drastically in $k$ (the discrete, $z$ wavenumber).
During the proof of \eqref{ineq:EulerHiNrm} in \S5 of \cite{BM13}, the disparity introduced by $J$ is recovered using that it only occurs near the critical times, and hence decay of the velocity field can be transferred back to regularity. 
This will not work here; instead we will make use of the high norm estimate \eqref{ineq:energyHi} to recover the derivative and take advantage of the rapid decay of the velocity field given by \eqref{ineq:LossyEllip2}. 

First consider the estimation of $\mathcal{E}_1$, which as various techniques in e.g. \cite{LevermoreOliver97,KukavicaVicol09,BM13}, begins with the commutator trick (note that $P_{= 0}g = g$ and $A^\nu P_{\neq 0} = A^{\nu}$): 
\begin{align} 
 - \int A^\nu f_k A^\nu \left( g\partial_v f \right) dv dz  = \frac{1}{2} \int \partial_v g \abs{A^\nu f}^2 dv dz + \int A^\nu f \left[g \partial_v A^\nu f - A^\nu \left( g \partial_v f\right) \right]dv dz. \label{eq:commg}
\end{align}
The first term in \eqref{eq:commg} is controlled by Sobolev embedding and the decay from \eqref{ineq:Boot_gc},
\begin{align} 
\frac{1}{2}\int \partial_v g \abs{A^\nu f}^2 dv dz \lesssim \norm{g}_{\infty} \norm{A^\nu f}_2^2 \lesssim  \frac{\epsilon}{\jap{t}^{2-K_D\epsilon/2}} \norm{A^\nu f}_2^2. \label{ineq:AnuLowerComm}
\end{align}
The latter term in \eqref{eq:commg} is expanded with a paraproduct (in both $z$ and $v$): 
\begin{align} 
\int A^\nu f \left[g \partial_v A^\nu f - A^\nu \left( g \partial_v f\right) \right]dv dz & = \sum_{N \geq 8}\int A^\nu f \left[g_{<N/8} \partial_v A^\nu f_N - A^\nu \left( g_{<N/8} \partial_v f_{N}\right) \right]dv dz  \nonumber \\ 
& \quad + \sum_{N \geq 8}\int A^\nu f \left[g_{N} \partial_v A^\nu f_{<N/8} - A^\nu \left( g_{N} \partial_v f_{<N/8}\right) \right]dv dz \nonumber  \\ 
& \quad + \sum_{N \in \mathbb D}\sum_{N/8 \leq N^\prime \leq 8N} \int A^\nu f \left[g_{N^\prime} \partial_v A^\nu f_{N} - A^\nu \left( g_{N^\prime} \partial_v f_{N}\right) \right]dv dz \nonumber \\ 
& =  \sum_{N \geq 8} T^0_N + \sum_{N \geq 8} R^0_N + \R^0.  \label{eq:Anu0EulerPara}
\end{align}
On the Fourier side, 
\begin{align*} 
T_N^0 & = -\frac{i}{2\pi}\sum_{k \neq 0}\int_{\eta,\xi}A^\nu \overline{\hat{f}}_k(\eta)\left[A^\nu_k(\eta) - A^\nu_k(\xi)\right]\widehat{g}(\eta-\xi)_{<N/8} \xi \hat{f}_k(\xi)_N d\eta d\xi, 
\end{align*}
and on the support of the integrand there holds 
\begin{subequations} \label{ineq:TN0}
\begin{align} 
\abs{\abs{k,\eta} - \abs{k,\xi}} & \leq \abs{\eta-\xi} \leq \frac{3}{16}\abs{k,\xi}, \\
\frac{13}{16}\abs{k,\xi} &\leq \abs{k,\eta} \leq \frac{19}{16}\abs{k,\xi}. 
\end{align} 
\end{subequations}
For the commutator, we write, 
\begin{align} 
A^\nu_k(\eta) - A^\nu_k(\xi) & = \jap{D(t,\eta)}^\alpha\left[e^{\lambda\abs{k,\eta}^s}\jap{k,\eta}^{\beta} - e^{\lambda\abs{k,\xi}^s}\jap{k,\xi}^{\beta}\right] \nonumber \\ & \quad + e^{\lambda\abs{k,\xi}^s}\jap{k,\xi}^{\beta}\left[\jap{D(t,\eta)}^\alpha - \jap{D(t,\xi)}^\alpha\right]. \label{eq:AnuComm}
\end{align}
Then, 
\begin{align*} 
T^0_N & = -i\sum_{k \neq 0}\int_{\eta,\xi}A^\nu \overline{\hat{f}}_k(\eta)\jap{D(\eta)}^{\alpha}\left[e^{\lambda\abs{k,\eta}^s}\jap{k,\eta}^{\beta} - e^{\lambda\abs{k,\xi}^s}\jap{k,\xi}^{\beta}\right]\widehat{g}(\eta-\xi)_{<N/8} \xi\hat{f}_k(\xi)_N d\eta d\xi \\ 
& \quad - i\sum_{k \neq 0}\int_{\eta,\xi}A^\nu \overline{\hat{f}}_k(\eta) e^{\lambda\abs{k,\xi}^s}\jap{k,\xi}^{\beta}\left[\jap{D(t,\eta)}^\alpha - \jap{D(t,\xi)}^\alpha\right]\widehat{g}(\eta-\xi)_{<N/8} \xi\hat{f}_k(\xi)_N d\eta d\xi \\ 
& = T_N^{0;1} + T_N^{0;2}. 
\end{align*} 
Consider the more standard $T_N^{0;1}$ first. 
We claim that by \eqref{ineq:TN0}, on the support of the integrand there holds, for some $c \in (0,1)$ (depending only on $s$ and our Littlewood-Paley conventions),
\begin{align} 
 \abs{e^{\lambda\abs{k,\eta}^s}\jap{k,\eta}^{\beta} - e^{\lambda\abs{k,\xi}^s}\jap{k,\xi}^{\beta}} \lesssim \frac{\abs{\eta-\xi}}{\jap{\xi}^{1-s}} e^{c\lambda\abs{\eta-\xi}^s} e^{\lambda\abs{k,\xi}^s}\jap{k,\xi}^{\beta}; \label{ineq:transtriv}
\end{align} 
see e.g. \S5 of \cite{BM13} or \cite{LevermoreOliver97} combined with \eqref{lem:scon}. 
Applying \eqref{ineq:transtriv}, \eqref{ineq:TN0} and \eqref{ineq:DRat} implies
\begin{align*}
\abs{T_N^{0;1}} &  \lesssim \sum_{k \neq 0}\int_{\eta,\xi}\abs{A^\nu \hat{f}_k(\eta) \jap{\eta-\xi}^{3\alpha + 1}e^{c\lambda\abs{\eta-\xi}^s}\widehat{g}(\eta-\xi)_{<N/8} \frac{\xi}{\jap{\xi}^{1-s}} A^\nu \hat{f}_k(\xi)_N} d\eta d\xi \\ 
& \lesssim \sum_{k \neq 0}\int_{\eta,\xi}\abs{ \abs{k,\eta}^{s/2} A^\nu \hat{f}_k(\eta) \jap{\eta-\xi}^{3\alpha + 1}e^{c\lambda\abs{\eta-\xi}^s}\widehat{g}(\eta-\xi)_{<N/8} \abs{k,\xi}^{s/2} A^\nu \hat{f}_k(\xi)_N} d\eta d\xi. 
\end{align*}  
Therefore, by $\sigma > 3\alpha + 8$ and \eqref{ineq:L2L2L1} followed by \eqref{ineq:Boot_gc}, we get 
\begin{align} 
T_N^{0;1} & \lesssim \norm{g_{<N/8}}_{\G^{\lambda(t),\sigma-6}} \norm{ \abs{\grad}^{s/2} A^\nu f_{\sim N}}_2^2 \lesssim \frac{\epsilon}{\jap{t}^{2-K_D\epsilon}}\norm{ \abs{\grad}^{s/2} A^\nu f_{\sim N}}_2^2. \label{ineq:TN01}
\end{align} 
After summing in $N$ and choosing $\epsilon$ small, this term is absorbed by $CK_\lambda^\nu$. 

For $T_N^{0;2}$ we crucially use that \eqref{ineq:Dcomm} applies; this is due to the fact that $g$ does not depend on $z$ and will not work to treat $\mathcal{E}_2$. Indeed, \eqref{ineq:Dcomm} implies
\begin{align*} 
T_N^{0;2} & \lesssim \sum_{k \neq 0}\int_{\eta,\xi} \abs{A^\nu \hat{f}_k(\eta) e^{\lambda\abs{k,\xi}^s}\jap{k,\xi}^{\beta}\jap{\eta-\xi}^{3\alpha} \widehat{g}(\eta-\xi)_{<N/8}\jap{D(t,\xi)}^{\alpha} \hat{f}_k(\xi)_N} d\eta d\xi.
\end{align*} 
By $\sigma > 3\alpha+8$ and \eqref{ineq:L2L2L1} followed by \eqref{ineq:Boot_gc} implies 
\begin{align} 
T_N^{0;2} & \lesssim \norm{g}_{\G^{\lambda(t),\sigma-6}}\norm{A^\nu f_{\sim N}}_2^2  \lesssim \frac{\epsilon}{\jap{t}^{2-K_D\epsilon}}\norm{A^\nu f_{\sim N}}_2^2 , \label{ineq:TN02}
\end{align} 
which is an integrable contribution. 

The `reaction' term $R_N^0$ is dealt with easily using Lemma \ref{lem:AnuProd} and the bootstrap controls on the higher norms.
First note that $A^\nu(g \partial_v f) =A^\nu(g \partial_v P_{\neq 0}f)$. 
Then, by Cauchy-Schwarz, \eqref{ineq:Ddistri}, $\sigma \geq \beta + 3\alpha + 7$, Cauchy-Schwarz in $N$ and \eqref{ineq:Boot_gc},
\begin{align}
\abs{\sum_{N \geq 8}\int A^\nu f A^\nu\left(g_{N} \partial_v P_{\neq 0} f_{<N/8}\right) dv dz} & \leq \sum_{N \geq 8} \norm{A^\nu f_{\sim N}}_2 \norm{A^\nu\left(g_{N} \partial_vP_{\neq 0}f_{<N/8}\right)}_2 \nonumber \\ 
& \lesssim \sum_{N \geq 8} \norm{A^\nu f_{\sim N}}_2 \norm{g_{N}}_{\G^{\lambda,\beta+ 3\alpha}} \norm{A^\nu \partial_vP_{\neq 0} f_{<N/8}}_2 \nonumber \\ 
& \lesssim \sum_{N \geq 8} \norm{A^\nu f_{\sim N}}_2 N \norm{g_{N}}_{\G^{\lambda,\beta+ 3\alpha}} \norm{A^\nu P_{\neq 0} f_{<N/8}}_2 \nonumber \\ 
& \lesssim \sum_{N \geq 8} \norm{A^\nu f_{\sim N}}_2 \norm{g_{N}}_{\G^{\lambda,\beta+ 3\alpha+1}} \norm{A^\nu P_{\neq 0} f_{<N/8}}_2 \nonumber \\ 
& \lesssim \norm{g}_{\G^{\lambda,\sigma-6}} \norm{A^\nu f}^2_2 \nonumber \\ 
& \lesssim \frac{\epsilon}{\jap{t}^{2-K_D\epsilon/2}} \norm{A^\nu f}_2^2, \label{ineq:RNneq_a}
\end{align}
which is an integrable contribution. 
Similarly, by Bernstein's inequalities, Cauchy-Schwarz in $N$ and \eqref{ineq:Boot_gc} we have
\begin{align}
\abs{\sum_{N \geq 8}\int A^\nu f g_{N} \partial_v A^\nu P_{\neq 0} f_{<N/8} dv dz} & \leq \sum_{N \geq 8} \norm{A^\nu f_{\sim N}}_2 \norm{g_N}_{\infty} \norm{\partial_v A^\nu P_{\neq 0} f_{<N/8}}_2 \nonumber \\ 
& \lesssim \sum_{N \geq 8} \norm{A^\nu f_{\sim N}}_2 N^{3/2} \norm{g_N}_{2} \norm{A^\nu f_{<N/8}}_2 \nonumber \\ 
& \lesssim \norm{g}_{H^{3/2}}\norm{A^\nu f}_2^2 \nonumber \\ 
& \lesssim \frac{\epsilon}{\jap{t}^{2-K_D\epsilon/2}} \norm{A^\nu f}_2^2, \label{ineq:RNneq_b}
\end{align}
which again is integrable.
The remainder terms $\R^0$ in \eqref{eq:Anu0EulerPara} are treated very similar to the reaction terms $R_N^0$ just completed; 
hence we omit the details and simply conclude 
\begin{align} 
\abs{\R^0} & \lesssim \frac{\epsilon}{\jap{t}^{2-K_D\epsilon/2}} \norm{A^\nu f}_2^2. \label{ineq:AnuRemg}
\end{align}

Next turn to $\mathcal{E}_2$.
If the zero mode in $z$ interacts with non-zero modes, possible now due to the $z$ dependence of the velocity field, then we can no longer apply \eqref{ineq:Dcomm} to gain regularity. 
Physically, the issue is that gradients in the large zero frequency can be converted to large gradients in non-zero frequencies by transport.
By Cauchy-Schwarz and two applications of \eqref{ineq:Ddistri}, 
\begin{align*} 
\mathcal{E}_2 & \leq \norm{A^\nu f}_2 \norm{A^\nu \left( v^\prime \grad^\perp P_{\neq 0} \phi\right)}_2 \norm{\grad f}_{\G^{\lambda,\beta + 3\alpha}} \\ 
& \lesssim \norm{A^\nu f}_2 \norm{f}_{\G^{\lambda,\beta + 3\alpha+1}}\left(\norm{A^\nu \grad^\perp P_{\neq 0} \phi}_2 + \norm{A^\nu\left( (v^\prime - 1) \grad^\perp P_{\neq 0} \phi \right) }_2 \right) \\ 
& \lesssim \norm{A^\nu f}_2 \norm{f}_{\G^{\lambda,\beta + 3\alpha+1}} \left(1 + \norm{v^\prime - 1}_{\G^{\lambda,\beta + 3\alpha}}\right)\norm{A^\nu \grad^\perp P_{\neq 0} \phi}_2.   
\end{align*} 
Then by $\sigma \geq \beta + 3\alpha + 1$, Lemma \ref{lem:LossyEllip2}, \eqref{ineq:Boot_energyHi}, \eqref{ineq:Boot_energyLo} and \eqref{ineq:Boot_ARh}, 
\begin{align} 
\mathcal{E}_2 & \lesssim \norm{A^\nu f}_2\norm{Af}_2 \left(1 + \norm{A(v^\prime - 1)}_2\right)\frac{\norm{Af}_2 + \norm{A^\nu f}_2}{\jap{t}^2}  \lesssim \frac{\epsilon^3}{\jap{t}^2}, \label{ineq:E2}
\end{align} 
an integrable, cubic in $\epsilon$ contribution and hence this completes the treatment of $\mathcal{E}_2$ 
and of the entire Euler nonlinearity. 

\subsection{Dissipation error term} \label{sec:NZModeDissLo}
Since at high frequencies $A^\nu$ defines a weaker norm than $A$, it will be possible to apply \eqref{ineq:triTriv} 
and the stronger control coming from \eqref{ineq:Boot_ARh}. 
Indeed, writing $E^{\nu}$ on the Fourier side and applying \eqref{ineq:triTriv} implies
\begin{align*}
\abs{E^{\nu}} & \lesssim \nu\sum_{k \neq 0}\int_{\eta,\xi} \abs{A^\nu \widehat{f}_{k}(\eta) A^\nu_k(\eta) \widehat{G}(\eta-\xi) \abs{\xi-tk}^2 \hat{f}_k(\xi)} d\eta d\xi \\ 
& \lesssim \nu\sum_{k \neq 0}\int_{\eta,\xi} \abs{\sqrt{-\Delta_L} A^\nu \widehat{f}_{k}(\eta) A^\nu_k(\eta) \jap{\eta-\xi} \widehat{G}(\eta-\xi) \abs{\xi-tk} \hat{f}_k(\xi)} d\eta d\xi. 
\end{align*}
By \eqref{lem:smoretrivial}, \eqref{ineq:DRat} and $\jap{\eta}^\beta \lesssim \jap{\eta-\xi}^\beta \jap{\xi}^\beta$ we have
\begin{align*} 
\abs{E^{\nu}} & \lesssim \nu\sum_{k \neq 0}\int_{\eta,\xi} \abs{\sqrt{-\Delta_L} A^\nu \widehat{f}_{k}(\eta) \jap{\eta-\xi}^{\beta+3\alpha+1} e^{\lambda\abs{\eta-\xi}^s} \widehat{G}(\eta-\xi) \abs{\xi-tk} A^\nu \hat{f}_k(\xi)} d\eta d\xi. 
\end{align*}
Therefore, by \eqref{ineq:L2L2L1}, $\sigma > \beta + 3\alpha + 2$ and $k \neq 0$ followed by \eqref{ineq:coefbds} we get
\begin{align} 
\abs{E^{\nu}} & \lesssim \nu\sum_{k \neq 0} \norm{1 - (v^\prime)^2}_{\G^{\lambda,\sigma}} \norm{\sqrt{-\Delta_L}A^\nu f_k}_2^2 \nonumber \\ 
& \lesssim \epsilon \nu \norm{\sqrt{-\Delta_L}A^\nu f}_2^2, \label{ineq:EnuNeq}
\end{align} 
which is absorbed by the leading order dissipation in \eqref{ineq:EDdiss} for $\epsilon$ small. 
 
Putting together the estimates on \eqref{ineq:EDdiss} from the Euler nonlinearity (\eqref{ineq:AnuLowerComm}, \eqref{ineq:TN01}, \eqref{ineq:TN02}, \eqref{ineq:RNneq_a}, \eqref{ineq:RNneq_b}, \eqref{ineq:AnuRemg} and \eqref{ineq:E2}) and from the dissipation error terms \eqref{ineq:EnuNeq}, we deduce \eqref{ineq:energyLo} after integrating in time and choosing $\epsilon$ sufficiently small.  

\section{Coordinate system higher regularity controls}\label{sec:CoordHiReg}
In \S8 of \cite{BM13}, there are three main estimates on the coordinate system that need to be made, corresponding here to \eqref{ineq:bhc}, \eqref{ineq:gc} and \eqref{ineq:ARh}. 
That \eqref{ineq:CK1} can be deduced from the proof of \eqref{ineq:bhc}, \eqref{ineq:gc} and \eqref{ineq:ARh} is shown in 
\S8 of \cite{BM13} and is not repeated here as there is little difference.   

The main new issue for the Navier-Stokes case is to confirm that the variable coefficients in $\tilde{\Delta_t}$ do not slow down or otherwise impede the decay estimates.
This could be possible, for example, if the diffusion coefficient oscillated too much relative to the strength of the damping terms, as gradients in the diffusion coefficient induce momentum transport.    
This effect is controlled by the gradient decay estimate in \eqref{ineq:hdecay}.  

\subsection{Proof of \eqref{ineq:gc}} \label{sec:gc}
Computing from \eqref{def:PDEg} implies 
\begin{align} 
\frac{d}{dt}\left(\jap{t}^{4-K_D\epsilon}\norm{g}_{\G^{\lambda(t),\sigma-6}}^2\right) & = (4-K_D\epsilon)t \jap{t}^{2-K_D\epsilon}\norm{g}_{\G^{\lambda(t),\sigma-6}}^2 \nonumber \\ & \quad\quad + \jap{t}^{4-K_D\epsilon}\frac{d}{dt}\norm{\frac{A}{\jap{\partial_v}^s}g}_{\G^{\lambda(t),\sigma-6}}^2.  \label{ineq:ddtpartt}
\end{align}
Denoting the multiplier $A^S(t,\partial_v) = e^{\lambda(t)\abs{\partial_v}^s}\jap{\partial_v}^{\sigma-6}$ (`S' for `simple'), the latter term gives
\begin{align} 
\jap{t}^{4-K_D\epsilon}\frac{d}{dt}\norm{g}_{\G^{\lambda(t),\sigma-6}}^2 & = 2\jap{t}^{4-K_D\epsilon} \dot\lambda(t)\norm{\abs{\partial_v}^{s/2}g}_{\G^{\lambda(t),\sigma-6}}^2 +  
  2\jap{t}^{4-K_D\epsilon}\int A^Sg A^S \partial_tg dv, \label{ineq:timederivpartialt}
\end{align} 
From \eqref{def:PDEg}, 
\begin{align} 
 2\jap{t}^{4-K_D\epsilon}\int A^S g A^S \partial_tg dv & = -\frac{4\jap{t}^{4-K_D\epsilon}}{t}\norm{g}_{\G^{\lambda(t),\sigma-6}}^2 \nonumber \\ 
& \quad - 2\jap{t}^{4-K_D\epsilon} \int A^S g A^S \left(g \partial_v g\right) dv 
 \nonumber \\ & \quad - \frac{2\jap{t}^{4-K_D\epsilon}}{t}\int A^Sg A^S\left(v^\prime < \grad^\perp P_{\neq 0}\phi\cdot \grad \tilde u > \right) dv \nonumber \\ 
& \quad + 2\nu \jap{t}^{4-K_D\epsilon}\int A^S g A^S\left (\tilde{\Delta_t} g\right) dv \nonumber \\ 
& = V_1 + V_2 + V_3 + V_D. \label{def:V1V2V3}
\end{align}
The first three terms are basically the same as in Euler, and are treated accordingly as in \cite{BM13}. 
Indeed, in \S8 of \cite{BM13} it is shown that the bootstrap hypotheses together with \eqref{ineq:GAlg} and \eqref{eq:barhgRelat}  imply 
\begin{align}
V_2 & \leq \frac{K_D\epsilon}{8} \jap{t}^{3-K_D\epsilon-s}  \norm{g}_{\G^{\lambda,\sigma-6}}^2, \label{ineq:V2}
\end{align} 
where we define $K_D$ to be the maximum of the constant appearing in this term and several other below. 

Treating $V_3$ is not hard due to the regularity gap of $6$ derivatives. 
Note that
\begin{align} 
\grad \tilde u  = 
-\begin{pmatrix} 
v^\prime (\partial_v - t\partial_z)\partial_z\phi \\ 
\partial_vv^\prime (\partial_v - t\partial_z)\phi + v^\prime(\partial_v - t\partial_z)\partial_v \phi 
\end{pmatrix}
, \label{def:gradtu}
\end{align}
and therefore by \eqref{ineq:GAlg}, \eqref{def:gradtu}, Lemma \ref{lem:LossyElliptic} and the bootstrap hypotheses, 
\begin{align} 
\norm{\grad P_{\neq 0} \tilde u(t)}_{\G^{\lambda(t),\sigma-5}} \lesssim \frac{\epsilon}{\jap{t}}. \label{ineq:gradulossy}
\end{align}
It is shown in \S8 of \cite{BM13} that \eqref{ineq:gradulossy}, together with \eqref{ineq:GAlg}, Lemma \ref{lem:LossyElliptic} and the bootstrap hypotheses, implies 
we have for some $C > 0$, 
\begin{align} 
V_3 & \lesssim \jap{t}^{-K_D\epsilon} \epsilon^2 \norm{g}_{\G^{\lambda,\sigma-6}} \leq \frac{K_D\epsilon \jap{t}^{4-K_D\epsilon}}{8t} \norm{g}_{\G^{\lambda,\sigma-6}}^2 + C\epsilon^3 t^{-3 -K_D\epsilon}. \label{ineq:V3}
\end{align}

Focus now on the term that is new for Navier-Stokes, $V_D$. 
As in many other estimates in this work, we write this as two contributions: 
\begin{align}
V_D & = -2 \nu \jap{t}^{4-K_D\epsilon} \norm{\partial_v A^S g}_2^2 -2 \nu \jap{t}^{4-K_D\epsilon} \int A^S g A^S\left(1 - (v^\prime)^2\right)\partial_{vv}g dv \nonumber \\ & = -2 \nu \jap{t}^{4-K_D\epsilon} \norm{\partial_v A^S g}_2^2 - V_{DE}, \label{def:VDE}
\end{align}
a leading order dissipation term and an error term which remains to be controlled. 
We use again the shorthand \eqref{def:Gshort} and decompose the error term with a homogeneous paraproduct 
\begin{align*} 
V_{DE} & = 2\nu \sum_{M \in 2^\Integers}\jap{t}^{4-K_D\epsilon} \int A^S g A^S\left(G_{M} \partial_{vv}g_{<M/8}\right) dv \\ 
& \quad + 2\nu \sum_{M \in 2^\Integers}\jap{t}^{4-K_D\epsilon} \int A^S g A^S\left(G_{<M/8} \partial_{vv}g_{M}\right) dv \\
& \quad + 2\nu \sum_{M \in 2^\Integers}\jap{t}^{4-K_D\epsilon} \sum_{M/8 \leq M^\prime \leq 8M} \int A^S g A^S\left(G_{M^\prime} \partial_{vv}g_{M}\right) dv \\ 
& = V_{DE,HL} + V_{DE,LH} + V_{DE,\mathcal{R}}. 
\end{align*} 
As in the zero frequency dissipation error terms in \S\ref{sec:ZeroModeDiss}, the remainder term is the most delicate due to the contributions of low frequencies. 

First we deal with $V_{DE,LH}$, which on the Fourier side is written as
\begin{align*} 
V_{DE,LH} & \lesssim \nu \sum_{M \in 2^\Integers}\jap{t}^{4-K_D\epsilon} \int_{\eta,\xi} \abs{A^S \hat{g}(\eta) A^S(\eta) \widehat{G}(\eta-\xi)_{<M/8} \abs{\xi}^2 \hat{g}(\xi)_M } d\eta d\xi. 
\end{align*} 
On the support of the integrand, \eqref{ineq:E0HLFreqLoc} holds with the role of $\xi$ and $\eta-\xi$ reversed.
Hence, by \eqref{lem:scon}, for some $c \in (0,1)$:  
\begin{align*} 
V_{DE,LH} & \lesssim \nu \sum_{M \in 2^\Integers}\jap{t}^{4-K_D\epsilon} \int_{\eta,\xi} \abs{\eta A^S \hat{g}(\eta) e^{c\lambda\abs{\eta-\xi}^s}\widehat{G}(\eta-\xi)_{<M/8} \xi A^S(\xi) \hat{g}(\xi)_M } d\eta d\xi. 
\end{align*} 
Therefore by \eqref{ineq:L2L2L1}, $\sigma > 8$ followed by \eqref{ineq:GeneralOrtho} and \eqref{ineq:vp2m1bd} we have
\begin{align} 
V_{DE,LH} & \lesssim \nu \jap{t}^{4-K_D\epsilon} \sum_{M \in 2^\Integers}\norm{\partial_v A^S g_{\sim M}}_2^2 \norm{A^S \left((v^\prime)^2 - 1\right)_{<M/8}}_2 \nonumber \\ 
& \lesssim \epsilon \nu \jap{t}^{4-K_D\epsilon} \norm{\partial_v A^S g}_2^2, \label{ineq:VDELH}
\end{align} 
which is absorbed by the leading order dissipation term in \eqref{def:VDE} for $\epsilon$ small.  

The treatment of $V_{DE,HL}$ is similar to $V_{DE,HL}$. Indeed,  we have
\begin{align*} 
V_{DE,HL} & \lesssim \nu \sum_{M \in 2^\Integers}\jap{t}^{4-K_D\epsilon} \int_{\eta,\xi} \abs{A^S \hat{g}(\eta) A^S(\eta) \hat{G}(\eta-\xi)_{M} \abs{\xi}^2 \hat{g}(\xi)_{<M/8} } d\eta d\xi, 
\end{align*} 
and \eqref{ineq:E0HLFreqLoc} holds on the support of the integrand. 
Therefore, on the support of the integrand there still holds $\abs{\xi} \lesssim \abs{\eta}$, and hence using \eqref{lem:scon} we have for some $c \in (0,1)$, 
\begin{align*} 
V_{DE,HL} & \lesssim \nu \sum_{M \in 2^\Integers}\jap{t}^{4-K_D\epsilon} \int_{\eta,\xi} \abs{\eta A^S \hat{g}(\eta) A^S(\eta-\xi)\widehat{G}(\eta-\xi)_{M} \xi e^{c\lambda\abs{\xi}^s}\hat{g}(\xi)_{<M/8} } d\eta d\xi. 
\end{align*} 
As in \eqref{ineq:VDELH}, it follows from \eqref{ineq:L2L2L1}, $\sigma > 2$, \eqref{ineq:LPOrthoProject} and \eqref{ineq:vp2m1bd} that
\begin{align} 
V_{DE,HL} & \lesssim \epsilon \nu \jap{t}^{4-K_D\epsilon} \norm{\partial_v A^S g}_2^2, \label{ineq:VDEHL}
\end{align}
which is absorbed by the leading order dissipation in \eqref{def:VDE} for $\epsilon$ small. 

Finally turn to the remainder, $V_{DE,\R}$. 
As in \S\ref{sec:ZeroModeDiss}, divide first into low and high frequencies 
\begin{align*} 
V_{DE,\mathcal{R}} & = \nu \jap{t}^{4-K_D\epsilon} \sum_{M \in 2^\Integers} \sum_{M^\prime \approx M} \int (A^S g)_{\leq 1} A^S\left(G_{M^\prime} \partial_{vv}g_{M}\right) dv \\ 
& \quad + \nu \jap{t}^{4-K_D\epsilon} \sum_{M \in 2^\Integers} \sum_{M^\prime \approx M} \int (A^S g)_{> 1} A^S\left(G_{M^\prime} \partial_{vv}g_{M}\right) dv \\ 
& = V_{DE,\mathcal{R}}^{L} + V_{DE,\mathcal{R}}^{H}. 
\end{align*} 
The high frequencies $V_{DE,\mathcal{R}}^{H}$, as in \S\ref{sec:ZeroModeDiss}, can be treated by adding a derivative on the first factor and absorbing by the leading order dissipation in \eqref{def:VDE}. 
We omit the details and simply conclude
\begin{align} 
V_{DE,\mathcal{R}}^{H} & \lesssim \epsilon \nu  \jap{t}^{4-K_D\epsilon} \norm{\partial_v A^S g}_{2}^2, \label{ineq:VDERH}
\end{align} 
which is absorbed by the leading order dissipation in \eqref{def:VDE} for $\epsilon$ small. 

To treat $V_{DE,\mathcal{R}}^L$, also similar to \S\ref{sec:ZeroModeDiss}, we use Cauchy-Schwarz followed by Bernstein's inequalities, 
\begin{align*} 
 V_{DE,\mathcal{R}}^{L} & \lesssim \nu \jap{t}^{4-K_D\epsilon} \sum_{M \in 2^\Integers} \sum_{M^\prime \approx M} \norm{P_{\leq 1} A^S g}_{2} \norm{P_{\leq 1} A^S\left(G_{M^\prime} \partial_{vv}g_{M}\right)}_{2} \\ 
& \lesssim \nu \jap{t}^{4-K_D\epsilon} \sum_{M \in 2^\Integers} \sum_{M^\prime \approx M} \norm{g_{<1}}_{2} \norm{G_{M^\prime} \partial_{vv}g_{M}}_{2} \\ 
& \lesssim \nu \jap{t}^{4-K_D\epsilon} \sum_{M \in 2^\Integers} \sum_{M^\prime \approx M} \norm{g_{<1}}_{2} \norm{G_{M^\prime}}_\infty \norm{\partial_{vv}g_{M}}_{2} \\ 
& \lesssim \nu \jap{t}^{4-K_D\epsilon} \sum_{M \in 2^\Integers} \sum_{M^\prime \approx M} M^{3/2} \norm{g}_{2} \norm{G_{M^\prime}}_2 \norm{\partial_{v}g_{M}}_{2} \\
& \lesssim \nu \jap{t}^{4-K_D\epsilon} \sum_{M \in 2^\Integers} \sum_{M^\prime \approx M} M^{1/2} \norm{g}_{2} \norm{\partial_v G_{M^\prime}}_2 \norm{\partial_{v}g_{M}}_{2} \\ 
& \lesssim \nu \jap{t}^{4-K_D\epsilon} \sum_{M \in 2^\Integers:M \leq 1} M^{1/2} \norm{g}_{2} \norm{\partial_v\left(1 - (v^\prime)^2\right)_{\sim M}}_2 \norm{\partial_{v}g_{M}}_{2} \\ 
& \quad + \nu \jap{t}^{4-K_D\epsilon} \sum_{M \in 2^\Integers:M > 1} M^{1/2} \norm{g}_{2} \norm{\partial_v\left(1 - (v^\prime)^2\right)_{\sim M}}_2 \norm{\partial_{v}g_{M}}_{2}. 
\end{align*} 
The first term is summed by Cauchy-Schwarz in $M$ (and \eqref{ineq:GeneralOrtho})
whereas the second term is summed by paying additional derivatives on the last factor to introduce a negative power of $M$ and then applying Cauchy-Schwarz.  
Therefore, since $\sigma > 7$, 
\begin{align*} 
V_{DE,\mathcal{R}}^{L} & \lesssim \nu\jap{t}^{4-K_D\epsilon} \norm{g}_{2} \norm{\partial_v\left(1 - (v^\prime)^2\right) }_2 \norm{\partial_{v}g}_{2} + \nu\jap{t}^{4-K_D\epsilon} \norm{g}_{2} \norm{\partial_v \left(1 - (v^\prime)^2\right)}_2 \norm{\partial_{v}g}_{\G^{\lambda,\sigma-6}} \\ 
& \lesssim \nu\jap{t}^{4-K_D\epsilon} \norm{g}_{2} \norm{\partial_v \left(1 - (v^\prime)^2\right)}_2 \norm{\partial_{v}g}_{\G^{\lambda,\sigma-6}} \\ 
& \lesssim \epsilon \nu\jap{t}^{4-K_D\epsilon}\norm{\partial_{v}g}_{\G^{\lambda,\sigma-6}}^2 + \epsilon^{-1} \nu\jap{t}^{4-K_D\epsilon} \norm{g}_2^2\norm{v^\prime \partial_v v^{\prime}}_2^2. 
\end{align*} 
Now we crucially use \eqref{ineq:vppdecay} (from \eqref{ineq:Boot_hdecay}) to deduce for some $C > 0$
\begin{align} 
V_{DE,\mathcal{R}}^{L} & \leq C \epsilon \nu\jap{t}^{4-K_D\epsilon}\norm{\partial_{v}g}_{\G^{\lambda,\sigma-6}}^2 + \frac{K_D}{8}\epsilon \jap{t}^{3-K_D\epsilon} \norm{g}_2^2. \label{ineq:VDERL}
\end{align}
The first term is absorbed by the leading order dissipation in \eqref{def:VDE} and the latter term is absorbed by the damping term $V_1$ in \eqref{def:V1V2V3} by choosing $\epsilon$ sufficiently small.  

By integrating \eqref{ineq:ddtpartt} using \eqref{ineq:timederivpartialt}, \eqref{def:V1V2V3}, \eqref{ineq:V2}, \eqref{ineq:V3} \eqref{def:VDE}, \eqref{ineq:VDELH}, \eqref{ineq:VDEHL}, \eqref{ineq:VDERH} and \eqref{ineq:VDERL}, 
the proof of \eqref{ineq:gc} is completed by choosing $\epsilon$ sufficiently small. 

\subsection{Proof of \eqref{ineq:bhc}} 
We extend the proof of the corresponding statement in \cite{BM13}. 
From \eqref{def:barh} we have
\begin{align}
\frac{d}{dt} \left( \jap{t}^{2+2s} \norm{\frac{A}{\jap{\partial_v}^s} \bar{h}}_2^2\right) & = -(2-2s)t \jap{t}^{2s}\norm{\frac{A}{\jap{\partial_v}^s} \bar{h}}_2^2 - CK^{v,2}_\lambda - CK^{v,2}_w \nonumber \\ & \quad - 2 \jap{t}^{2+2s} \int \frac{A}{\jap{\partial_v}^s} \bar{h} \frac{A}{\jap{\partial_v}^s}\left(g \partial_v\bar{h} \right) dv \nonumber \\ 
& \quad + 2 t^{-1}\jap{t}^{2+2s} \int \frac{A}{\jap{\partial_v}^s} \bar{h} \frac{A}{\jap{\partial_v}^s} <v^\prime \grad^\perp P_{\neq 0} \phi \cdot \grad f > dv  \nonumber \\ & \quad + 2\nu \jap{t}^{2+2s}  \int \frac{A}{\jap{\partial_v}^s}\bar{h} \frac{A}{\jap{\partial_v}^s}\left(\tilde{\Delta_t} \bar{h}\right) dv   \nonumber \\ 
& = -CK_L^{v,2} - CK^{v,2}_\lambda - CK^{v,2}_w + \mathcal{T}^h + F + \mathcal{D}^h. \label{ineq:barhenergy}
\end{align} 
The main nonlinear terms, $\mathcal{T}^h$ and $F$, are treated in \S8.2 of \cite{BM13}. 
Using the techniques therein implies:
\begin{align}
\abs{\mathcal{T}^h} & \lesssim \epsilon CK_\lambda^{v,2} + \epsilon CK_{\lambda}^{v,1} + \epsilon \jap{t}^{2s + K_D\epsilon/2}\norm{\frac{A}{\jap{\partial_v}^s} \bar{h} }^2_2.   \label{ineq:Th}
\end{align}
Controlling the `forcing' term $F$ is one of the key estimates made in \cite{BM13} (found in \S8.2): 
\begin{align} 
\abs{F} & \lesssim \epsilon CK_{\lambda}^{v,2} + \epsilon CK_w^{v,2} + \epsilon CK_\lambda + \epsilon CK_w \nonumber \\ & \quad + \epsilon^3 \sum_{i=1}^2 CCK_\lambda^i + CCK_w^i + \epsilon CK_L^{v,2} + \epsilon^3 \jap{t}^{2s-3}. \label{ineq:F}
\end{align} 
Hence we only need to focus on the dissipation error term $\mathcal{D}^h$, which we treat in a manner 
very similar to $V_{DE}$. 
Indeed, write 
\begin{align}
\mathcal{D}^h & = -2\nu \jap{t}^{2+2s} \norm{\partial_v\frac{A}{\jap{\partial_v}^s}\bar{h}}_2^2 - 2\nu \jap{t}^{2+2s} \int \frac{A}{\jap{\partial_v}^s}\bar{h} \frac{A}{\jap{\partial_v}^s}\left( (1-(v^\prime)^2) \partial_{vv}\bar{h}\right) dv \nonumber \\ 
& = -2\nu\jap{t}^{2+2s}  \norm{\partial_v \frac{A}{\jap{\partial_v}^s}\bar{h}}_2^2 + \mathcal{E}^h \label{eq:DhEh} 
\end{align} 
As usual, we use the shorthand \eqref{def:Gshort} and decompose the error term via homogeneous paraproduct:
\begin{align*} 
\mathcal{E}^h & = -2\nu\jap{t}^{2+2s} \sum_{M \in 2^\Integers} \int \frac{A}{\jap{\partial_v}^s} \bar{h} \frac{A}{\jap{\partial_v}^s} \left(G_{M} \partial_{vv}\bar{h}_{<M/8}\right) dv \\ 
& \quad - 2\nu\jap{t}^{2+2s} \sum_{M \in 2^\Integers}\int \frac{A}{\jap{\partial_v}^s} \bar{h} \frac{A}{\jap{\partial_v}^s} \left(G_{<M/8} \partial_{vv} \bar{h}_{M}\right) dv \\
& \quad - 2\nu \jap{t}^{2+2s}\sum_{M \in 2^\Integers}\sum_{M^\prime \approx M} \int  \frac{A}{\jap{\partial_v}^s} \bar{h}\frac{A}{\jap{\partial_v}^s} \left(G_{M^\prime} \partial_{vv}\bar{h}_{M}\right) dv \\ 
& = \mathcal{E}^h_{HL} + \mathcal{E}^h_{LH} + \mathcal{E}^h_{\mathcal{R}}. 
\end{align*} 
The treatments of $\mathcal{E}^h_{HL}$ and $\mathcal{E}^h_{LH}$ mirror that of $V_{DE,HL}$ and $V_{DE,LH}$ except using also \eqref{ineq:BasicJswap}. The argument is essentially the same so it is omitted for brevity and we conclude that 
\begin{align}
 \mathcal{E}^h_{HL} + \mathcal{E}^h_{LH} & \lesssim \epsilon \nu \jap{t}^{2+2s} \norm{\partial_v \frac{A}{\jap{\partial_v}^s}\bar{h}}_2^2,\label{ineq:EhHLLH}
\end{align}
which is absorbed by the leading order dissipation in \eqref{eq:DhEh} after choosing $\epsilon$ small.  

The treatment of $\mathcal{E}^h_{\mathcal{R}}$ is also very similar to the treatment of $V_{DE,\R}$. 
Indeed, we first divide into low and high frequencies: 
\begin{align*}
\mathcal{E}^h_{\mathcal{R}} & = -2\nu \jap{t}^{2+2s}\sum_{M \in 2^\Integers}\sum_{M^\prime \approx M} \int  \left(\frac{A}{\jap{\partial_v}^s} \bar{h}\right)_{\leq 1} \frac{A}{\jap{\partial_v}^s} \left(G_{M^\prime} \partial_{vv}\bar{h}_{M}\right) dv \\ 
& \quad - 2\nu \jap{t}^{2+2s}\sum_{M \in 2^\Integers}\sum_{M^\prime \approx M} \int \left(\frac{A}{\jap{\partial_v}^s} \bar{h}\right)_{> 1}\frac{A}{\jap{\partial_v}^s} \left(G_{M^\prime} \partial_{vv}\bar{h}_{M}\right) dv \\ 
& = \mathcal{E}_{\mathcal{R}}^{L} + \mathcal{E}_{\mathcal{R}}^{H}. 
\end{align*} 
As for $V_{DE}$, the high frequencies are treated without much effort as in \S\ref{sec:ZeroModeDiss} and absorbed by the leading order dissipation \eqref{eq:DhEh} after choosing $\epsilon$ small. Hence this contribution is omitted.
To treat the lower frequencies we use an argument similar to that used to treat $V_{DE,\mathcal{R}}^L$. Indeed by Cauchy-Schwarz and Bernstein's inequalities we have as above (skipping repetitive details):
\begin{align*} 
\abs{\mathcal{E}_{\mathcal{R}}^{L}} & \lesssim \nu \jap{t}^{2+2s} \sum_{M \in 2^\Integers:M \leq 1} M^{1/2} \norm{\bar{h}}_{2} \norm{\partial_v\left(1 - (v^\prime)^2\right)_{\sim M}}_2 \norm{\partial_{v} \bar{h} _{M}}_{2} \\ 
& \quad + \nu \jap{t}^{2+2s} \sum_{M \in 2^\Integers:M > 1} M^{1/2} \norm{\bar{h}}_{2} \norm{\partial_v\left(1 - (v^\prime)^2\right)_{\sim M}}_2 \norm{\partial_{v}\bar{h}_{M}}_{2}. 
\end{align*} 
The first (low frequency) term is summed by Cauchy-Schwarz in $M$ (and \eqref{ineq:GeneralOrtho})
whereas the second term is summed by paying additional derivatives on the last factor to introduce a negative power of $M$ and then applying Cauchy-Schwarz.  
As in \S\ref{sec:gc}, since $\sigma > s + 1/2$,
\begin{align*} 
\abs{\mathcal{E}_{\mathcal{R}}^{L}} & \lesssim \nu\jap{t}^{2 + 2s} \norm{\bar{h}}_{2} \norm{\partial_v\left(1 - (v^\prime)^2\right) }_2 \norm{\partial_{v}\bar{h}}_{2} + \nu\jap{t}^{2+2s} \norm{\bar{h}}_{2} \norm{\partial_v \left(1 - (v^\prime)^2\right)}_2 \norm{\partial_{v} \frac{A}{\jap{\partial_v}^{s}}\bar{h}}_{2} \\ 
& \lesssim \epsilon \nu\jap{t}^{2+2s}  \norm{\partial_{v}\frac{A}{\jap{\partial_v}^{s}}\bar{h}}_{2}^2 + \epsilon^{-1} \nu\jap{t}^{2+2s} \norm{\bar{h}}_2^2 \norm{\partial_v \left( 1 -  (v^{\prime})^2 \right)}_2^{2}. 
\end{align*} 
By \eqref{ineq:vppdecay} we have 
\begin{align} 
\abs{\mathcal{E}_{\mathcal{R}}^{L}} & \lesssim \epsilon \nu \jap{t}^{2+2s}\norm{\partial_{v}\frac{A}{\jap{\partial_v}^{s}}\bar{h}}_{2}^2 + \epsilon \jap{t}^{1+2s} \norm{\bar{h}}_2^2. \label{ineq:EhRL}
\end{align}
The first term is absorbed by the leading order dissipation in \eqref{eq:DhEh} and the latter term is absorbed by the damping term $CK_L^{v,2}$ in \eqref{ineq:barhenergy} by choosing $\epsilon$ sufficiently small.  

Combining \eqref{ineq:barhenergy} with \eqref{ineq:Th}, \eqref{ineq:F}, \eqref{ineq:EhHLLH} and  \eqref{ineq:EhRL} and integrating in time implies \eqref{ineq:bhc} after choosing $\epsilon$ sufficiently small. 

\subsection{Proof of \eqref{ineq:ARh}}
From \eqref{def:PDEh}  we have
\begin{align}
\frac{1}{2}\frac{d}{dt}\norm{A^R h}_2^2 & = -CK^{h}_\lambda - CK^h_w - \int A^Rh A^R\left(g\partial_v h \right) dv \nonumber \\ & \quad + \int A^R h A^R \bar{h} dv + \nu \int A^R h A^R \left( \tilde{\Delta_t} h \right) dv, \label{eq:henergy}
\end{align} 
where 
\begin{subequations} 
\begin{align}
CK_w^{h}(\tau) & = \norm{\sqrt{\frac{\partial_t w}{w}}A^R h(\tau)}_2^2 \\ 
CK_\lambda^{h}(\tau) & = (-\dot{\lambda}(\tau))\norm{\abs{\partial_v }^{s/2} A^R h(\tau)}_2^2.
\end{align} 
\end{subequations}
The nonlinear transport term is controlled in \S8 of \cite{BM13}; here we simply recall the result: 
\begin{align}
\abs{\int A^Rh A^R\left(g\partial_v h \right) dv} \lesssim \epsilon CK_\lambda^h + \epsilon CK_w^h + \epsilon CK_\lambda^{v,1} + \epsilon CK_w^{v,1} + \epsilon^3 \jap{t}^{-2 + K_D\epsilon/2}.  \label{ineq:ARTrans}
\end{align} 
Similarly, the linear driving term from $\bar{h}$ is treated in \cite{BM13}; the result is for some $C_{h} > 0$: 
\begin{align} 
\abs{\int A^R h A^R \bar{h} dv} \leq \frac{1}{4}CK_\lambda^h + \frac{1}{4}CK_w^h + C_h\left(CK_\lambda^{v,2} + CK_w^{v,2}\right). \label{ineq:ARdrive}
\end{align} 
The presence of the constant $C_h$ is the primary reason for the constant $K_v$ in the main bootstrap. 

The new term we need to treat in \eqref{eq:henergy} is the dissipation term, which we treat in the same manner as the zero mode dissipation in the proof of \eqref{ineq:energyHi} in \S\ref{sec:ZeroModeDiss}. 
As there, we write
\begin{align} 
\nu \int A^R h A^R \left( \tilde{\Delta_t} h \right) dv & = - \nu \norm{\partial_v A^R h}_2^2 + \nu\int A^R h A^R \left( \left(1 - (v^\prime)^2 \right) \partial_{vv} h \right) dv \nonumber \\ 
& = - \nu \norm{\partial_v A^R h}_2^2 + E^R. \label{ineq:ARhDiss}
\end{align} 
The treatment of $E^R$ is essentially the same as $E^0$ in \S\ref{sec:ZeroModeDiss} and is hence omitted for brevity. Indeed, the coefficients of $\tilde{\Delta_t}$ have resonant regularity from \eqref{ineq:coefbds} and hence there is little difference between $A^R$ and $A$ in the proof.  
Analogous to \S\ref{sec:ZeroModeDiss} we get for $\epsilon$ sufficiently small, 
\begin{align}
E^R & \lesssim \epsilon \nu \norm{\partial_v A^R h}_2^2 + \frac{\nu \epsilon}{\jap{\nu t}^{3/2}} \norm{A^R h}_2^2. \label{ineq:E}
\end{align} 

Putting together \eqref{eq:henergy}, \eqref{ineq:ARTrans}, \eqref{ineq:ARdrive}, \eqref{ineq:ARhDiss} and \eqref{ineq:E} and integrating in time completes the proof of \eqref{ineq:ARh} after choosing $\epsilon$ small and $K_v$ sufficiently large (independent of $\epsilon$ of course).  

\section{Low frequency estimates} \label{sec:LoFreq}
\subsection{Kinetic energy and energy dissipation control: \eqref{ineq:KE}}
Compute the time evolution from the momentum equation
\begin{align} 
\frac{1}{2}\frac{d}{dt}\norm{\tilde u_0}_2^2 & = -\int \tilde u_0  \left(g \partial_v \tilde u_0\right) dv - \int \tilde u_0 v^\prime <\grad^\perp P_{\neq 0}\phi \cdot \grad \tilde u > dv
 \nonumber \\ & \quad + \nu\int  \tilde u_0 (v^{\prime})^2 \partial_{vv}\tilde u_0 dv. \label{ineq:KEEvo}
\end{align} 
For the first term we use integration by parts following by Sobolev embedding and \eqref{ineq:Boot_gc}, 
\begin{align}
-\int \tilde u_0  \left(g \partial_v \tilde u_0\right) dv & = \frac{1}{2}\int \partial_vg \abs{\tilde u_0}^2 dv  \lesssim \norm{\partial_v g}_\infty \norm{\tilde u_0}_2^2 \lesssim \frac{\epsilon}{\jap{t}^{2-K_D\epsilon/2}}\norm{\tilde u_0}_2^2. \label{ineq:KEgv}
\end{align} 
For the forcing from the non-zero frequencies, we use H\"older's inequality followed by \eqref{ineq:Boot_vprime}, Gagliardo-Nirenberg-Sobolev and Cauchy-Schwarz in $k$: 
\begin{align*} 
\abs{\int \tilde u_0 v^\prime <\grad^\perp P_{\neq 0} \phi \cdot \grad \tilde u > dv} & = \sum_{k \neq 0}  \int \tilde u_0  v^\prime \grad^\perp \phi_k \cdot \grad \tilde u_{-k}  dv \\ 
& \leq \norm{\tilde u_0}_2 \left(1 + \norm{v^\prime - 1}_\infty \right) \sum_{k \neq 0} \norm{\grad^\perp \phi_k}_4 \norm{\grad \tilde u_{-k}}_4 \\ 
& \lesssim \norm{\tilde u_0}_2  \norm{\grad^\perp P_{\neq 0} \phi}_{H^{2}} \norm{\grad P_{\neq 0} \tilde u}_{H^2}. 
\end{align*} 
Therefore, by Lemma \ref{lem:LossyElliptic}, \eqref{ineq:Boot_energyHi} and \eqref{ineq:gradulossy} we have
\begin{align} 
\abs{\int \tilde u_0 v^\prime <\grad^\perp P_{\neq 0} \phi \cdot \grad \tilde u > dv} & \lesssim \frac{\epsilon^2}{\jap{t}^3}\norm{\tilde u_0}_2 \lesssim \frac{\epsilon}{\jap{t}^3}\norm{\tilde u_0}_2^2 +\frac{\epsilon^3}{\jap{t}^3}. \label{ineq:ForcingTerm}
\end{align} 
Finally we deal with the dissipation term \eqref{ineq:KEEvo}: 
\begin{align} 
\nu\int  \tilde u_0 (v^{\prime})^2 \partial_{vv}\tilde u_0 dv = -\nu\norm{v^\prime \partial_v \tilde u_0}_2^2 - 2\nu\int \tilde u_0 v^\prime \partial_v v^\prime \partial_v \tilde u_0 dv. \label{ineq:u0Diss}
\end{align} 
Note that by \eqref{ineq:Boot_vprime}, 
\begin{align*} 
\norm{\partial_v \tilde u_0}_2^2 \approx \norm{v^\prime \partial_v \tilde u_0}_2^2.  
\end{align*}
To deal with the error term in \eqref{ineq:u0Diss}, we employ the following Gagliardo-Nirenberg-Sobolev inequality for functions $X=X(v)$,
\begin{align}
\norm{X}_\infty \lesssim \norm{X}_2^{1/2} \norm{\partial_v X}_2^{1/2}.  \label{ineq:1DGNS}
\end{align}
Applying \eqref{ineq:1DGNS} along with \eqref{ineq:hdecay} to the error term in \eqref{ineq:u0Diss} implies,
\begin{align} 
\abs{2\nu\int \tilde u_0 v^\prime \partial_v v^{\prime} \partial_v \tilde u_0 dv} & \leq 2\nu \norm{\tilde u_0}_\infty \norm{\partial_v v^{\prime}}_2 \norm{v^\prime \partial_v \tilde u_0}_2 \nonumber \\ 
& \lesssim \nu \norm{\tilde u_0}_2^{1/2} \norm{\partial_v v^{\prime}}_2 \norm{v^\prime \partial_v \tilde u_0}^{3/2}_2 \nonumber \\ 
& \lesssim \nu \epsilon^{-3} \norm{\tilde u_0}_2^{2} \norm{\partial_v v^{\prime}}_2^4 + \epsilon \nu \norm{v^\prime \partial_v \tilde u_0}^{2}_2 \nonumber \\ 
& \lesssim \frac{\nu \epsilon}{\jap{\nu t}^{4}} \norm{\tilde u_0}_2^{2}  +  \epsilon \nu\norm{v^\prime \partial_v \tilde u_0}^{2}_2. \label{ineq:KEdissErr}
\end{align} 
The second term is absorbed by the dissipation in \eqref{ineq:u0Diss} and the first term gives an integrable contribution. 
Hence, integrating \eqref{ineq:KEEvo} with \eqref{ineq:KEdissErr}, \eqref{ineq:u0Diss}, \eqref{ineq:ForcingTerm} and \eqref{ineq:KEgv} implies \eqref{ineq:KE}. 

\subsection{Decay estimate \eqref{ineq:hdecay}} \label{sec:hdecay}
First, \eqref{ineq:hdecay} follows from \eqref{ineq:Boot_energyHi} for $t \leq \nu^{-1}$, so hence assume 
$t > \nu^{-1}$ for the rest of the section.  
The $L^2$ decay estimate is obtained via an energy estimate and the application of a suitable Gagliardo-Nirenberg-Sobolev inequality. 
The $L^2$ gradient decay estimate is then obtained via iterating from the $L^2$ estimate as in standard Moser iteration methods.
The iteration procedure could be carried out for higher derivatives and higher $L^p$ norms, but this is not necessary for our purposes. 

Begin with the $L^2$ norm by computing the time evolution from \eqref{def:PDEh},  
\begin{align} 
\frac{1}{2}\frac{d}{dt}\norm{h}_2^2 = -\int h g \partial_v h dv + \int h \bar{h} dv + \nu\int h (v^\prime)^2 \partial_{vv} h dv. \label{ineq:hlowEvo}
\end{align} 
The first term is controlled via integration by parts, Sobolev embedding and \eqref{ineq:Boot_gc},
\begin{align}
-\int h g \partial_v h dv  = \frac{1}{2}\int h^2 \partial_v g dv & \leq \frac{1}{2}\norm{\partial_v g}_\infty \norm{h}_2^2 \lesssim \epsilon \jap{t}^{-2+K_D\epsilon/2}\norm{h}_2^2. \label{ineq:hlow1}
\end{align}
By Cauchy-Schwarz, \eqref{eq:barhgRelat},  \eqref{ineq:Boot_gc} and \eqref{ineq:Boot_vprime} we have, 
\begin{align} 
\int h \bar{h} dv \leq \norm{h}_2 \norm{v^\prime \partial_v g}_2 &  \lesssim \frac{\epsilon}{\jap{t}^{2 - K_D\epsilon/2}}\norm{h}_2 \nonumber \\ 
& \lesssim \frac{1}{\jap{t}^{3/2 - K_D\epsilon}}\norm{h}_2^2 + \frac{\epsilon^2}{\jap{t}^{5/2}}; \label{ineq:hlow2}
\end{align}
the choice of $5/2$ is not sharp or significant. 

Turn next to the dissipation term. By integration by parts,
\begin{align} 
\nu\int h (v^\prime)^2 \partial_{vv} h dv = -\nu\norm{v^\prime \partial_v h}_2^2 - 2\nu \int h v^\prime \partial_v v^{\prime} \partial_{v}h dv. \label{eq:hlowdiss}
\end{align}  
However, since $\partial_v v^\prime = \partial_v h$ we have by Sobolev embedding, \eqref{ineq:Boot_vprime} and \eqref{ineq:Boot_ARh},
\begin{align}
\abs{2\nu \int h v^\prime \partial_v v^{\prime} \partial_{v}h dv} & \leq 2\nu \norm{h}_\infty \norm{(v^\prime)^{-1}}_\infty \norm{v^\prime \partial_v h}_2^2 \nonumber \\ 
& \lesssim \nu\norm{h}_{H^{2}}\norm{v^\prime \partial_v h}_2^2 \nonumber \\ 
& \lesssim \epsilon \nu \norm{v^\prime \partial_v h}_2^2. \label{ineq:hlowDE}
\end{align}  
Therefore, for $\epsilon$ chosen sufficiently small, this term can be absorbed by the leading order dissipation in \eqref{eq:hlowdiss}.  

Combining \eqref{ineq:hlowEvo}, \eqref{ineq:hlow1}, \eqref{ineq:hlow2}, \eqref{eq:hlowdiss} and \eqref{ineq:hlowDE} with \eqref{ineq:Boot_vprime} yields the differential inequality,  
\begin{align} 
\frac{1}{2}\frac{d}{dt}\norm{h}_2^2 \lesssim -\nu \norm{\partial_v h}_2^2 + \frac{1}{\jap{t}^{3/2 - K_D\epsilon}}\norm{h}_2^2 + \frac{\epsilon^2}{\jap{t}^{5/2}}.\label{ineq:hdecMainEst}
\end{align}
We next show that \eqref{ineq:hdecMainEst} implies $\norm{h}_2 \lesssim \epsilon\jap{\nu t}^{-1/4}$. 
First, by the Gagliardo-Nirenberg-Sobolev inequality and the $L^1$ control \eqref{ineq:L1h} we have
\begin{align*} 
\norm{h}_2 \lesssim  \norm{h}_1^{2/3} \norm{\partial_v h}_2^{1/3} \lesssim \epsilon^{2/3}\norm{\partial_v h}_2^{1/3}.
\end{align*}
Hence, denoting $X(t) = \norm{h(t)}_2^2$, this implies with \eqref{ineq:hdecMainEst} that for $t \geq \nu^{-1}$ we have the following (\eqref{ineq:hdecay} is immediate from \eqref{ineq:ARh} for earlier times), 
\begin{align} 
\frac{d}{dt}\left( (\nu t)^{1/2} X(t) \right) \lesssim \frac{\sqrt{\nu}}{\sqrt{t}}X(t) - \nu(\nu t)^{1/2}\epsilon^{-4}X^3(t) + \frac{1}{\jap{t}^{3/2-K_D\epsilon}}(\nu t)^{1/2} X(t) + \frac{\epsilon^2 \nu^{1/2} }{\jap{t}^{2}}. \label{ineq:Xdinq}
\end{align} 
Next use the $O(1)$ integrating factor
\begin{align*} 
Y(t) = X(t)\exp\left[ -\int_{\nu^{-1}}^t \jap{\tau}^{K_D\epsilon - 3/2} d\tau \right]; 
\end{align*}
note that $X(t) \approx Y(t)$. 
This collapses \eqref{ineq:Xdinq} to
\begin{align}  
\frac{d}{dt}\left( (\nu t)^{1/2}Y(t)\right) \lesssim \frac{\sqrt{\nu}}{\sqrt{t}}Y(t) -\nu(\nu t)^{1/2}\epsilon^{-4}Y^3(t) + \frac{\epsilon^2 \nu^{1/2} }{\jap{t}^{2}}. \label{ineq:Yinq}
\end{align} 
For $K_h^\prime$ chosen sufficiently large (depending on the constants in \eqref{ineq:Yinq}), the following is a supersolution of the differential equality corresponding to \eqref{ineq:Yinq}:
\begin{align*} 
Y_{S}(t) = \frac{K_h^\prime \epsilon^2}{\jap{\nu t}^{1/2}} + \frac{K_h^\prime}{\jap{\nu t}^{1/2}} \int_{\nu^{-1}}^t \frac{\epsilon^2 \nu^{1/2} }{\jap{\tau}^{2}} d\tau.
\end{align*} 
Moreover, for $K_h^\prime$ large, we have $Y(\nu^{-1}) \leq Y_S(\nu^{-1})$ and hence by comparison we have
$Y(t) \leq Y_S(t)$ for all $t \geq \nu^{-1}$, from which the $L^2$ estimate in \eqref{ineq:hdecay} follows after
possibly adjusting $K_v$.

Turn next to the $\partial_v h$ bound stated in \eqref{ineq:hdecay}, for which we will bootstrap off the $L^2$ decay estimate.
Taking a derivative of \eqref{def:PDEh} gives
\begin{align*} 
\partial_t \partial_v h + \partial_v g \partial_v h + g \partial_{vv} h = \partial_v \bar{h} + \nu \partial_v \left((v^\prime)^2 \partial_{vv}h\right).    
\end{align*}  
Hence calculating the evolution of the $L^2$ norm and applying integration by parts: 
\begin{align} 
\frac{1}{2}\frac{d}{dt}\norm{\partial_v h}_2^2 & = -\int \partial_vg\abs{\partial_v h}^2 dv 
  - \int \partial_v h g \partial_{vv} h dv \nonumber \\
& \quad + \int \partial_v h \partial_v \bar{h} dv - \nu \norm{v^\prime \partial_{vv}h}_2^2  \nonumber \\ 
& =  -\frac{1}{2}\int \abs{\partial_vh}^2 \partial_v g dv + \int \partial_v h \partial_v \bar{h} dv - \nu\norm{v^\prime \partial_{vv}h}_2^2.  \label{ineq:ghdec_Evo}
\end{align}
The first term is treated with Sobolev embedding and \eqref{ineq:Boot_gc},
\begin{align} 
-\frac{1}{2}\int \abs{\partial_vh}^2 \partial_v g dv \lesssim \epsilon \jap{t}^{-2+K_D\epsilon/2} \norm{\partial_v h}_2^2. \label{ineq:ghdec_gt}
\end{align} 
The second term in \eqref{ineq:ghdec_Evo} is controlled by \eqref{ineq:Boot_gc} and \eqref{ineq:Boot_vprime} via \eqref{eq:barhgRelat} and Sobolev embedding:  
\begin{align}
\abs{\int \partial_v h \partial_v \bar{h} dv} & \leq \norm{\partial_v h}_2 \left(\norm{v^\prime}_\infty\norm{\partial_{vv} g}_2 + \norm{\partial_v h}_2\norm{\partial_{v} g}_\infty\right) \nonumber  \\ 
& \lesssim \epsilon \jap{t}^{-2+K_D\epsilon/2}\norm{\partial_v h}_2 + \epsilon \jap{t}^{-2+K_D\epsilon/2} \norm{\partial_v h}_2^2 \nonumber \\
& \lesssim \jap{t}^{-\frac{5}{4}+K_D\epsilon}\norm{\partial_v h}^2_2 + \epsilon^2 \jap{t}^{-11/4}. \label{ineq:ghdec_barh}
\end{align} 
We note that the $11/4$ is not sharp or significant, however, it must be chosen larger than $5/2$ to get the optimal decay (and cannot be as large as $3$).
Putting together \eqref{ineq:ghdec_Evo}, \eqref{ineq:ghdec_gt} and \eqref{ineq:ghdec_barh} with the bootstrap control on $v^\prime$, we get
\begin{align*}
\frac{d}{dt}\norm{\partial_v h}_2^2 & \lesssim -\nu\norm{\partial_{vv}h}_2^2 + \jap{t}^{-\frac{5}{4}+K_D\epsilon}\norm{\partial_v h}^2_2 + \epsilon^2 \jap{t}^{-11/4}. 
\end{align*}
By interpolation and the $L^2$ decay estimate in \eqref{ineq:Boot_hdecay} we have
\begin{align*} 
\norm{\partial_v h}_2 \leq \norm{h}_2^{1/2}\norm{\partial_{vv}h}_2^{1/2} \lesssim \jap{\nu t}^{-1/8}\epsilon^{1/2} \norm{\partial_{vv}h}_2^{1/2}, 
\end{align*}
which therefore implies, 
\begin{align*} 
\frac{d}{dt}\norm{\partial_v h}_2^2 & \lesssim -\nu \epsilon^{-2} \jap{\nu t}^{1/2}\norm{\partial_v h}^4_2  + \jap{t}^{-\frac{5}{4}+K_D\epsilon}\norm{\partial_v h}^2_2 + \epsilon^2 \jap{t}^{-11/4}. 
\end{align*}
Following the proof of the $L^2$ decay estimate, we define 
 $Y(t) =\norm{\partial_v h}_2^2\exp\left[-\int_{\nu^{-1}}^t \jap{\tau}^{-5/4 + K_D\epsilon} d\tau\right]$ and for $t \geq \nu^{-1}$ we have
\begin{align*}
\frac{d}{dt}\left( (\nu t)^{3/2}Y(t) \right) \lesssim \frac{(\nu t)^{3/2}}{t}Y(t) -\nu(\nu t)^{2}\epsilon^{-2}Y^2(t) + \frac{\epsilon^2 \nu^{3/2} }{\jap{t}^{5/4}}.
\end{align*} 
By comparing against the supersolution (for $K^{\prime\prime}_h$ sufficiently large), 
\begin{align*} 
Y_S(t) = \frac{K_h^{\prime\prime} \epsilon^2}{ (\nu t)^{3/2} } + \frac{K_h^{\prime\prime} \epsilon^2}{ (\nu t)^{3/2} } \int_{\nu^{-1}}^t\frac{ \nu^{3/2} }{\jap{\tau}^{5/4}} d\tau, 
\end{align*}
and choosing $K_h^{\prime\prime}$ large enough to ensure $Y(\nu^{-1}) < Y_S(\nu^{-1})$ we deduce $Y(t) \lesssim Y_S(t)$ for $t \geq \nu^{-1}$, which implies the $\norm{\partial_v h}_2$ decay estimate in \eqref{ineq:hdecay} (after possibly adjusting $K_v$ further).

\subsection{Zero frequency \texorpdfstring{$L^2$}{L 2} decay}
Here we prove \eqref{ineq:LoDecay}, which is not necessary for the  proof of the other statements in  Proposition \ref{prop:BootProp}, but is included in the statement of Theorem \ref{thm:Main} and gives the 
relaxation to the Couette flow.  
Begin by computing the evolution of the $L^2$ norm of the zero frequency, 
\begin{align} 
\frac{1}{2}\frac{d}{dt}\norm{f_0}_2^2 = -\int f_0 g \partial_v f_0 dv - \int f_0  v^\prime <\grad^\perp P_{\neq 0} \phi \cdot \grad f > dv + \nu\int f_0 (v^\prime)^2 \partial_{vv}{f} dv. \label{ineq:f0dvdt}
\end{align} 
The first term is controlled via integration by parts, Sobolev embedding and \eqref{ineq:Boot_gc}, 
\begin{align}
-\int f_0 g \partial_v f_0 dv & = \frac{1}{2}\int \abs{f_0}^2 \partial_v g dv \lesssim \epsilon \jap{t}^{-2+K_D\epsilon/2}\norm{f_0}_2^2. \label{ineq:f0gdv}
\end{align}
Using the same proof as in \eqref{ineq:ForcingTerm}, it follows from Lemma \ref{lem:LossyElliptic}, \eqref{ineq:Boot_vprime}  and \eqref{ineq:Boot_energyHi} that
\begin{align} 
-\int f_0  v^\prime <\grad^\perp P_{\neq 0} \phi \cdot \grad f > dv & \lesssim \frac{\epsilon^2}{\jap{t}^2}\norm{f_0}_2 \lesssim \frac{\epsilon}{\jap{t}^{3/2}}\norm{f_0}_2^2 + \frac{\epsilon^3}{\jap{t}^{5/2}}. \label{ineq:forcingf0}
\end{align} 
Turn next to the dissipation term: 
\begin{align} 
\nu\int f_0 (v^\prime)^2 \partial_{vv}{f} dv = -\nu\norm{v^\prime \partial_v f_0}_2^2 - 2\nu \int v^\prime \partial_v v^{\prime} f_0 \partial_{v}{f} dv.  \label{ineq:f0Diss}
\end{align}  
By H\"older's inequality, the Gagliardo-Nirenberg-Sobolev inequality \eqref{ineq:1DGNS}, and \eqref{ineq:vppdecay} we have 
\begin{align} 
\nu\int f_0 v^{\prime\prime} \partial_v f_0 dv & \leq \nu\norm{f_0}_\infty \norm{v^{\prime\prime}}_2 \norm{\partial_v f_0}_2 \nonumber \\ 
& \lesssim \nu\norm{f_0}_2^{1/2} \norm{v^{\prime\prime}}_2 \norm{\partial_v f_0}^{3/2}_2 \nonumber \\ 
& \lesssim \epsilon^{-3} \nu \norm{v^{\prime\prime}}_2^4 \norm{f_0}_2^{2} + \epsilon \nu \norm{\partial_v f_0}^{2}_2 \nonumber \\ 
& \lesssim \epsilon \nu \jap{\nu t}^{-3} \norm{f_0}_2^{2} +\epsilon \nu\norm{\partial_v f_0}^{2}_2. \label{ineq:f0DissError}
\end{align} 
Choosing $\epsilon$ sufficiently small and putting \eqref{ineq:f0gdv},  \eqref{ineq:forcingf0}, \eqref{ineq:f0Diss}, \eqref{ineq:f0DissError} together with \eqref{ineq:f0dvdt} gives the differential inequality 
\begin{align} 
\frac{1}{2}\frac{d}{dt}\norm{f_0}_2^2 \lesssim -\nu \norm{\partial_v f_0}_2^2 + \left(\frac{\epsilon}{\jap{t}^{3/2}} + \frac{\epsilon \nu}{\jap{\nu t}^{3}} \right)\norm{f_0}_2^2 + \frac{\epsilon^3}{\jap{t}^{5/2}}, \label{ineq:f0diff}
\end{align}
(the $5/2$ is not sharp or significant).
By the Gagliardo-Nirenberg-Sobolev inequality together with \eqref{ineq:fLp} and \eqref{ineq:Boot_vprime} we get, 
\begin{align} 
\norm{f_0}_2 \lesssim  \norm{f_0}_1^{2/3} \norm{\partial_v f_0}_2^{1/3} \lesssim \epsilon^{2/3}\norm{\partial_v f_0}_2^{1/3}.\label{ineq:Gagbdf0}
\end{align}
As in the proof of \eqref{ineq:hdecay} in \S\ref{sec:hdecay}, 
we use the integrating factor 
\begin{align*} 
Y(t) = \norm{f_0(t)}_2^2\exp\left[ -\int_{\nu^{-1}}^t \left(\frac{\epsilon}{\jap{t}^{3/2}} + \frac{\epsilon \nu}{\jap{\nu t}^{3}} \right) d\tau \right], 
\end{align*}
together with \eqref{ineq:Gagbdf0}, to reduce \eqref{ineq:f0diff} to the differential inequality, 
\begin{align} 
\frac{d}{dt}\left( (\nu t)^{1/2}Y(t)\right) \lesssim \frac{\sqrt{\nu}}{\sqrt{t}}Y(t) -\nu(\nu t)^{1/2}\epsilon^{-4}Y^3(t) + \frac{\epsilon^3 \nu^{1/2} }{\jap{t}^{2}}. \label{ineq:f0Y}
\end{align} 
For $K_f$ chosen sufficiently large (depending on the constants in \eqref{ineq:f0Y}), the following is a supersolution of the differential equality corresponding to \eqref{ineq:f0Y}:
\begin{align*} 
Y_{S}(t) = \frac{K_f\epsilon^2}{\jap{\nu t}^{1/2}} + \frac{K_f}{\jap{\nu t}^{1/2}} \int_{\nu^{-1}}^t \frac{\epsilon^3 \nu^{1/2} }{\jap{\tau}^{2}} d\tau.
\end{align*} 
For $K_f$ large, we have $Y(\nu^{-1}) \leq Y_S(\nu^{-1})$ and hence by comparison it follows that
$Y(t) \leq Y_S(t)$ for all $t \geq \nu^{-1}$, from which we deduce the $L^2$ decay stated in \eqref{ineq:LoDecay}.

\section*{Acknowledgments}
The authors would like to thank the following people for references and suggestions: Margaret Beck, Steve Childress, Peter Constantin, Pierre Germain, Yan Guo and Gene Wayne. The work of JB was in part supported by NSF Postdoctoral Fellowship in Mathematical Sciences DMS-1103765 and NSF grant DMS-1413177, the work of NM was in part supported by the NSF grant DMS-1211806, while the work of VV was in part supported by the NSF grant DMS-1348193.

\appendix
\section{Appendix}
\subsection{Littlewood-Paley decomposition and paraproducts} \label{Apx:LPProduct}
In this section we fix conventions and notation regarding Fourier analysis, Littlewood-Paley and paraproduct decompositions. 
See e.g. \cite{Bony81,BCD11}  for more details. 

For $f(z,v)$ in the Schwartz space, we define the Fourier transform $\hat{f}_k(\eta)$ where $(k,\eta) \in \Integer \times \Real$, 
\begin{align*} 
\hat{f}_k(\eta) = \frac{1}{2\pi}\int_{\Torus \times \Real} e^{-i z k - iv\eta} f(z,v) dz dv. 
\end{align*}
Similarly we have the Fourier inversion formula, 
\begin{align*} 
f(z,v) = \frac{1}{2\pi}\sum_{k \in \Integer} \int_{\Real} e^{i z k + iv\eta} \hat{f}_k(\eta) d\eta. 
\end{align*} 
As usual the Fourier transform and its inverse are extended to $L^2$ via duality. 
We also need to apply the Fourier transform to functions of $v$ alone, for which we use analogous conventions.
With these conventions note, 
\begin{align*} 
\int f(z,v) \overline{g}(z,v) dzdv & = \sum_{k}\int \hat{f}_k(\eta) \overline{g}_{k}(\eta) d\eta \\ 
\widehat{fg} & = \frac{1}{2\pi}\hat{f} \ast \hat{g} \\ 
(\widehat{\grad f})_k(\eta) & = (ik,i\eta)\hat{f}_k(\eta). 
\end{align*}

This work makes heavy use of the Littlewood-Paley dyadic decomposition.  
Here we fix conventions and review the basic properties of this classical theory, see e.g. \cite{BCD11} for more details. 
First we define the Littlewood-Paley decomposition only in the $v$ variable. 
Let $\psi \in C_0^\infty(\Real;\Real)$ be such that $\psi(\xi) = 1$ for $\abs{\xi} \leq 1/2$ and $\psi(\xi) = 0$ for $\abs{\xi} \geq 3/4$ and define $\rho(\xi) = \psi(\xi/2) - \psi(\xi)$, supported in the range $\xi \in (1/2,3/2)$. 
Then we have the partition of unity 
\begin{align*} 
1 = \psi(\xi) + \sum_{M \in 2^\Naturals} \rho(M^{-1}\xi), 
\end{align*}  
where we mean that the sum runs over the dyadic numbers $M = 1,2,4,8,...,2^{j},...$
 and we define the cut-off $\rho_M(\xi) = \rho(M^{-1}\xi)$, each supported in $M/2 \leq \abs{\xi} \leq 3M/2$. 
For $f \in L^2(\Real)$ we define
\begin{align*} 
f_{M} & = \rho_M(\abs{\partial_v})f, \\ 
f_{\frac{1}{2}} & =  \psi(\abs{\partial_v})f, \\ 
f_{< M} & = f_{\frac{1}{2}} + \sum_{K \in 2^{\Naturals}: K < M} f_K, 
\end{align*}
which defines the decomposition  
\begin{align*} 
f = f_{\frac{1}{2}} + \sum_{M \in 2^\Naturals} f_M.  
\end{align*}
Normally one would use $f_0$ rather than the slightly inconsistent $f_{\frac{1}{2}}$, however $f_0$ is reserved for the much more commonly used projection onto the zero mode only in $z$ (or $x$).  
Recall the definition of $\mathbb D$ from \S\ref{sec:Notation}. There holds the almost orthogonality and the approximate projection property 
\begin{subequations} \label{ineq:LPOrthoProject}
\begin{align} 
\norm{f}^2_2 & \approx \sum_{M \in \mathbb D} \norm{f_M}_2^2 \\
 \norm{f_M}_2 & \approx  \norm{(f_{M})_M}_2. 
\end{align}
\end{subequations}
The following is also clear:
\begin{align*} 
\norm{\abs{\partial_v}f_M}_2 \approx M \norm{f_M}_2. 
\end{align*}
We make use of the notation 
\begin{align*} 
f_{\sim M} = \sum_{K \in \mathbb{D}: \frac{1}{C}M \leq K \leq CM} f_{K}, 
\end{align*}
for some constant $C$ which is independent of $M$.
Generally the exact value of $C$ which is being used is not important; what is important is that it is finite and independent of $M$. 
Similar to \eqref{ineq:LPOrthoProject} but more generally, if $f = \sum_{k} D_k$ for any $D_k$ with $\frac{1}{C}2^{k} \subset \textup{supp}\, D_k \subset C2^{k}$ it follows that 
\begin{align} 
\norm{f}^2_2 \approx_C \sum_{k} \norm{D_k}_2^2. \label{ineq:GeneralOrtho}
\end{align}
During much of the proof we are also working with Littlewood-Paley decompositions defined in the $(z,v)$ variables, 
with the notation conventions being analogous. 
Our convention is to use $N$ to denote Littlewood-Paley projections in $(z,v)$ and $M$ to denote projections only in the $v$ direction. 

We have opted to use the compact notation above, rather than the commonly used alternatives
\begin{align*} 
\Delta_{j}f  = f_{2^j}, \quad\quad S_jf  = f_{<2^j},
\end{align*}
in order to reduce the number of characters in long formulas. 
The last unusual notation we use is 
\begin{align*} 
P_{\neq 0}f = f - <f>, 
\end{align*}
which denotes projection onto the non-zero modes in $z$.

Another key Fourier analysis tool employed in this work is the paraproduct decomposition, introduced by Bony \cite{Bony81} (see also \cite{BCD11}). 
Given suitable functions $f,g$ we may define the \emph{inhomogeneous} paraproduct decomposition (in either $(z,v)$ or just $v$), 
\begin{align} 
fg & = T_fg + T_gf + \mathcal{R}(f,g) \label{def:lameppd}\\ 
& = \sum_{N \geq 8} f_{<N/8}g_N + \sum_{N \geq 8} g_{<N/8}f_N + \sum_{N \in \mathbb{D}}\sum_{N/8 \leq N^\prime \leq 8N} g_{N^\prime}f_{N}, \nonumber
\end{align}
where all the sums are understood to run over $\mathbb D$, or the \emph{homogeneous} paraproduct
\begin{align*} 
fg & =  \sum_{N \in 2^\Integer} f_{<N/8}g_N + \sum_{N \in 2^\Integer} g_{<N/8}f_N + \sum_{N \in 2^\Integer}\sum_{N/8 \leq N^\prime \leq 8N} g_{N^\prime}f_{N}.  
\end{align*}
The choice of which one we employ depends on the role that low frequencies are playing in the proof. 
In our work we do not employ the notation in \eqref{def:lameppd} since at most steps we are working in non-standard regularity spaces and are usually applying multipliers which do not satisfy any version of $AT_f g \approx T_f Ag$. 
Hence, we normally have to prove most everything by hand and only rely on standard para-differential calculus as a guide.

\subsection{Elementary inequalities and Gevrey spaces} \label{apx:Gev}

In the sequel we show some basic inequalities which are extremely useful for working in this scale of spaces.  
The first three are versions of Young's inequality (applied on the frequency-side here).
\begin{lemma}
Let $f(\xi),g(\xi) \in L_\xi^2(\Real^d)$, $\jap{\xi}^\sigma h(\xi) \in L_\xi^2(\Real^d)$ and $\jap{\xi}^\sigma b(\xi) \in L_\xi^2(\Real^d)$  for $\sigma > d/2$, 
Then we have 
\begin{align} 
\norm{f \ast h}_2 & \lesssim_{\sigma, d} \norm{f}_2\norm{\jap{\cdot}^\sigma h}_2, \label{ineq:L2L1}  \\
\int \abs{f(\xi) (g \ast h)(\xi)} d\xi & \lesssim_{\sigma,d} \norm{f}_2\norm{g}_2\norm{\jap{\cdot}^\sigma h}_2 \label{ineq:L2L2L1} \\ 
\int \abs{f(\xi) (g \ast h \ast b) (\xi)} d\xi & \lesssim_{\sigma,d} \norm{f}_2\norm{g}_2\norm{\jap{\cdot}^\sigma h}_2\norm{\jap{\cdot}^\sigma b}_2. \label{ineq:L2L2L1L1}  
\end{align}
\end{lemma}

The next set of inequalities show that one can often gain on the index of regularity when comparing frequencies which are not too far apart (provided $0 < s < 1$).
They are crucial for doing paradifferential calculus in Gevrey regularity. 

\begin{lemma}
Let $0 < s < 1$ and $x \geq y \geq 0$ (without loss of generality).  
\begin{itemize} 
\item[(i)] If $x + y > 0$,
\begin{align} 
\abs{x^s - y^s} \lesssim_s \frac{1}{x^{1-s} + y^{1-s}}\abs{x-y}. \label{ineq:TrivDiff}
\end{align}
\item[(ii)] If $\abs{x-y} \leq x/K$ for some $K > 1$ then 
\begin{align} 
\abs{x^s - y^s} \leq \frac{s}{(K-1)^{1-s}}\abs{x-y}^s. \label{lem:scon}
\end{align} 
Note $\frac{s}{(K-1)^{1-s}} < 1$ as soon as $s^{\frac{1}{1-s}} + 1 < K$. 
\item[(iii)] In general, 
\begin{align} 
\abs{x + y}^s \leq \left(\frac{x}{x+y}\right)^{1-s}\left(x^s + y^s\right). \label{lem:smoretrivial}
\end{align}  
In particular, if $\frac{1}{K}y \leq x \leq Ky$ for some $K < \infty$ then 
\begin{align} 
\abs{x + y}^s \leq \left(\frac{K}{1 + K}\right)^{1-s}\left(x^s + y^s\right). \label{lem:strivial}
\end{align} 
\end{itemize}
\end{lemma}

Using \eqref{ineq:L2L2L1}, \eqref{lem:scon} and \eqref{lem:strivial} together with a paraproduct expansion, the following product lemma is relatively straightforward. 
For contrast, the lemma holds when $s = 1$ only for $c = 1$. 
\begin{lemma}[Product lemma] \label{lem:GevProdAlg}
For all $0<s<1$, $\sigma \geq 0$ and $\sigma_0 > 1$, there exists $c = c(s,\sigma,\sigma_0) \in (0,1)$ such that the following holds for all $f,g \in \mathcal{G}^{\lambda,\sigma;s}$:
\begin{subequations}
\begin{align} 
\norm{fg}_{\G^{\lambda,\sigma}} & \lesssim \norm{f}_{\G^{c\lambda,\sigma_0}} \norm{g}_{\G^{\lambda,\sigma}} + \norm{g}_{\G^{c\lambda,\sigma_0}} \norm{f}_{\lambda,\sigma}, \label{ineq:GProduct}
\end{align}
\end{subequations}
in particular, $\mathcal{G}^{\lambda,\sigma;s}$ has the algebra property:
\begin{align} 
\norm{f g}_{\G^{\lambda,\sigma}} & \lesssim \norm{f}_{\G^{\lambda,\sigma}} \norm{g}_{\G^{\lambda,\sigma}} \label{ineq:GAlg} 
\end{align} 
\end{lemma} 

Gevrey and Sobolev regularities can be related with the following two inequalities. 
\begin{itemize}
\item[(i)] For all $x \geq 0$, $\alpha > \beta \geq 0$, $C,\delta > 0$, 
\begin{align} 
e^{Cx^{\beta}} \leq e^{C\left(\frac{C}{\delta}\right)^{\frac{\beta}{\alpha - \beta}}} e^{\delta x^{\alpha}};  \label{ineq:IncExp}
\end{align}
\item[(ii)] For all $x \geq 0$, $\alpha,\sigma,\delta > 0$, 
\begin{align} 
e^{-\delta x^{\alpha}} \lesssim \frac{1}{\delta^{\frac{\sigma}{\alpha}} \jap{x}^{\sigma}}. \label{ineq:SobExp}
\end{align}
\end{itemize}
Together these inequalities show that for $\alpha > \beta \geq 0$, $\norm{f}_{\mathcal{G}^{\beta;C,\sigma}} \lesssim_{\alpha,\beta,C,\delta,\sigma} \norm{f}_{\mathcal{G}^{\alpha;\delta,0}}$. 

\bibliographystyle{plain} 
\bibliography{eulereqns_vlad}

\end{document}